\documentclass[11pt]{amsart}
\usepackage{amsxtra}
\usepackage[all]{xy}
\usepackage{amssymb}
\usepackage{amsmath}
\usepackage{tikz-cd}
\usepackage{array}
\usepackage{booktabs}
\usepackage{ifpdf}
\usepackage{enumerate}
\usepackage{tikz}
\usepackage[active]{srcltx}
\usepackage{cases}
\usepackage{appendix}

\theoremstyle{plain}
\newcommand{\cleqn}{\setcounter{equation}{0}}
\newcommand{\clth}{\setcounter{theorem}{0}}
\newcommand {\sectionnew}[1]{\section{#1}\cleqn\clth}
\newtheorem{theorem}{Theorem}[section]
\newtheorem{lemma}[theorem]{Lemma}
\newtheorem{definition-theorem}[theorem]{Definition-Theorem}
\newtheorem{proposition}[theorem]{Proposition}
\newtheorem{corollary}[theorem]{Corollary}
\newtheorem{definition}[theorem]{Definition}
\newtheorem{example}[theorem]{Example}
\newtheorem{remark}[theorem]{Remark}
\newtheorem{notation}[theorem]{Notation}
\newtheorem{assumption}[theorem]{Assumption}
\newtheorem{lemma-definition}[theorem]{Lemma-Definition}
\newtheorem{lemma-notation}[theorem]{Lemma-Notation}
\newtheorem{question}[theorem]{Question}
\newtheorem{remark-definition}[theorem]{Remark-Definition}
\newcommand \bth[1] { \begin{theorem}\label{t#1} }
\newcommand \ble[1] { \begin{lemma}\label{l#1} }

\newcommand \bpr[1] { \begin{proposition}\label{p#1} }
\newcommand \bco[1] { \begin{corollary}\label{c#1} }
\newcommand \bde[1] { \begin{definition}\label{d#1}\rm }
\newcommand \bex[1] { \begin{example}\label{e#1}\rm }
\newcommand \bre[1] { \begin{remark}\label{r#1}\rm }

\newcommand \bnota[1] {\begin{notation}\label{n#1}\rm }
\newcommand \bas[1] { \begin{assumption}\label{a#1}\rm }

\newcommand \bqu[1] { \begin{question}\label{q#1}\rm }

\newcommand {\ele} { \end{lemma} }

\newcommand {\epr} { \end{proposition} }
\newcommand {\eco} { \end{corollary} }
\newcommand {\ede} { \end{definition} }
\newcommand {\eex} { \end{example} }
\newcommand {\ere} { \end{remark} }
\newcommand {\enota} { \end{notation} }
\newcommand {\eas} {\end{assumption}}

\newcommand {\equ} {\end{question}}
\newcommand \lb[1]{\label{#1}}

\newcommand{\beqa}{\begin{eqnarray*}}                     %added by Lu
\newcommand{\eeqa}{\end{eqnarray*}}

\def \sC {{\scriptscriptstyle C}}

\def \wF_mn {\wF_m \times \wF_n}
\def \wF_mnC {\wF_{m, n, \, \sC}}

\def \wF {\widetilde{F}}

\usepackage{lipsum}

\newcommand\blfootnote[1]{%
  \begingroup
  \renewcommand\thefootnote{}\footnote{#1}%
  \addtocounter{footnote}{-1}%
  \endgroup
}
\begin{document}

\setlength{\baselineskip}{1.2\baselineskip}
%%%%%%%%%%%%%%%%%%%%%%%%%%%%%%%%%%%%%%%%%%%%%%%%%%%%%%%%%%%%%%%%%%%%%%%%%%%
%%%%%%%%%%%%%%%%%%%%%%    Title    %%%%%%%%%%%%%%%%%%%%%%%%%%%%%%%%%%%%%%%%
\title[Classification of Transposed Poisson 3-Lie algebras of dimension 3]
{Classification of Transposed Poisson 3-Lie algebras of dimension 3}
\author{Yaxi Jiang}
\address{
School of Mathematical Sciences    \\
Zhejiang Normal University\\
Jinhua 321004              \\
China}
\email{jiangyx@zjnu.edu.cn}

\author{Chuangchuang Kang}
\address{
School of Mathematical Sciences    \\
Zhejiang Normal University\\
Jinhua 321004              \\
China}
\email{kangcc@zjnu.edu.cn}

\author{Jiafeng L\"u}
\address{
School of Mathematical Sciences    \\
Zhejiang Normal University\\
Jinhua 321004              \\
China}
\email{jiafenglv@zjnu.edu.cn}

\blfootnote{*Corresponding Author: Chuangchuang Kang. Email: kangcc@zjnu.edu.cn.}

\date{}
\begin{abstract}
Transposed Poisson $3$-Lie algebra is a dual notion of Nambu-Poisson algebra of order 3. In this paper, we explicitly determine all $\frac{1}{3}$-derivations and automorphisms of the unique nontrivial $3$-dimensional complex $3$-Lie algebra $(A_3,[\cdot,\cdot,\cdot])$. Based on the one-one correspondence between $\frac{1}{3}$-derivations and transposed Poisson 3-Lie algebras, up to isomorphism, we classify transposed Poisson $3$-Lie algebras of dimension $3$ under the case that $L_{e_1}$ is  trivial over the complex field $\mathbb{C}$.
\end{abstract}


\maketitle
%\tableofcontents
%%%%%%%%%%%%%%%%%%%%   Introduction   %%%%%%%%%%%%%%%%%%%%%%%%%%%%%%%%%%%%%%%%
\allowdisplaybreaks
\sectionnew{Introduction }\lb{intro}

In \rm{\cite{Filippov}}, Filippov introduced the concept of $n$-Lie algebras, which is different from Lie algebras \rm{\cite{Kasymov}} and plays an important role in mathematical fields and string theory \rm{\cite{$n$-ary,Bagger,Bagger2}. An $n$-Lie algebra is a vector space $A$ with a linear multiplication (we call it the $n$-Lie bracket) $[\cdot,\cdots,\cdot]: \wedge^n A\rightarrow A$ satisfying the Filippov-Jacobi identity:
\begin{equation*}
  [x_1,\cdots,x_{n-1},[y_1,\cdots,y_n]]=\sum\limits_{i=1}^n[y_1,\cdots,y_{i-1},[x_1,\cdots,x_{n-1},y_i],y_{i+1},\cdots,y_n],
\end{equation*}
where for all $x_1,\cdots,x_{n-1},y_1,\cdots,y_{n} \in A$.
A Poisson $n$-Lie algebra \cite{Cantarini} is a commutative associative algebra $(A,\cdot)$ equipped with an $n$-Lie bracket, satisfying the generalized Leibniz rule:
\begin{equation*}
[x_1,\cdots,x_i \cdot \tilde{x},\cdots,x_n]= x_i \cdot [x_1,\cdots,\tilde{x},\cdots,x_n] + \tilde{x} \cdot [x_1,\cdots,x_i ,\cdots,x_n],
\end{equation*}
where for all $x_1,\cdots,x_n,\tilde{x} \in A$. Poisson $n$-Lie algebras (also called Nambu-Poisson algebras) appeared from the study of Nambu mechanics \cite{Nambu} and provided an useful method to describe algebraical and geometrical structures of Nambu-Poisson manifolds \cite{Takhtajan}. Cantarini and Kac classified all complex linearly compact Poisson $n$-Lie algebras in \rm{\cite{Cantarini}}.

In recent years, much work has been done in transposed Poisson algebras (\rm{\cite{BaiC,Beites,Ouaridi,Kaygorodov-1,Kaygorodov-2,Kaygorodov-3,Kaygorodov-4,Ma,Yuan}}). The transposed Poisson algebra can be generalized to the transposed Poisson $n$-Lie algebra, which consists of a commutative algebra $(A,\cdot)$ and an $n$-Lie bracket with $n \geq 2$ satisfying:
\begin{equation*}
n \tilde{x} \cdot [x_1,\cdots,x_i,\cdots,x_n]=\sum_{i=1}^n [x_1,\cdots, \tilde{x} \cdot x_i,\cdots,x_n],
\end{equation*}
where for all $x_1,\cdots,x_n,\tilde{x} \in A$. The difference between transposed Poisson $n$-Lie algebras and Poisson $n$-Lie algebras lies in the different positions of their multiplications and $n$-Lie bracket. In the special case $n=3$, transposed Poisson $3$-Lie algebras can be deduced not only from $3$-Lie algebras and Poisson algebras, but also from $3$-Lie algebras and transposed Poisson algebras \cite{BaiC}.

Beites, Ouaridi and Kaygorodov separated all non-isomorphic cases of $3$-dimensional complex Lie algebras and $3$-dimensional complex transposed Poisson algebras in \rm{\cite{Beites}}.
In \cite{Beites-2}, Beites, Ferreira and Kaygorodov proposed the following question:
\begin{question} {\rm(\cite{Beites-2}, Question 38)}
How to give examples and classify simple transposed Poisson $n$-Lie algebras?
\end{question}
\noindent Ferreira, Kaygorodov and Lopatkin showed that there are no non-trivial transposed Poisson $n$-Lie algebra structures defined on a complex semisimple finite-dimensional $n$-Lie algebra \cite{Ferreira}. For the non-semisimple finite-dimensional $n$-Lie algebras, transposed Poisson $n$-Lie algebras can be constructed by $\frac{1}{n}$-derivations of $n$-Lie algebras. $\frac{1}{n}$-derivations are the special cases of $\delta$-derivations \cite{Kaygorodov-5}, which are specific types of generalized derivations \rm{\cite{Beites1,Leger}} introduced by Filippov \rm{\cite{Filippov-d1,Filippov-d2,Filippov-d3}}. In \rm{\cite{Filippov}}, Filippov classified the $(n+1)$-dimensional $n$-Lie algebras over an algebraically closed field of characteristic zero. In \rm{\cite{BaiR}}, Bai, Song and Zhang gave the complete classification of $(n+2)$-dimensional $n$-Lie algebras by isomorphism. However, it is complicated to classify $n$-dimensional $(n\geq 3)$ transposed Poisson $n$-Lie algebras. Our aim is to separate all non-isomorphic cases of $3$-dimensional transposed Poisson $3$-Lie algebras under some certain conditions.
%The above question motivates us to classify transposed Poisson $3$-Lie algebras.
%n-Lie alg的分类
%In \rm{\cite{Filippov}}, Filippov classified the $(n+1)$-dimensional $n$-Lie algebras over an algebraically closed field of characteristic zero. In \rm{\cite{BaiR}}, Bai, Song and Zhang gave the complete classification of $(n+2)$-dimensional $n$-Lie algebras by isomorphism.
%Nambu \rm{\cite{Nambu}} proposed the generalization of the classical Hamiltonian dynamics to a $3$-dimensional phase space, firstly extended the Poisson bracket to the $3$-ary multilinear operation, i.e., $n=3$ case.

%TPA的分类
%Their classification strategies are the following two steps. $(1)$ Find all automorphisms of Lie algebra $(A,[\cdot,\cdot])$. $(2)$ Consider the multiplication table of commutative associative algebra $(A,\cdot)$ under the action of elements from automorphisms of Lie algebra.
There exists a unique nontrivial $3$-dimensional complex $3$-Lie algebra $(A_3,[\cdot,\cdot,\cdot])$ (see Proposition \ref{3-dim}).
%We can describe all 3-dimensional transposed Poisson $3$-Lie algebras by the following way: obtain the classification of all $\frac{1}{3}$-derivations of 3-dimensional $3$-Lie algebras and construct all possible transposed Poisson structures on these algebras. Thus, $\frac{1}{3}$-derivation of $3$-Lie algebra, which is associated with automorphism of $3$-Lie algebra, plays an extremely important role in the classification process.
In order to classify the non-isomorphic $3$-dimensional transposed Poisson $3$-Lie algebras, we define the automorphisms of transposed Poisson $3$-Lie algebras that are both the automorphisms of commutative algebras and the automorphisms of $3$-Lie algebras. Our classification strategies are the following three steps. $(1)$ Compute all the $\frac{1}{3}$-derivations and automorphisms on $(A_3,[\cdot,\cdot,\cdot])$. $(2)$ Give the multiplications of $3$-dimensional transposed Poisson $3$-Lie algebra derived by $\frac{1}{3}$-derivations of $(A_3,[\cdot,\cdot,\cdot])$. $(3)$ By using the automorphisms of $(A_3,[\cdot,\cdot,\cdot])$ to act on the commutative algebra, we classify the non-isomorphic $3$-dimensional transposed Poisson $3$-Lie algebra under certain conditions.

%More details are as follows.
%Firstly, we fix the non-zero multiplication table on the basis of $3$-dimensional $3$-Lie algebra $A_3$ in Proposition \ref{3-dim}. Secondly, we compute all the $\frac{1}{3}$-derivations on the class of $A_3$ in Proposition \ref{vd-beta}. Thirdly, by the results we got in the step two and Proposition \ref{comm}, the commutative multiplication results of transposed Poisson $3$-Lie algebra on the basis can be obtained in Proposition \ref{com-multi}. Fourthly, under the action of an automorphism of $A_3$, given in Proposition \ref{phi-lamdae}, we rewrite the multiplication table of $(A_3,\cdot)$. Then, we separate the multiplications into different isomorphic classes under some certain conditions. Finally, we get all the classification results in Theorem \ref{classification}. Furthermore, we give the conditions when the commutative multiplication of transposed Poisson $3$-Lie algebra $(A_3,\cdot,[\cdot,\cdot,\cdot])$ is associative in Remark \ref{asso}. See Section 4 for details of the proof process. In order to classify $n$-dimensional transposed Poisson $3$-Lie algebra, we also give the necessary and sufficient conditions about $\frac{1}{3}$-derivation and automorphism on the $n$-dimensional transposed Poisson $3$-Lie algebra in Proposition \ref{n-13der} and \ref{n-auto}.

The paper is organized as follows. In Section 2, we introduce some preliminaries. In Section 3, we classify complex transposed Poisson $3$-Lie algebras under the case that $L_{e_1}$ is  trivial. First, we provide the necessary and sufficient conditions of $\frac{1}{3}$-derivations and automorphisms on $n$-dimensional $3$-Lie algebra (Propositions \ref{n-13der} and \ref{n-auto}). Second, we give all the $\frac{1}{3}$-derivations and automorphisms of $(A_3,[\cdot,\cdot,\cdot])$ (Propositions \ref{vd-beta} and \ref{phi-lamdae}). Third, we describe the multiplications of $3$-dimensional transposed Poisson $3$-Lie algebra derived by $\frac{1}{3}$-derivations of $(A_3,[\cdot,\cdot,\cdot])$ (Proposition \ref{com-multi}).
%Under the action of an automorphism of $A_3$, given in Proposition \ref{phi-lamdae}, we rewrite the multiplication table of $(A_3,\cdot)$.
Finally, we get the explicit classification of $3$-dimensional complex transposed Poisson $3$-Lie algebra under the case that $L_{e_1}$ is  trivial (Theorem \ref{classification}). In Section 4, we give the proof of Theorem \ref{classification}.

Throughout this paper, unless otherwise specified, all the vector spaces and algebras are finite-dimensional over an algebraically closed field $\mathbb{F}$ of characteristic zero, $\mathbb{C}$ is the complex field.

\section{Preliminaries}

%\begin{definition}
%Let $L$ be a vector space equipped with two nonzero bilinear operations $-\cdot-$ and $[\cdot,\cdot]$. The triple $(L,\cdot,[\cdot,\cdot])$ is called a \textbf{transposed Poisson algebra} if $(L,\cdot)$ is a commutative associative algebra and $(L,[\cdot,\cdot])$ is a Lie algebra that satisfies the following compatibility condition
%\begin{equation}\label{eq:tp-alg}
%  2z\cdot[x,y]=[z\cdot x,y]+[x,z\cdot y], \quad \forall~ x,y,z\in L.
%\end{equation}
%\end{definition}

%\begin{definition}
%%tp alg 的导子,既是交换代数的又是Lie代数的导子
%
%
%
%
%\end{definition}

In this section, we give some preliminaries and basic results on $3$-Lie algebras, transposed Poisson $3$-Lie algebras, and $\frac{1}{3}$-derivations.

\begin{definition}\rm{(\cite{Filippov})}
A \textbf{$3$-Lie algebra} is a vector space $A$ with a 3-ary multi-linear map ($3$-Lie bracket) $[\cdot,\cdot,\cdot]: \wedge^3 A \rightarrow A$ such that the following \textbf{Fundamental Identity} holds:
\begin{equation*}
[[x,y,z],u,v]=[[x,u,v],y,z]+[[y,u,v],z,x]+[[z,u,v],x,y], \quad \forall~ x,y,z,u,v\in A.
\end{equation*}
\end{definition}

\begin{definition} \rm{(\cite{BaiC})}
Let $(A,\cdot)$ be a commutative algebra and $(A,[\cdot,\cdot,\cdot])$ be a $3$-Lie algebra.
The triple $(A,\cdot,[\cdot,\cdot,\cdot])$ is called a \textbf{transposed Poisson $3$-Lie algebra} if it satisfies the following condition:
\begin{equation} \label{eq:TP3}
3u \cdot [x,y,z]=[u\cdot x,y,z]+[x,u\cdot y,z]+[x,y,u\cdot z], \quad \forall~ x,y,z,u\in A.
\end{equation}
\end{definition}

%\begin{pro}
%pro4.7
%
%
%
%
%
%\end{pro}

\begin{definition} \rm{(\cite{Ferreira})} \label{3-deriv}
Let $(A,[\cdot,\cdot,\cdot])$ be a ternary algebra with the multiplication $[\cdot,\cdot,\cdot]$ and $\varphi$ be a linear map. Then $\varphi$ is a \textbf{$\frac{1}{3}$-derivation} if it satisfies
\begin{equation}\label{eq:var-d}
  \varphi[x,y,z]=\frac{1}{3}([\varphi(x),y,z]+[x,\varphi(y),z]+[x,y,\varphi(z)]), \quad \forall~ x,y,z\in A.
\end{equation}
\end{definition}

\begin{proposition} \rm{(\cite{Ferreira})} \label{comm}
Let $(A,\cdot)$ be a commutative algebra, $(A, [\cdot,\cdot,\cdot])$ be a  $3$-Lie algebra, and for all $x,y\in A$, $L_x: A \rightarrow A$ given by $L_x(y) = x \cdot y$ be a left multiplication of $(A,\cdot)$. Then $L_x$ is a $\frac{1}{3}$-derivation of $(A, [\cdot,\cdot,\cdot])$ if and only if $(A,\cdot, [\cdot,\cdot,\cdot])$ is a transposed Poisson $3$-Lie algebra.
\end{proposition}

\begin{definition} \label{isomorphism}
Let $(A_1,\cdot_1,[\cdot,\cdot,\cdot]_1)$ and $(A_2,\cdot_2,[\cdot,\cdot,\cdot]_2)$ be two transposed Poisson $3$-Lie algebras. A \textbf{homomorphism} of transposed Poisson $3$-Lie algebra is a linear map $\varphi:A_1 \rightarrow A_2$ such that
\begin{equation} \label{eq:f-iso}
  \varphi([x,y,z]_1) = [\varphi(x),\varphi(y),\varphi(z)]_2,
\end{equation}
\begin{equation} \label{eq:y-iso}
  \varphi(x \cdot_1 y) =\varphi(x) \cdot_2 \varphi(y), \quad \forall~ x,y,z\in A_1.
\end{equation}
An \textbf{isomorphism} of transposed Poisson $3$-Lie algebras is an invertible homomorphism. Two transposed Poisson $3$-Lie algebras are said to be \textbf{isomorphic} if there exists an isomorphism between them. In particular, if $\varphi$ is an isomorphism and $(A_1,\cdot_1,[\cdot,\cdot,\cdot]_1)=(A_2,\cdot_2,[\cdot,\cdot,\cdot]_2)$, then $\varphi$ is called an \textbf{automorphism} of transposed Poisson $3$-Lie algebra. If $(A_1,\cdot_1,[\cdot,\cdot,\cdot]_1)$ and $(A_2,\cdot_2,[\cdot,\cdot,\cdot]_2)$ with trivial commutative algebra structure, then $\varphi$ is an isomorphism of $3$-Lie algebra, in particular, if $(A_1,[\cdot,\cdot,\cdot]_1)=(A_2,[\cdot,\cdot,\cdot]_2)$, then $\varphi$ is called an automorphism of $3$-Lie algebra.
\end{definition}

\begin{proposition}\rm{(\cite{BaiR})} \label{3-dim}
There exists a unique nontrivial $3$-dimensional complex $3$-Lie algebra. It has a basis $\left\{e_{1}, e_{2}, e_{3}\right\}$ with respect to which the nonzero product is given by
\begin{equation}\label{eq:basis}
  [e_1,e_2,e_3]=e_1.
\end{equation}
We denote this $3$-Lie algebra by $(A_3,[\cdot,\cdot,\cdot])$.
\end{proposition}

\section{Classification of $3$-dimensional transposed Poisson $3$-Lie algebras}

In this section, by the $\frac{1}{3}$-derivation and automorphism of $(A_3,[\cdot,\cdot,\cdot])$, we give the classification of $3$-dimensional transposed Poisson $3$-Lie algebra in the isomorphic sense under the case $L_{e_1}$ is  trivial.
%by the structure constants of $n$-dimensional $3$-Lie algebras, we obtain the sufficient and necessary conditions of $\frac{1}{3}$-derivation (Proposition \ref{n-13der}) and automorphism (Proposition \ref{n-auto}). Then, through the classification result of $3$-dimensional complex $3$-Lie algebra $A_3$ given in Proposition \ref{3-dim}, we obtain the $\frac{1}{3}$-derivation (Proposition \ref{vd-beta}) and automorphism (Proposition \ref{phi-lamdae}) on $A_3$. Besides, according to the relation between the $\frac{1}{3}$-derivation of $A_3$ and the left multiplication of $(A_3,\cdot)$ of $3$-dimensional transposed Poisson $3$-Lie algebra, we have the commutative multiplication table on $(A_3,\cdot,[\cdot,\cdot,\cdot])$ (Lemma \ref{com-multi}). Last but not least, Theorem \ref{classification} gives the classification results of commutative $(A_3,\cdot,[\cdot,\cdot,\cdot])$ in the isomorphic sense under some certain conditions.

\begin{proposition} \label{n-13der}
Let $(A,[\cdot,\cdot,\cdot])$ be an $n$-dimensional $3$-Lie algebra with a basis $\left\{e_{1}, e_{2}, \cdots, e_{n}\right\}$, $\varphi:A \rightarrow A$ be a linear map. Suppose
\begin{equation} \label{eq:eiejek}
  [e_i,e_j,e_k]=\sum_{s=1}^{n} C_{ijk}^s e_s,\quad \forall~ C_{ijk}^s\in \mathbb{C},~ 1 \leq i,j,k \leq n,
\end{equation}
\begin{equation} \label{eq:varei}
\varphi(e_i)=\sum_{j=1}^{n} \beta_{ij} e_j,\quad \forall~ \beta_{ij}\in \mathbb{C},~ 1 \leq i \leq n.
\end{equation}
Then $\varphi$ is a $\frac{1}{3}$-derivation if and only if the structure constants $C_{ijk}^s$ and $\beta_{ij}$ satisfy
\begin{equation} \label{eq:Cbeta}
\sum_{s=1}^n (C_{sjk}^t \beta_{is} + C_{isk}^t \beta_{js} + C_{ijs}^t \beta_{ks} -3C_{ijk}^s \beta_{st}) = 0,\quad 1 \leq t \leq n.
\end{equation}
\end{proposition}
\begin{proof}
By \eqref{eq:eiejek} and \eqref{eq:varei}, for all $ 1 \leq i,j,k \leq n$, we have
\begin{eqnarray*}
  &&\frac{1}{3}([\varphi(e_i),e_j,e_k]+[e_i,\varphi(e_j),e_k]+[e_i,e_j,\varphi(e_k)])
  -\varphi([e_i,e_j,e_k])   \\
  &=& \frac{1}{3}([\sum_{p=1}^n \beta_{ip} e_p , e_j, e_k]+[e_i, \sum_{p=1}^n \beta_{jp} e_p,e_k]+[e_i,e_j,\sum_{p=1}^n \beta_{kp} e_p])-\varphi([e_i,e_j,e_k]) \\
  &=& \frac{1}{3}(\sum_{s=1}^n \beta_{is} \sum_{t=1}^n C_{sjk}^t e_t +\sum_{s=1}^n
  \beta_{js} \sum_{t=1}^n C_{isk}^t e_t + \sum_{s=1}^n \beta_{ks} \sum_{t=1}^n C_{ijs}^t e_t) - \varphi(\sum_{s=1}^n C_{ijk}^s e_s) \\
  &=& \sum_{t=1}^n (\frac{1}{3}\sum_{s=1}^n (\beta_{is} C_{sjk}^t+\beta_{js} C_{isk}^t +\beta_{ks} C_{ijs}^t) - \sum_{s=1}^n C_{ijk}^s \beta_{st} ) e_t.
\end{eqnarray*}
Then $\varphi$ is a $\frac{1}{3}$-derivation if and only if \eqref{eq:Cbeta} holds.
\end{proof}

\begin{proposition} \label{n-auto}
Let $(A,[\cdot,\cdot,\cdot])$ be an $n$-dimensional $3$-Lie algebra with a basis $\left\{e_{1}, e_{2}, \cdots, e_{n}\right\}$. If the $3$-Lie bracket $[\cdot,\cdot,\cdot]$ is given by \eqref{eq:eiejek} and a linear map $\phi:A \rightarrow A$ satisfies
\begin{equation}\label{eq:phi-ei}
  \phi(e_i)=\sum_{j=1}^{n} \lambda_{ij} e_j,\quad \forall~ \lambda_{ij}\in \mathbb{C}, ~ 1 \leq i \leq n,
\end{equation}
then $\phi$ is an automorphism of $(A,[\cdot,\cdot,\cdot])$ if and only if the structure constants $C_{ijk}^s$ and $\lambda_{ij}$, $ 1 \leq i,j,k,s \leq n $, satisfy
\begin{equation} \label{eq:Clam}
  \sum_{u,v,w} \lambda_{iu} \lambda_{jv} \lambda_{kw} C_{uvw}^t - \sum_{s=1}^n C_{ijk}^s \lambda_{st} = 0, \quad 1 \leq t \leq n.
\end{equation}
\end{proposition}
\begin{proof}
By \eqref{eq:eiejek} and \eqref{eq:phi-ei}, for all $ 1 \leq i,j,k \leq n$, we have
\begin{eqnarray*}
  &&[\phi(e_i),\phi(e_j),\phi(e_k)]-\phi([e_i,e_j,e_k])  \\
  &=& [\sum_{u=1}^n \lambda_{iu} e_u, \sum_{v=1}^n \lambda_{jv} e_v, \sum_{w=1}^n \lambda_{kw} e_w] - \phi(\sum_{s=1}^n C_{ijk}^s e_s)   \\
  &=& \sum_{u,v,w} (\lambda_{iu} \lambda_{jv} \lambda_{kw})(\sum_{t=1}^n C_{uvw}^t e_t)-\sum_{s=1}^n C_{ijk}^s \sum_{t=1}^n \lambda_{st} e_t  \\
  &=& \sum_{t=1}^n (\sum_{u,v,w} \lambda_{iu} \lambda_{jv} \lambda_{kw} C_{uvw}^t - \sum_{s=1}^n C_{ijk}^s \lambda_{st}) e_t.
\end{eqnarray*}
Then $\phi$ is an automorphism of $(A,[\cdot,\cdot,\cdot])$ if and only if \eqref{eq:Clam} holds.
\end{proof}

\begin{proposition} \label{vd-beta}
Let $ \varphi:A_3\rightarrow A_3 $ be a $\frac{1}{3}$-derivation on the $3$-dimensional $3$-Lie algebra $(A_3,[\cdot,\cdot,\cdot])$, which is
given by
\begin{equation} \label{eq:varei-3}
\varphi(e_i)=\sum_{j=1}^{3} \beta_{ij} e_j,\quad \forall~ \beta_{ij}\in \mathbb{C},~ 1 \leq i \leq 3.
\end{equation}
Then
\begin{equation}  \label{eq:var-3deri}
  \left\{\begin{split}
  \varphi(e_1) &= \frac{\beta_{22}+\beta_{33}}{2}e_1, \\
  \varphi(e_2) &= \beta_{21} e_1 + \beta_{22} e_2 + \beta_{23} e_3,\\
  \varphi(e_3) &= \beta_{31} e_1 + \beta_{32} e_2 + \beta_{33} e_3.
  \end{split}\right.
\end{equation}
\end{proposition}
\begin{proof}
By Proposition \ref{3-dim}, we have $C_{123}^1=1$, the rest are zero. Then by \eqref{eq:Cbeta}, we have
\begin{equation*}
  \left\{\begin{split}
  2 \beta_{11} &= \beta_{22} + \beta_{33}, \\
  \beta_{12} &= \beta_{13} = 0.
  \end{split}\right.
\end{equation*}
By solving the above equations, we obtain the $\frac{1}{3}$-derivations of  $(A_3,[\cdot,\cdot,\cdot])$.
\end{proof}

\begin{proposition}  \label{phi-lamdae}
Let $\phi:A_3\rightarrow A_3$ be an automorphism of the $3$-dimensional $3$-Lie algebra $(A_3,[\cdot,\cdot,\cdot])$, which is given by
\begin{equation}\label{eq:phi-ei-3}
  \phi(e_i)=\sum_{j=1}^{3} \lambda_{ij} e_j,\quad \forall~ \lambda_{ij}\in \mathbb{C}, ~ 1 \leq i \leq 3.
\end{equation}
Then
\begin{equation}\label{eq:phi-lamdae}
\phi(e_1, e_2, e_3)=
\left(
  \begin{array}{ccc}
    \lambda_{11} & 0 & 0 \\
    \lambda_{21} & \lambda_{22} & \lambda_{23} \\
    \lambda_{31} & \lambda_{32} & \lambda_{33} \\
  \end{array}
\right)
\left(
  \begin{array}{c}
    e_1 \\
    e_2 \\
    e_3 \\
  \end{array}
\right),
\end{equation}
where $ \lambda_{22} \lambda_{33} -\lambda_{23} \lambda_{32} = 1, \lambda_{11} \neq 0.$ The above matrix $(\lambda_{ij})_{3\times3}$ is called the automorphism matrix of $A_3$, denoted by $\Phi$.
\end{proposition}
\begin{proof}
By \eqref{eq:Clam}, we have
\begin{equation*}
  \left\{\begin{split}
  \lambda_{11} ( \lambda_{22} \lambda_{33} - \lambda_{23} \lambda_{32} - 1 ) = 0,\\
  \lambda_{12} = \lambda_{13} =0.
  \end{split}\right.
\end{equation*}
Since $\phi$ is an automorphism, $\phi(e_1)\neq 0$ , we have $\lambda_{11} \neq 0$. Thus, the conclusion holds.
\end{proof}

\begin{remark}
Let $(A,\cdot,[\cdot,\cdot,\cdot])$ be a transposed Poisson $3$-Lie algebra with the basis $\left\{e_{1}, e_{2}, \cdots, e_{n}\right\}$ and let $L_{e_i}:A \rightarrow A$ be the left multiplication on $(A,\cdot)$.
Then, for all $1\leq i,j \leq n$, we have
\begin{equation}\label{eq:var-com}
  L_{e_i} (e_j) = e_i \cdot e_j = e_j \cdot e_i = L_{e_j} (e_i).
\end{equation}
Moreover, $L_{e_i}$  gives a $\frac{1}{3}$-derivation of the 3-Lie algebra $(A,[\cdot,\cdot,\cdot])$.
Conversely, for any $\frac{1}{3}$-derivation $\varphi:A \rightarrow A$ defined by \eqref{eq:varei}, it induces a transposed Poisson $3$-Lie algebra.
\end{remark}

\begin{proposition}  \label{com-multi}
Let $(A_3,\cdot,[\cdot,\cdot,\cdot])$ be a $3$-dimensional transposed Poisson $3$-Lie algebra defined on $(A_3,[\cdot,\cdot,\cdot])$, and let $L_{e_i}:A_3 \rightarrow A_3$ be the left multiplication of $(A_3,\cdot)$ defined by
\begin{equation} \label{eq:eimultiej}
  L_{e_i}(e_j)= e_i \cdot e_j = \sum_{k=1}^{3} \beta_{jk}^i e_k,\quad \forall~ \beta_{jk}^i \in \mathbb{C},~ 1 \leq i,j \leq 3.
\end{equation}
Then the commutative multiplication of $(A_3,\cdot,[\cdot,\cdot,\cdot])$ is given by
\begin{equation}  \label{eq:com-multi}
  \left\{\begin{split}
  e_1\cdot e_1&=0,\\
  e_1 \cdot e_2 &= \frac{\beta_{22}^2+\beta_{33}^2}{2}e_1, \\
  e_1 \cdot e_3 &= \frac{\beta_{32}^2+\beta_{33}^3}{2}e_1, \\
  e_2 \cdot e_2 &= \beta_{21}^2 e_1+\beta_{22}^2 e_2 +\beta_{23}^2 e_3, \\
  e_2 \cdot e_3 &= \beta_{31}^2 e_1+\beta_{32}^2 e_2 +\beta_{33}^2 e_3, \\
  e_3 \cdot e_3 &= \beta_{31}^3 e_1+\beta_{32}^3 e_2 +\beta_{33}^3 e_3,
  \end{split}\right.
\end{equation}
where $\beta_{jk}^i \in \mathbb{C},~ 2 \leq i,j \leq 3,~ 1 \leq k \leq 3$.
\end{proposition}
\begin{proof}
By Proposition \ref{vd-beta} and \eqref{eq:eimultiej}, for all $e_i, ~1 \leq i \leq 3$, we have
\begin{equation} \label{eq:betai}
  \left\{\begin{split}
  L_{e_i} (e_1) &= e_i \cdot e_1 = \frac{\beta_{22}^i + \beta_{33}^i }{2} e_1,\\
  L_{e_i} (e_2) &= e_i \cdot e_2 = \beta_{21}^i e_1 + \beta_{22}^i e_2 + \beta_{23}^i e_3,\\
  L_{e_i} (e_3) &= e_i \cdot e_3 = \beta_{31}^i e_1 + \beta_{32}^i e_2 + \beta_{33}^i e_3,
  \end{split}\right.
\end{equation}
where $\beta_{jk}^i \in \mathbb{C},~ 2 \leq j \leq 3,~ 1 \leq k \leq 3$.
By the commutative law of $(A,\cdot)$, we have
\begin{equation*}
  \left\{\begin{split}
  2\beta_{21}^1 &= \beta_{22}^2 + \beta_{33}^2 ,\\
  2\beta_{31}^1 &= \beta_{22}^3 + \beta_{33}^3 , \\
  \beta_{21}^3 &= \beta_{31}^2, ~\beta_{22}^3 = \beta_{32}^2, ~\beta_{23}^3 = \beta_{33}^2,\\
    \beta_{22}^1 &= \beta_{23}^1 =\beta_{32}^1 = \beta_{33}^1 =0.\\
  \end{split}\right.
\end{equation*}
Thus, we get the  commutative multiplication of $(A_3,\cdot,[\cdot,\cdot,\cdot])$ satisfying \eqref{eq:com-multi}.
%We can verify the result satisfies the \eqref{eq:TP3}.
\end{proof}

\begin{remark}  \label{asso}%交换不一定结合，方程未解出
If the commutative multiplication in Proposition \ref{com-multi} are associative, then the structure constants $\beta_{jk}^i \in \mathbb{C},~ 2 \leq i,j \leq 3,~ 1 \leq k \leq 3$, satisfy the following equations
\begin{equation*}
  \left\{\begin{split}
  \beta_{22}^2 \beta_{22}^2 + 2\beta_{23}^2 \beta_{32}^2 + 2\beta_{23}^2 \beta_{33}^3 &= \beta_{33}^2 \beta_{33}^2,  \\
  \beta_{22}^2 \beta_{32}^2 + 3\beta_{32}^2 \beta_{33}^2 + \beta_{33}^2 \beta_{33}^3 &= \beta_{22}^2 \beta_{33}^3,  \\
  2\beta_{22}^2 \beta_{32}^3 + 2\beta_{33}^2 \beta_{32}^3 + \beta_{33}^3 \beta_{33}^3 &= \beta_{32}^2 \beta_{32}^2,  \\
  \beta_{21}^2 \beta_{33}^3 + \beta_{22}^2 \beta_{31}^2 + 2\beta_{23}^2 \beta_{31}^3 &= \beta_{21}^2 \beta_{32}^2 + 3\beta_{31}^2 \beta_{33}^2 , \\
  \beta_{32}^2 \beta_{33}^2 &= \beta_{23}^2 \beta_{32}^3, \\
  \beta_{23}^2 \beta_{32}^2 + \beta_{33}^2 \beta_{33}^2 &= \beta_{22}^2 \beta_{33}^2 + \beta_{23}^2 \beta_{33}^3, \\
  \beta_{31}^3 \beta_{22}^2 + 2\beta_{21}^2 \beta_{32}^3 + \beta_{31}^2 \beta_{33}^3 &= \beta_{31}^3 \beta_{33}^2 + 3\beta_{31}^2 \beta_{32}^2, \\
  \beta_{32}^2 \beta_{32}^2 + \beta_{33}^2 \beta_{32}^3 &= \beta_{22}^2 \beta_{32}^3 + \beta_{32}^2 \beta_{33}^3.
  \end{split}\right.
\end{equation*}
\end{remark}

The automorphisms of $3$-dimensional transposed Poisson $3$-Lie algebra $(A_3,\cdot,[\cdot,\cdot,\cdot])$ are both the automorphisms of commutative algebra $(A_3,\cdot)$ and the automorphisms of $3$-Lie algebra $(A_3,[\cdot,\cdot,\cdot])$.
The automorphisms of $3$-Lie algebra $(A_3,[\cdot,\cdot,\cdot])$ have already been described in Proposition \ref{phi-lamdae}. Now we give a description of an automorphism of the transposed Poisson 3-Lie algebra $(A_3,\cdot, [\cdot,\cdot,\cdot])$ in the case $L_{e_1}=0$.

\begin{lemma}\label{lem:equations}
The linear mapping $\phi$ defined by \eqref{eq:phi-ei-3} is an automorphism of $3$-dimensional transposed Poisson $3$-Lie algebra $(A_3,\cdot,[\cdot,\cdot,\cdot])$ with $L_{e_1}=0$ if and only if the structure constants $\beta_{jk}^i, \lambda_{pq} \in \mathbb{C},~ (2 \leq i,j \leq 3,~ 1 \leq k,p,q \leq 3)$, satisfy the following thirteen equations:
\begin{enumerate}[(i)]
\item $\beta_{22}^2+\beta_{33}^2=0$,
\item $\beta_{32}^2+\beta_{33}^3=0$,
  \item $\beta_{22}^2 + \beta_{33}^2 =\lambda_{33} (\beta_{22}^2 + \beta_{33}^2 ) - \lambda_{23} ( \beta_{32}^2 + \beta_{33}^3 ),$
  \item $\beta_{32}^2 + \beta_{33}^3 = -\lambda_{32} (\beta_{22}^2+\beta_{33}^2 ) + \lambda_{22} ( \beta_{32}^2 + \beta_{33}^3 ),$
  \item $\beta_{22}^2 = \lambda_{22} \lambda_{22} \lambda_{33} \beta_{22}^2 - \lambda_{22} \lambda_{22} \lambda_{32} \beta_{23}^2 + 2\lambda_{22}\lambda_{23}\lambda_{33}\beta_{32}^2 -2\lambda_{22} \lambda_{23} \lambda_{32} \beta_{33}^2 + \lambda_{23} \lambda_{23} \lambda_{33} \beta_{32}^3
      \\- \lambda_{23} \lambda_{23} \lambda_{32} \beta_{33}^3,$
  \item $\beta_{23}^2 = -\lambda_{22}\lambda_{22}\lambda_{23} \beta_{22}^2 + \lambda_{22}\lambda_{22}\lambda_{22}\beta_{23}^2 -2 \lambda_{22} \lambda_{23}\lambda_{23}\beta_{32}^2 +2\lambda_{22} \lambda_{22} \lambda_{23}\beta_{33}^2 -\lambda_{23}\lambda_{23}\lambda_{23}\beta_{32}^3 \\+\lambda_{22}\lambda_{23}\lambda_{23} \beta_{33}^3,$
  \item $\beta_{32}^2 =\lambda_{22}\lambda_{32}\lambda_{33} \beta_{22}^2 - \lambda_{22} \lambda_{32} \lambda_{32}\beta_{23}^2 +\lambda_{33} (2\lambda_{22}\lambda_{33}-1)\beta_{32}^2 -\lambda_{32} (2\lambda_{22}\lambda_{33}-1)\beta_{33}^2 +\lambda_{23}\lambda_{33}\lambda_{33}\beta_{32}^3
      \\-\lambda_{23} \lambda_{32}\lambda_{33}\beta_{33}^3,$
  \item $\beta_{33}^2 = -\lambda_{22} \lambda_{23} \lambda_{32}\beta_{22}^2 + \lambda_{22} \lambda_{22}\lambda_{32}\beta_{23}^2 -\lambda_{23}(2\lambda_{22}\lambda_{33}-1)\beta_{32}^2 +\lambda_{22}(2\lambda_{22} \lambda_{33}-1)\beta_{33}^2 - \lambda_{23} \lambda_{23} \lambda_{33} \beta_{32}^3
      \\+ \lambda_{22} \lambda_{23} \lambda_{33}\beta_{33}^3,$
  \item $\beta_{32}^3 =\lambda_{32} \lambda_{32} \lambda_{33} \beta_{22}^2 - \lambda_{32} \lambda_{32} \lambda_{32} \beta_{23}^2 + 2 \lambda_{32} \lambda_{33} \lambda_{33} \beta_{32}^2 -2\lambda_{32} \lambda_{32} \lambda_{33} \beta_{33}^2 +\lambda_{33} \lambda_{33}\lambda_{33} \beta_{32}^3
      \\-\lambda_{32}\lambda_{33}\lambda_{33} \beta_{33}^3,$
  \item $\beta_{33}^3 = -\lambda_{23} \lambda_{32} \lambda_{32} \beta_{22}^2 + \lambda_{22}\lambda_{32}\lambda_{32}\beta_{23}^2 - 2\lambda_{23} \lambda_{32} \lambda_{33} \beta_{32}^2 +2\lambda_{22} \lambda_{32} \lambda{33} \beta_{33}^2 - \lambda_{23} \lambda_{33}\lambda_{33} \beta_{32}^3
      \\+\lambda_{22}\lambda_{33}\lambda_{33} \beta_{33}^3,$
  \item $\beta_{21}^2 =\frac{1}{\lambda_{11}}(\lambda_{22} \lambda_{22}\beta_{21}^2 -\lambda_{31} \beta_{23}^2 + \lambda_{21} \beta_{33}^2 + 2\lambda_{22} \lambda_{23} \beta_{31}^2 + \lambda_{23} \lambda_{23} \beta_{31}^3),$
  \item $\beta_{31}^2 =\frac{1}{\lambda_{11}}(\frac{\lambda_{31}(\beta_{22}^2 - \beta_{33}^2)+\lambda_{21}(\beta_{33}^3 - \beta_{32}^2)}{2} + \lambda_{22} \lambda_{32} \beta_{21}^2 +(\lambda_{22}\lambda_{33}+ \lambda_{23} \lambda_{32})\beta_{31}^2+\lambda_{23}\lambda_{33} \beta_{31}^3),$
  \item $\beta_{31}^3 = \frac{1}{\lambda_{11}}(\lambda_{32} \lambda_{32} \beta_{21}^2 + 2 \lambda_{32} \lambda_{33} \beta_{31}^2 + \lambda_{31} \beta_{32}^2 + \lambda_{33} \lambda_{33} \beta_{31}^3 -\lambda_{21} \beta_{32}^3),$
\end{enumerate}
where $\lambda_{22} \neq 0$, $\lambda_{23} \neq 0$, $\lambda_{32} \neq 0$, $\lambda_{33} \neq 0$, $\lambda_{11} \neq 0$.
\end{lemma}

\begin{proof}
Since  $L_{e_1}=0$, we have $e_1 \cdot e_2 = e_1 \cdot e_3 =0$ and
\begin{equation*}\label{eq:conditions}
  \beta_{22}^2+\beta_{33}^2=0,\quad \beta_{32}^2+\beta_{33}^3=0.
\end{equation*}
%If $\phi$ is an automorphism of $3$-dimensional transposed Poisson $3$-Lie algebra $(A_3,\cdot,[\cdot,\cdot,\cdot])$, then we have
%\begin{equation*}
% \begin{split}
%  \phi(e_1 \cdot e_2) = &\phi(e_1)\cdot \phi(e_2),\quad
%  \phi(e_1 \cdot e_3) = \phi(e_1)\cdot \phi(e_3),\quad
%  \phi(e_2 \cdot e_2) = \phi(e_2)\cdot \phi(e_2), \\
%  &\phi(e_2 \cdot e_3) = \phi(e_2)\cdot \phi(e_3),\quad
%  \phi(e_3 \cdot e_3) = \phi(e_3)\cdot \phi(e_3).
% \end{split}
%\end{equation*}
By \eqref{eq:phi-lamdae} and \eqref{eq:com-multi}, we have
\begin{eqnarray*}
  && \phi(e_1 \cdot e_2) - \phi(e_1)\cdot \phi(e_2) \\
  &=& \phi(\frac{\beta_{22}^2+\beta_{33}^2}{2}e_1) - \lambda_{11} e_1 \cdot (\lambda_{21} e_1 + \lambda_{22} e_2 + \lambda_{23} e_3)  \\
  &=& \frac{\beta_{22}^2+\beta_{33}^2}{2} \lambda_{11} e_1 - \lambda_{11}\lambda_{22} \frac{\beta_{22}^2+\beta_{33}^2}{2}e_1 - \lambda_{11}\lambda_{23} \frac{\beta_{32}^2+\beta_{33}^3}{2}e_1 \\
  &=& \lambda_{11} (\frac{\beta_{22}^2+\beta_{33}^2}{2} -\lambda_{22} \frac{\beta_{22}^2+\beta_{33}^2}{2} - \lambda_{23} \frac{\beta_{32}^2+\beta_{33}^3}{2}) e_1, \\
  && \phi(e_1 \cdot e_3) - \phi(e_1)\cdot \phi(e_3) \\
  &=& \phi(\frac{\beta_{32}^2+\beta_{33}^3}{2}e_1) - \lambda_{11} e_1 \cdot (\lambda_{31} e_1 + \lambda_{32} e_2 + \lambda_{33} e_3)  \\
  &=& \frac{\beta_{32}^2+\beta_{33}^3}{2} \lambda_{11} e_1 - \lambda_{11}\lambda_{32} \frac{\beta_{22}^2+\beta_{33}^2}{2}e_1 - \lambda_{11}\lambda_{33} \frac{\beta_{32}^2+\beta_{33}^3}{2}e_1 \\
  &=& \lambda_{11} (\frac{\beta_{32}^2+\beta_{33}^3}{2} -\lambda_{32} \frac{\beta_{22}^2+\beta_{33}^2}{2} - \lambda_{33} \frac{\beta_{32}^2+\beta_{33}^3}{2}) e_1,  \\
  && \phi(e_2 \cdot e_2) - \phi(e_2)\cdot \phi(e_2) \\
  &=& \phi(\beta_{21}^2 e_1+\beta_{22}^2 e_2 +\beta_{23}^2 e_3) - (\lambda_{21} e_1 + \lambda_{22} e_2 + \lambda_{23} e_3) \cdot (\lambda_{21} e_1 + \lambda_{22} e_2 + \lambda_{23} e_3)  \\
  &=& \Big(\lambda_{11}\beta_{21}^2 + \lambda_{21}\beta_{22}^2 + \lambda_{31}\beta_{23}^2 - \lambda_{21}\lambda_{22}(\beta_{22}^2+\beta_{33}^2) - \lambda_{21}\lambda_{23} (\beta_{32}^2+\beta_{33}^3) -2\lambda_{22}\lambda_{23} \beta_{31}^2 \\
  &&- \lambda_{22}\lambda_{22} \beta_{21}^2 - \lambda_{23}\lambda_{23} \beta_{31}^3 \Big)e_1 + (\lambda_{22}\beta_{22}^2 + \lambda_{32}\beta_{23}^2 - 2\lambda_{22}\lambda_{23} \beta_{32}^2 - \lambda_{22}\lambda_{22} \beta_{22}^2 - \lambda_{23}\lambda_{23} \beta_{32}^3)e_2 \\
  &&+ (\lambda_{23}\beta_{22}^2 + \lambda_{33}\beta_{23}^2 - 2\lambda_{22}\lambda_{23} \beta_{33}^2 - \lambda_{22}\lambda_{22} \beta_{23}^2 - \lambda_{23}\lambda_{23} \beta_{33}^3)e_3, \\
  && \phi(e_2 \cdot e_3) - \phi(e_2)\cdot \phi(e_3) \\
  &=& \phi(\beta_{31}^2 e_1+\beta_{32}^2 e_2 +\beta_{33}^2 e_3) - (\lambda_{21} e_1 + \lambda_{22} e_2 + \lambda_{23} e_3) \cdot (\lambda_{31} e_1 + \lambda_{32} e_2 + \lambda_{33} e_3)  \\
  &=& \Big(\lambda_{11}\beta_{31}^2 + \lambda_{21}\beta_{32}^2 + \lambda_{31}\beta_{33}^2 - (\lambda_{21}\lambda_{32}+ \lambda_{22}\lambda_{31})\frac{\beta_{22}^2+\beta_{33}^2}{2} - (\lambda_{21}\lambda_{33}+ \lambda_{23}\lambda_{31}) \frac{\beta_{32}^2+\beta_{33}^3}{2} \\
  &&- (\lambda_{22}\lambda_{33} +\lambda_{23}\lambda_{32}) \beta_{31}^2 - \lambda_{22}\lambda_{32} \beta_{21}^2 - \lambda_{23}\lambda_{33} \beta_{31}^3 \Big)e_1 + \Big(\lambda_{22}\beta_{32}^2 + \lambda_{32}\beta_{33}^2 \\
  &&- (\lambda_{22}\lambda_{33} +\lambda_{23}\lambda_{32}) \beta_{32}^2 -\lambda_{22}\lambda_{32} \beta_{22}^2 - \lambda_{23}\lambda_{33} \beta_{32}^3\Big)e_2 + \Big(\lambda_{23}\beta_{32}^2 + \lambda_{33}\beta_{33}^2 \\
  &&- (\lambda_{22}\lambda_{33} +\lambda_{23}\lambda_{32}) \beta_{33}^2 - \lambda_{22}\lambda_{32} \beta_{23}^2 - \lambda_{23}\lambda_{33} \beta_{33}^3\Big)e_3,  \\
  && \phi(e_3 \cdot e_3) - \phi(e_3)\cdot \phi(e_3) \\
  &=& \phi(\beta_{31}^3 e_1+\beta_{32}^3 e_2 +\beta_{33}^3 e_3) - (\lambda_{31} e_1 + \lambda_{32} e_2 + \lambda_{33} e_3) \cdot (\lambda_{31} e_1 + \lambda_{32} e_2 + \lambda_{33} e_3)  \\
  &=& \Big(\lambda_{11}\beta_{31}^3 + \lambda_{21}\beta_{32}^3 + \lambda_{31}\beta_{33}^3 - \lambda_{31}\lambda_{32}(\beta_{22}^2+\beta_{33}^2) - \lambda_{31}\lambda_{33} (\beta_{32}^2+\beta_{33}^3) -2\lambda_{32}\lambda_{33} \beta_{31}^2 \\
  &&- \lambda_{32}\lambda_{32} \beta_{21}^2 - \lambda_{33}\lambda_{33} \beta_{31}^3 \Big)e_1 + (\lambda_{22}\beta_{32}^3 + \lambda_{32}\beta_{33}^3 - 2\lambda_{32}\lambda_{33} \beta_{32}^2 - \lambda_{32}\lambda_{32} \beta_{22}^2 - \lambda_{33}\lambda_{33} \beta_{32}^3)e_2 \\
  &&+ (\lambda_{23}\beta_{32}^3 + \lambda_{33}\beta_{33}^3 - 2\lambda_{32}\lambda_{33} \beta_{33}^2 - \lambda_{32}\lambda_{32} \beta_{23}^2 - \lambda_{33}\lambda_{33} \beta_{33}^3)e_3.
\end{eqnarray*}
Then $\phi$ is an automorphism of $3$-dimensional transposed Poisson $3$-Lie algebra $(A_3,\cdot,[\cdot,\cdot,\cdot])$ if and only if the equations (iii) $\sim$ (xiii) hold.
\end{proof}

From Lemma \ref{lem:equations}, the structure constants $\beta_{ij}^k$ of the commutative algebra $(A_3,\cdot)$ are only satisfy following four cases:
\begin{itemize}
  \item Case (I): $\beta_{32}^2=\beta_{33}^3=0$, $\beta_{22}^2=-\beta_{33}^2 \neq 0$;
  \item Case (II): $\beta_{22}^2=\beta_{33}^2=0$, $\beta_{32}^2=-\beta_{33}^3 \neq 0$;
  \item Case (III): $\beta_{22}^2=-\beta_{33}^2 \neq 0$, $\beta_{32}^2=-\beta_{33}^3\neq 0$;
  \item Case (IV): $\beta_{22}^2=\beta_{33}^2=\beta_{32}^2=\beta_{33}^3=0$.
\end{itemize}
By the above notations, we give the classification theorem of $(A_3, \cdot, [\cdot,\cdot,\cdot])$ in the case $L_{e_1}=0$, and see Section \ref{sec:proof} for the detailed proof.

\begin{theorem}\label{classification}
Let $(A_3,\cdot,[\cdot,\cdot,\cdot])$ be a $3$-dimensional transposed Poisson $3$-Lie algebra and the commutative multiplication be given by \eqref{eq:eimultiej} with $L_{e_1}=0$. Then $(A_3,\cdot,[\cdot,\cdot,\cdot])$ is isomorphic to one of the following algebras:
$$
  \quad T_1:~\left\{\begin{split}
  &e_2 \cdot e_2 = \alpha e_2,~ e_2 \cdot e_3 = -\alpha e_3,~ e_3 \cdot e_3 = -3\alpha e_2,\\
  &[e_1,e_2,e_3] =e_1;
  \end{split}\right.\qquad\qquad\qquad\qquad\qquad\qquad\qquad\qquad\qquad
$$
$$
  \quad T_2:~\left\{\begin{split}
  &e_2 \cdot e_2 = \theta e_1 + \alpha e_2,~ e_2 \cdot e_3 = - \alpha e_3,~ e_3 \cdot e_3 = 3\theta e_1 - 3\alpha e_2,\\
  &[e_1,e_2,e_3]=e_1;
  \end{split}\right.\qquad\qquad\qquad\qquad\qquad\qquad\qquad
$$
$$
  \quad T_3:~\left\{\begin{split}
  &e_2 \cdot e_2 = \alpha e_2 + \alpha e_3,~ e_2 \cdot e_3 = -\alpha e_3,~ e_3 \cdot e_3 = -3\alpha e_2,\\
  &[e_1,e_2,e_3]=e_1;
  \end{split}\right.\qquad\qquad\qquad\qquad\qquad\qquad\qquad\qquad
$$
$$
  \quad T_4:~\left\{\begin{split}
  &e_2 \cdot e_2 = \alpha e_1 + \alpha e_2 + \alpha e_3,~ e_2 \cdot e_3 = - \frac{1}{2} \alpha e_1 - \alpha e_3,~ e_3 \cdot e_3 = - 3\alpha e_2,\\
  &[e_1,e_2,e_3]=e_1;
  \end{split}\right.\qquad\qquad\qquad\qquad\qquad
$$
$$
  \quad T_{5}:~\left\{\begin{split}
  &e_2 \cdot e_2 = \alpha e_2,~ e_2 \cdot e_3 = \alpha e_1 -3 \alpha e_2 -\alpha e_3,~ e_3 \cdot e_3 = -2\alpha e_1 + 6\alpha e_2 + 3\alpha e_3,\\
  &[e_1,e_2,e_3]=e_1;
  \end{split}\right.\qquad\qquad\qquad\qquad\qquad\qquad\qquad\qquad
$$
$$
  \quad T_{6}:~\left\{\begin{split}
  &e_2 \cdot e_2 = \alpha e_1 +\alpha e_2,~ e_2 \cdot e_3 = -\alpha e_1 -3\alpha e_2 -\alpha e_3,~ e_3 \cdot e_3 = 6\alpha e_2 + 3\alpha e_3,\\
  &[e_1,e_2,e_3]=e_1;
  \end{split}\right.\qquad\qquad\qquad\qquad\qquad\qquad
$$
%$$
%  \quad T_{7}:~\left\{\begin{split}
%  &e_2 \cdot e_2 = \alpha e_2 +\theta e_3,~ e_2 \cdot e_3 = (\alpha-2\theta) e_1 +(3\theta -3 \alpha) e_2 -\alpha e_3,\\
%  &e_3 \cdot e_3 = (6\theta -2\alpha) e_1 + (6\alpha -9\theta) e_2 + (3\alpha-3\theta) e_3,\\
%  &[e_1,e_2,e_3]=e_1;
%  \end{split}\right.\qquad\qquad\qquad\qquad\qquad\qquad\qquad\qquad
%$$
$$
  \quad T_{7}:~\left\{\begin{split}
  &e_2 \cdot e_2 = \theta e_1 +\alpha e_2 +\theta e_3,~ e_2 \cdot e_3 = (3\theta -3\alpha) e_2 -\alpha e_3,\\
  &e_3 \cdot e_3 = (3\theta -3\alpha) e_1 + (6\alpha-9\theta) e_2 + (3\alpha -3\theta) e_3,\\
  &[e_1,e_2,e_3]=e_1;
  \end{split}\right.\qquad\qquad\qquad\qquad\qquad\qquad\qquad
$$
$$
  \quad T_{8}:~\left\{\begin{split}
  &e_2 \cdot e_2 = (\alpha -\theta) e_1 +\alpha e_2 +\theta e_3,~ e_2 \cdot e_3 = (-\alpha +\frac{3}{2}\theta) e_1 + (3\theta -3\alpha) e_2 -\alpha e_3,\\
  &e_3 \cdot e_3 = (6\alpha -9\theta) e_2 + (3\alpha-3\theta) e_3,\\
  &[e_1,e_2,e_3]=e_1;
  \end{split}\right.\qquad\qquad\qquad\qquad\qquad\qquad
$$
$$
  \quad T_{9}:~\left\{\begin{split}
  &e_2 \cdot e_2 = \rho e_1 + \alpha e_2 +\theta e_3,~ e_2 \cdot e_3 = -\frac{3}{2} \rho e_1 +(3\theta -3\alpha) e_2 -\alpha e_3,\\
  &e_3 \cdot e_3 = 3\rho e_1 + (6\alpha -9\theta) e_2 + (3\alpha -3\theta) e_3,\\
  &[e_1,e_2,e_3]=e_1;
  \end{split}\right.\qquad\qquad\qquad\qquad\qquad
$$
$$
  \quad T_{10}:~\left\{\begin{split}
  &e_2 \cdot e_2 = \alpha e_1,~ e_3 \cdot e_3 = -3\alpha e_1,\\
  &[e_1,e_2,e_3]=e_1.
  \end{split}\right.\qquad\qquad\qquad\qquad\qquad\qquad\qquad\qquad\qquad\qquad\qquad\qquad
$$
Where the parameters $\alpha, \theta, \rho \in \mathbb{C}, \alpha \theta \rho \neq 0$.
\end{theorem}

\section{Proof of Theorem \ref{classification}} \label{sec:proof}

In this section, we provide a detailed proof of Theorem \ref{classification}.
\subsection{The four cases in Lemma \ref{lem:equations} }
By discussing the four cases \textbf{(I-IV)} in Lemma 3.8, we obtain transposed Poisson $3$-Lie algebras $T_i', 1\leq i \leq 26$.

We first discuss the \textbf{Case (I)}. By (v) $\sim$ (x) in Lemma \ref{lem:equations} we have
\begin{equation} \label{eq:beta4}
  \left\{\begin{split}
  \beta_{22}^2 &= (\lambda_{22} \lambda_{22} \lambda_{33} + 2\lambda_{22} \lambda_{23} \lambda_{32}) \beta_{22}^2 - \lambda_{22} \lambda_{22} \lambda_{32} \beta_{23}^2 + \lambda_{23} \lambda_{23} \lambda_{33} \beta_{32}^3, \\
  \beta_{23}^2 &= -3\lambda_{22}\lambda_{22}\lambda_{23} \beta_{22}^2 + \lambda_{22} \lambda_{22} \lambda_{22} \beta_{23}^2 - \lambda_{23} \lambda_{23} \lambda_{23} \beta_{32}^3, \\
  0 &= \lambda_{32}(3\lambda_{22}\lambda_{33} - 1) \beta_{22}^2 - \lambda_{22} \lambda_{32} \lambda_{32}\beta_{23}^2 + \lambda_{23} \lambda_{33} \lambda_{33} \beta_{32}^3, \\
  \beta_{32}^3 &= 3\lambda_{32} \lambda_{32} \lambda_{33} \beta_{22}^2 - \lambda_{32} \lambda_{32} \lambda_{32} \beta_{23}^2 + \lambda_{33} \lambda_{33} \lambda_{33} \beta_{32}^3. \\
  \end{split}\right.
\end{equation}
Then we distinguish the following two subcases depending on whether $\beta_{23}^2 = 0$ or not.

\textbf{I-S1}: If $\beta_{23}^2 = 0$, then \eqref{eq:beta4} implies
\begin{align*}
  \beta_{22}^2 &= (\lambda_{22} \lambda_{22} \lambda_{33} + 2\lambda_{22} \lambda_{23} \lambda_{32}) \beta_{22}^2 + \lambda_{23} \lambda_{23} \lambda_{33} \beta_{32}^3, \\
  0 &= -3\lambda_{22}\lambda_{22}\lambda_{23} \beta_{22}^2 - \lambda_{23} \lambda_{23} \lambda_{23} \beta_{32}^3, \\
  0 &= \lambda_{32}(3\lambda_{22}\lambda_{33} - 1) \beta_{22}^2 + \lambda_{23} \lambda_{33} \lambda_{33} \beta_{32}^3, \\
  \beta_{32}^3 &= 3\lambda_{32} \lambda_{32} \lambda_{33} \beta_{22}^2 + \lambda_{33} \lambda_{33} \lambda_{33} \beta_{32}^3.
\end{align*}
We get
\begin{equation*}
  \lambda_{22} = \lambda_{33}= -\frac{1}{2},~  \lambda_{23} \lambda_{32} = -\frac{3}{4},~  \beta_{22}^2 = - \frac{4}{3} \lambda_{23} \lambda_{23} \beta_{32}^3.
\end{equation*}
Thus, $\beta_{32}^3 \neq 0$. Without loss of generality, we can suppose $\lambda_{11} =1 $. Then equations (xi) $\sim$ (xiii) in Lemma \ref{lem:equations} imply
\begin{equation} \label{eq:beta3.1}
  \left\{\begin{split}
  \beta_{21}^2 &= \frac{1}{4} \beta_{21}^2 - \lambda_{21} \beta_{22}^2 - \lambda_{23} \beta_{31}^2 + \lambda_{23} \lambda_{23} \beta_{31}^3, \\
  \beta_{31}^2 &= \lambda_{31} \beta_{22}^2 - \frac{1}{2} \lambda_{32} \beta_{21}^2 -\frac{1}{2} \beta_{31}^2 -\frac{1}{2} \lambda_{23} \beta_{31}^3,  \\
  \beta_{31}^3 &= \lambda_{32} \lambda_{32} \beta_{21}^2 - \lambda_{32} \beta_{31}^2 + \frac{1}{4} \beta_{31}^3 + \frac{3\lambda_{21} }{4\lambda_{23} \lambda_{23} } \beta_{22}^2.  \\
  \end{split}\right.
\end{equation}

Next, we distinguish the following four sub-subcases depending on whether $\beta_{21}^2, \beta_{31}^2, \beta_{31}^3$ equal to zero or not.

\textbf{(I-S1-a)}~ If $\beta_{21}^2 = 0$, then by \eqref{eq:beta3.1} we have
\begin{align*}
  0 &= - \lambda_{21} \beta_{22}^2 - \lambda_{23} \beta_{31}^2 + \lambda_{23} \lambda_{23} \beta_{31}^3, \\
  \beta_{31}^2 &= \lambda_{31} \beta_{22}^2 -\frac{1}{2} \beta_{31}^2 -\frac{1}{2} \lambda_{23} \beta_{31}^3,  \\
  \beta_{31}^3 &= - \lambda_{32} \beta_{31}^2 + \frac{1}{4} \beta_{31}^3 + \frac{3\lambda_{21} }{4\lambda_{23} \lambda_{23} } \beta_{22}^2.
\end{align*}
We can suppose $\lambda_{23}=\frac{1}{2}$, $\lambda_{32}=-\frac{3}{2}$, $\lambda_{21}=0$, $\lambda_{31}=0$. Then $\beta_{31}^2 = \beta_{31}^3 = 0$.
Under the action of the automorphism
$$\Phi_1=\left(
  \begin{array}{ccc}
    1 & 0 & 0 \\
    0 & -\frac{1}{2} & \frac{1}{2} \\
    0 & -\frac{3}{2} & -\frac{1}{2} \\
  \end{array}
\right),
$$
$(A_3,\cdot,[\cdot,\cdot,\cdot])$ is isomorphic to
\begin{equation*}
  T_1'(T_1):~ \left\{\begin{split}
  e_2 \cdot e_2 &= \beta_{22}^2 e_2, \\
  e_2 \cdot e_3 &= -\beta_{22}^2 e_3, \\
  e_3 \cdot e_3 &= -3\beta_{22}^2 e_2.
  \end{split}\right.
\end{equation*}

\textbf{(I-S1-b)}~ If $\beta_{21}^2 \neq 0$, $\beta_{31}^2 = 0$, then by \eqref{eq:beta3.1} we have
\begin{align*}
  \beta_{21}^2 &= \frac{1}{4} \beta_{21}^2 - \lambda_{21} \beta_{22}^2 + \lambda_{23} \lambda_{23} \beta_{31}^3, \\
  0 &= \lambda_{31} \beta_{22}^2 - \frac{1}{2} \lambda_{32} \beta_{21}^2 -\frac{1}{2} \lambda_{23} \beta_{31}^3,  \\
  \beta_{31}^3 &= \lambda_{32} \lambda_{32} \beta_{21}^2 + \frac{1}{4} \beta_{31}^3 + \frac{3\lambda_{21} }{4\lambda_{23} \lambda_{23} } \beta_{22}^2.
\end{align*}
We can suppose $\lambda_{23}=\frac{1}{2}$, $\lambda_{32}=-\frac{3}{2}$, $\lambda_{21}=0$, $\lambda_{31}=0$. Then $\beta_{31}^3 = 3\beta_{21}^2$.
Under the action of the automorphism
$$\Phi_2=\left(
  \begin{array}{ccc}
    1 & 0 & 0 \\
    0 & -\frac{1}{2} & \frac{1}{2} \\
    0 & -\frac{3}{2} & -\frac{1}{2} \\
  \end{array}
\right),
$$
$(A_3,\cdot,[\cdot,\cdot,\cdot])$ is isomorphic to
\begin{equation*}
  T_2'(T_2):~ \left\{\begin{split}
  e_2 \cdot e_2 &= \beta_{21}^2 e_1 + \beta_{22}^2 e_2,  \\
  e_2 \cdot e_3 &= -\beta_{22}^2 e_3,  \\
  e_3 \cdot e_3 &= 3\beta_{21}^2 e_1 - 3\beta_{22}^2 e_2.
  \end{split}\right.
\end{equation*}

\textbf{(I-S1-c)}~ If $\beta_{21}^2 \neq 0$, $\beta_{31}^2 \neq 0$, $\beta_{31}^3 = 0$, then by \eqref{eq:beta3.1} we have
\begin{align*}
  \beta_{21}^2 &= \frac{1}{4} \beta_{21}^2 - \lambda_{21} \beta_{22}^2 - \lambda_{23} \beta_{31}^2 , \\
  \beta_{31}^2 &= \lambda_{31} \beta_{22}^2 - \frac{1}{2} \lambda_{32} \beta_{21}^2 -\frac{1}{2} \beta_{31}^2,  \\
  0 &= \lambda_{32} \lambda_{32} \beta_{21}^2 - \lambda_{32} \beta_{31}^2 + \frac{3\lambda_{21} }{4\lambda_{23} \lambda_{23} } \beta_{22}^2.
\end{align*}
We can suppose $\lambda_{23}= \frac{1}{2}$, $\lambda_{32}= -\frac{3}{2}$, $\lambda_{21}=1$, $\lambda_{31}=3$. Then $\beta_{21}^2 = -2\beta_{22}^2,~ \beta_{31}^2 = \beta_{22}^2$.
Under the action of the automorphism
$$\Phi_3=\left(
  \begin{array}{ccc}
    1 & 0 & 0 \\
    1 & -\frac{1}{2} & \frac{1}{2} \\
    3 & -\frac{3}{2} & -\frac{1}{2} \\
  \end{array}
\right),
$$
$(A_3,\cdot,[\cdot,\cdot,\cdot])$ is isomorphic to
\begin{equation*}
  T_3':~ \left\{\begin{split}
  e_2 \cdot e_2 &= -2\beta_{22}^2 e_1 + \beta_{22}^2 e_2,  \\
  e_2 \cdot e_3 &= \beta_{22}^2 e_1 - \beta_{22}^2 e_3,  \\
  e_3 \cdot e_3 &= -3\beta_{22}^2 e_2.
  \end{split}\right.
\end{equation*}

\textbf{(I-S1-d)}~ If $\beta_{21}^2 \neq 0$, $\beta_{31}^2 \neq 0$, $\beta_{31}^3 \neq 0$ and suppose $\lambda_{23}= \frac{1}{2}$, $\lambda_{32}= -\frac{3}{2}$, $\lambda_{21}=0$, $\lambda_{31}=2$, then by \eqref{eq:beta3.1} we have $\beta_{31}^2 = \beta_{22}^2,~ \beta_{31}^3 = 3\beta_{21}^2 + 2\beta_{22}^2$.
Under the action of the automorphism
$$\Phi_4=\left(
  \begin{array}{ccc}
    1 & 0 & 0 \\
    0 & -\frac{1}{2} & \frac{1}{2} \\
    2 & -\frac{3}{2} & -\frac{1}{2} \\
  \end{array}
\right),
$$
$(A_3,\cdot,[\cdot,\cdot,\cdot])$ is isomorphic to
\begin{equation*}
  T_4':~ \left\{\begin{split}
  e_2 \cdot e_2 &= \beta_{21}^2 e_1 + \beta_{22}^2 e_2,  \\
  e_2 \cdot e_3 &= \beta_{22}^2 e_1 - \beta_{22}^2 e_3,  \\
  e_3 \cdot e_3 &= (3\beta_{21}^2 + 2\beta_{22}^2) e_1 - 3\beta_{22}^2 e_2.
  \end{split}\right.
\end{equation*}
%Hence, we obtain the classification (1) in Theorem \ref{classification}.

\textbf{I-S2}: If $\beta_{23}^2 \neq 0$, then we have \eqref{eq:beta4}.
If $\beta_{32}^3 = 0$, then we get $\lambda_{22} \lambda_{33} = \lambda_{23} \lambda_{32}$, which contradicts with $\lambda_{22} \lambda_{33} - \lambda_{23} \lambda_{32} = 1$. So, $\beta_{32}^3 \neq 0$.
By \eqref{eq:beta4}, we have $\lambda_{22} + \lambda_{33} = -1$. Suppose $\lambda_{22}=-2$. Then we have $\lambda_{33}=1$, $\lambda_{23} \lambda_{32}=-3$. Thus, \eqref{eq:beta4} implies
\begin{align*}
  \beta_{22}^2 &= 16 \beta_{22}^2 - 4\lambda_{32} \beta_{23}^2 + \lambda_{23} \lambda_{23} \beta_{32}^3,  \\
  \beta_{23}^2 &= -12\lambda_{23} \beta_{22}^2 -8\beta_{23}^2 -\lambda_{23} \lambda_{23} \lambda_{23} \beta_{32}^3,  \\
  0 &= -7\lambda_{32} \beta_{22}^2 + 2\lambda_{32} \lambda_{32}\beta_{23}^2 +\lambda_{23} \beta_{32}^3,  \\
  \beta_{32}^3 &= 3\lambda_{32} \lambda_{32} \beta_{22}^2 - \lambda_{32} \lambda_{32} \lambda_{32} \beta_{23}^2 + \beta_{32}^3.
\end{align*}
Therefore, we have $\lambda_{32} \beta_{23}^2 = 3\beta_{22}^2,~ \lambda_{23} \beta_{32}^3 = \lambda_{32} \beta_{22}^2$.
We can suppose $\lambda_{11} =1 $. Then equations (xi) $\sim$ (xiii) in Lemma \ref{lem:equations} imply
\begin{equation}\label{eq:beta3.2}
  \left\{\begin{split}
  \beta_{21}^2 &= 4\beta_{21}^2 - \lambda_{21} \beta_{22}^2 - 4\lambda_{23} \beta_{31}^2 + \lambda_{23} \lambda_{23} \beta_{31}^3 - \frac{3\lambda_{31}} {\lambda_{32}} \beta_{22}^2, \\
  \beta_{31}^2 &= \lambda_{31} \beta_{22}^2 - 2\lambda_{32} \beta_{21}^2 - 5\beta_{31}^2 + \lambda_{23} \beta_{31}^3,  \\
  \beta_{31}^3 &= \lambda_{32} \lambda_{32} \beta_{21}^2 + 2\lambda_{32} \beta_{31}^2 + \beta_{31}^3 - \lambda_{21} \frac{\lambda_{32}}{\lambda_{23}} \beta_{22}^2.
  \end{split}\right.
\end{equation}

Next, we distinguish the following four sub-subcases depending on whether $\beta_{21}^2, \beta_{31}^2, \beta_{31}^3 $ equal to zero or not.

\textbf{(I-S2-a)}~ If $\beta_{21}^2 = 0$, then by \eqref{eq:beta3.2} we have
\begin{align*}
  0 &= -\lambda_{21} \beta_{22}^2 -4\lambda_{23} \beta_{31}^2 + \lambda_{23} \lambda_{23} \beta_{31}^3 -\frac{3\lambda_{31}}{\lambda_{32}}\beta_{22}^2, \\
  \beta_{31}^2 &= \lambda_{31} \beta_{22}^2 - 5\beta_{31}^2 + \lambda_{23} \beta_{31}^3,  \\
  \beta_{31}^3 &= 2\lambda_{32} \beta_{31}^2 + \beta_{31}^3 - \lambda_{21} \frac{\lambda_{32}}{\lambda_{23}} \beta_{22}^2.
\end{align*}
We can suppose $\lambda_{23}= -1$, $\lambda_{32}= 3$, $\lambda_{21}=0$, $\lambda_{31}=0$. Then $\beta_{31}^2 = \beta_{31}^3 = 0$.
Under the action of the automorphism
$$\Phi_5=\left(
  \begin{array}{ccc}
    1 & 0 & 0 \\
    0 & -2 & -1 \\
    0 & 3 & 1 \\
  \end{array}
\right),
$$
$(A_3,\cdot,[\cdot,\cdot,\cdot])$ is isomorphic to
\begin{equation*}
  T_5'(T_3):~ \left\{\begin{split}
   e_2 \cdot e_2 &= \beta_{22}^2 e_2 + \beta_{22}^2 e_3, \\
   e_2 \cdot e_3 &= -\beta_{22}^2 e_3, \\
   e_3 \cdot e_3 &= -3\beta_{22}^2 e_2.
  \end{split}\right.
\end{equation*}

\textbf{(I-S2-b)}~ If $\beta_{21}^2 \neq 0$, $\beta_{31}^2 = 0$, then by \eqref{eq:beta3.2} we have
\begin{align*}
  \beta_{21}^2 &= 4\beta_{21}^2 -\lambda_{21}\beta_{22}^2 +\lambda_{23} \lambda_{23} \beta_{31}^3 -\frac{3\lambda_{31}}{\lambda_{32}}\beta_{22}^2,\\
  0 &= \lambda_{31} \beta_{22}^2 - 2\lambda_{32} \beta_{21}^2 + \lambda_{23} \beta_{31}^3, \\
  \beta_{31}^3 &= \lambda_{32} \lambda_{32} \beta_{21}^2 + \beta_{31}^3 - \lambda_{21} \frac{\lambda_{32}}{\lambda_{23}} \beta_{22}^2.
\end{align*}
We can suppose $\lambda_{23}=-1$, $\lambda_{32}=3$, $\lambda_{21}=-3$, $\lambda_{31}=3$. Then $\beta_{21}^2 = \beta_{22}^2,~ \beta_{31}^3 = -3\beta_{22}^2,~ \beta_{23}^2 = \beta_{22}^2,~ \beta_{32}^3 = -3\beta_{22}^2$.
Under the action of the automorphism
$$\Phi_6=\left(
  \begin{array}{ccc}
    1 & 0 & 0 \\
    -3 & -2 & -1 \\
    3 & 3 & 1 \\
  \end{array}
\right),
$$
$(A_3,\cdot,[\cdot,\cdot,\cdot])$ is isomorphic to
\begin{equation*}
  T_6':~ \left\{\begin{split}
   e_2 \cdot e_2 &= \beta_{22}^2 e_1 + \beta_{22}^2 e_2 + \beta_{22}^2 e_3,  \\
   e_2 \cdot e_3 &= -\beta_{22}^2 e_3,  \\
   e_3 \cdot e_3 &= -3\beta_{22}^2 e_1 - 3\beta_{22}^2 e_2.
  \end{split}\right.
\end{equation*}

\textbf{(I-S2-c)}~ If $\beta_{21}^2 \neq 0$, $\beta_{31}^2 \neq 0$, $\beta_{31}^3 = 0$, then by \eqref{eq:beta3.2} we have
\begin{align*}
  \beta_{21}^2 &= 4\beta_{21}^2 - \lambda_{21} \beta_{22}^2 - 4\lambda_{23} \beta_{31}^2 - \frac{3\lambda_{31}} {\lambda_{32}} \beta_{22}^2, \\
  \beta_{31}^2 &= \lambda_{31} \beta_{22}^2 - 2\lambda_{32} \beta_{21}^2 - 5\beta_{31}^2,  \\
  0 &= \lambda_{32} \lambda_{32} \beta_{21}^2 + 2\lambda_{32} \beta_{31}^2 - \lambda_{21} \frac{\lambda_{32}}{\lambda_{23}} \beta_{22}^2.
\end{align*}
We can suppose $\lambda_{23}= -1$, $\lambda_{32}= 3$, $\lambda_{21}=-2$, $\lambda_{31}=3 $. Then $\beta_{21}^2 = \beta_{22}^2,~ \beta_{31}^2 = -\frac{1}{2}\beta_{22}^2,~ \beta_{23}^2 = \beta_{22}^2,~ \beta_{32}^3 = -3\beta_{22}^2$.
Under the action of the automorphism
$$\Phi_7=\left(
  \begin{array}{ccc}
    1 & 0 & 0 \\
    -2 & -2 & -1 \\
    3 & 3 & 1 \\
  \end{array}
\right),
$$
$(A_3,\cdot,[\cdot,\cdot,\cdot])$ is isomorphic to
\begin{equation*}
  T_7'(T_4):~ \left\{\begin{split}
   e_2 \cdot e_2 &= \beta_{22}^2 e_1 + \beta_{22}^2 e_2 + \beta_{22}^2 e_3,  \\
   e_2 \cdot e_3 &= -\frac{1}{2}\beta_{22}^2 e_1 -\beta_{22}^2 e_3,  \\
   e_3 \cdot e_3 &= -3\beta_{22}^2 e_2.
  \end{split}\right.
\end{equation*}

\textbf{(I-S2-d)}~ If $\beta_{21}^2 \neq 0$, $\beta_{31}^2 \neq 0$, $\beta_{31}^3 \neq 0$ and suppose $\lambda_{23}= -1$, $\lambda_{32}= 3$, $\lambda_{21}=0$, $\lambda_{31}=0$, then by \eqref{eq:beta3.2} we have
$\beta_{31}^2 = -\frac{3}{2} \beta_{21}^2,~ \beta_{31}^3 = 3\beta_{21}^2,~  \beta_{23}^2 = \beta_{22}^2,~ \beta_{32}^3 = -3\beta_{22}^2$.
Under the action of the automorphism
$$\Phi_8=\left(
  \begin{array}{ccc}
    1 & 0 & 0 \\
    0 & -2 & -1 \\
    0 & 3 & 1 \\
  \end{array}
\right),
$$
$(A_3,\cdot,[\cdot,\cdot,\cdot])$ is isomorphic to
\begin{equation*}
  T_8':~ \left\{\begin{split}
   e_2 \cdot e_2 &= \beta_{21}^2 e_1 + \beta_{22}^2 e_2 + \beta_{22}^2 e_3,  \\
   e_2 \cdot e_3 &= -\frac{3}{2} \beta_{21}^2 e_1 - \beta_{22}^2 e_3,  \\
   e_3 \cdot e_3 &= 3\beta_{21}^2 e_1 - 3\beta_{22}^2 e_2.
  \end{split}\right.
\end{equation*}
%Hence, we obtain the classification (2) in Theorem \ref{classification}.

Secondly, we discuss the \textbf{Case (II)}. By (v) $\sim$ (x) in Lemma \ref{lem:equations}, we have
%\textbf{Case (II)}~ If $\beta_{22}^2=\beta_{33}^2=0$, $\beta_{32}^2=-\beta_{33}^3 \neq 0$, then by (iii)-(viii), we have
\begin{equation}\label{eq:beta4.2}
  \left\{\begin{split}
  0 &= - \lambda_{22} \lambda_{22} \lambda_{32} \beta_{23}^2 + \lambda_{23} \lambda_{23} \lambda_{33} \beta_{32}^3 -(2\lambda_{22}\lambda_{23} \lambda_{33} + \lambda_{23} \lambda_{23} \lambda_{32}) \beta_{33}^3, \\
  \beta_{23}^2 &= \lambda_{22} \lambda_{22} \lambda_{22} \beta_{23}^2 - \lambda_{23} \lambda_{23} \lambda_{23} \beta_{32}^3 + 3\lambda_{22} \lambda_{23}\lambda_{23} \beta_{33}^3, \\
  \beta_{32}^3 &= -\lambda_{32} \lambda_{32} \lambda_{32} \beta_{23}^2 + \lambda_{33} \lambda_{33} \lambda_{33} \beta_{32}^3 - 3\lambda_{32} \lambda_{33} \lambda_{33} \beta_{33}^3, \\
  \beta_{33}^3 &= \lambda_{22}\lambda_{32}\lambda_{32}\beta_{23}^2 - \lambda_{23} \lambda_{33}\lambda_{33} \beta_{32}^3 + (2\lambda_{23} \lambda_{32} \lambda_{33} + \lambda_{22} \lambda_{33} \lambda_{33}) \beta_{33}^3.
  \end{split}\right.
\end{equation}
If $\beta_{23}^2 = 0$, then by \eqref{eq:beta4.2} we have $\lambda_{23} = 0$, which contradicts with $\lambda_{23} \neq 0$. So, $\beta_{23}^2 \neq 0$. Then, we discuss the following two subcases depending on whether $\beta_{32}^3 = 0$ or not.

\textbf{II-S1}:  If $\beta_{32}^3 = 0$, then \eqref{eq:beta4.2} implies
\begin{align*}
  0 &= - \lambda_{22} \lambda_{22} \lambda_{32} \beta_{23}^2 - (2\lambda_{22} \lambda_{23} \lambda_{33}+\lambda_{23}\lambda_{23}\lambda_{32}) \beta_{33}^3, \\
  \beta_{23}^2 &= \lambda_{22} \lambda_{22} \lambda_{22} \beta_{23}^2 + 3\lambda_{22} \lambda_{23}\lambda_{23} \beta_{33}^3, \\
  0 &= -\lambda_{32} \lambda_{32} \lambda_{32} \beta_{23}^2 - 3\lambda_{32} \lambda_{33} \lambda_{33} \beta_{33}^3, \\
  \beta_{33}^3 &= \lambda_{22}\lambda_{32}\lambda_{32}\beta_{23}^2 + (2\lambda_{23} \lambda_{32} \lambda_{33} + \lambda_{22} \lambda_{33} \lambda_{33}) \beta_{33}^3.
\end{align*}
And we get $\lambda_{22} = -\frac{1}{2}$, $\lambda_{33} = -\frac{1}{2}$, $\lambda_{23} \lambda_{32} = -\frac{3}{4}$. Then $\beta_{23}^2 = -\frac{4}{3} \lambda_{23} \lambda_{23} \beta_{33}^3$.
We can suppose $\lambda_{11} = 1$. Then equations (xi) $\sim$ (xiii) in Lemma \ref{lem:equations} can be simplified into
\begin{equation}\label{eq:beta3.3}
  \left\{\begin{split}
  \beta_{21}^2 &= \frac{1}{4} \beta_{21}^2 - \lambda_{23} \beta_{31}^2 + \lambda_{23} \lambda_{23} \beta_{31}^3 + \frac{4}{3} \lambda_{23} \lambda_{23} \lambda_{31} \beta_{33}^3, \\
  \beta_{31}^2 &= -\frac{1}{2} \lambda_{32} \beta_{21}^2 - \frac{1}{2} \beta_{31}^2 -\frac{1}{2} \lambda_{23} \beta_{31}^3+\lambda_{21}\beta_{33}^3,  \\
  \beta_{31}^3 &= \lambda_{32} \lambda_{32} \beta_{21}^2 - \lambda_{32} \beta_{31}^2 + \frac{1}{4} \beta_{31}^3 - \lambda_{31} \beta_{33}^3.
  \end{split}\right.
\end{equation}

Next, we distinguish the following four sub-subcases depending on whether $\beta_{21}^2, \beta_{31}^2, \beta_{31}^3 $ equal to zero or not.

\textbf{(II-S1-a)}~ If $\beta_{21}^2 = 0$, then by \eqref{eq:beta3.3} we have
\begin{align*}
  0 &= - \lambda_{23} \beta_{31}^2 + \lambda_{23} \lambda_{23} \beta_{31}^3 + \frac{4}{3} \lambda_{23} \lambda_{23} \lambda_{31} \beta_{33}^3, \\
  \beta_{31}^2 &= -\frac{1}{2} \beta_{31}^2 - \frac{1}{2} \lambda_{23} \beta_{31}^3 + \lambda_{21} \beta_{33}^3,  \\
  \beta_{31}^3 &= -\lambda_{32} \beta_{31}^2 + \frac{1}{4} \beta_{31}^3 - \lambda_{31} \beta_{33}^3.
\end{align*}
We can suppose $\lambda_{23}=\frac{3}{2}$, $\lambda_{32}= -\frac{1}{2}$, $\lambda_{21}=0$, $\lambda_{31}=0$. Then $\beta_{31}^2 = 0,~ \beta_{31}^3 = 0,~ \beta_{23}^2 = -3 \beta_{33}^3$.
Under the action of the automorphism
$$\Phi_9=\left(
  \begin{array}{ccc}
    1 & 0 & 0 \\
    0 & -\frac{1}{2} & \frac{3}{2} \\
    0 & -\frac{1}{2} & -\frac{1}{2} \\
  \end{array}
\right),
$$
$(A_3,\cdot,[\cdot,\cdot,\cdot])$ is isomorphic to
\begin{equation*}
  T_9':~ \left\{\begin{split}
   e_2 \cdot e_2 &= -3\beta_{33}^3 e_3, \\
   e_2 \cdot e_3 &= -\beta_{33}^3 e_2, \\
   e_3 \cdot e_3 &= \beta_{33}^3 e_3.
  \end{split}\right.
\end{equation*}

\textbf{(II-S1-b)}~ If $\beta_{21}^2 \neq 0$, $\beta_{31}^2 = 0$, then by \eqref{eq:beta3.3} we have
\begin{align*}
  \beta_{21}^2 &= \frac{1}{4} \beta_{21}^2 + \lambda_{23} \lambda_{23} \beta_{31}^3 + \frac{4}{3}\lambda_{23} \lambda_{23} \lambda_{31} \beta_{33}^3, \\
  0 &= -\frac{1}{2} \lambda_{32} \beta_{21}^2 - \frac{1}{2} \beta_{31}^2 -\frac{1}{2} \lambda_{23} \beta_{31}^3 + \lambda_{21} \beta_{33}^3,  \\
  \beta_{31}^3 &= \lambda_{32} \lambda_{32} \beta_{21}^2 + \frac{1}{4} \beta_{31}^3 - \lambda_{31} \beta_{33}^3.
\end{align*}
We can suppose $\lambda_{23}= -\frac{3}{2}$, $\lambda_{32}= \frac{1}{2}$, $\lambda_{21}=0$, $\lambda_{31}=0$. Then $\beta_{21}^2 = 3\beta_{31}^3,~ \beta_{23}^2 = -3 \beta_{33}^3$.
Under the action of the automorphism
$$\Phi_{10}=\left(
  \begin{array}{ccc}
    1 & 0 & 0 \\
    0 & -\frac{1}{2} & -\frac{3}{2} \\
    0 & \frac{1}{2} & -\frac{1}{2} \\
  \end{array}
\right),
$$
$(A_3,\cdot,[\cdot,\cdot,\cdot])$ is isomorphic to
\begin{equation*}
  T_{10}':~ \left\{\begin{split}
   e_2 \cdot e_2 &= 3\beta_{31}^3 e_1 - 3\beta_{33}^3 e_3, \\
   e_2 \cdot e_3 &= -\beta_{33}^3 e_2, \\
   e_3 \cdot e_3 &= \beta_{31}^3 e_1 + \beta_{33}^3 e_3.
  \end{split}\right.
\end{equation*}

\textbf{(II-S1-c)}~ If $\beta_{21}^2 \neq 0$, $\beta_{31}^2 \neq 0$, $\beta_{31}^3 = 0$, then by \eqref{eq:beta3.3}, we have
\begin{align*}
  \beta_{21}^2 &= \frac{1}{4} \beta_{21}^2 - \lambda_{23} \beta_{31}^2 + \frac{4}{3} \lambda_{23} \lambda_{23} \lambda_{31} \beta_{33}^3, \\
  \beta_{31}^2 &= -\frac{1}{2} \lambda_{32} \beta_{21}^2 - \frac{1}{2} \beta_{31}^2 + \lambda_{21} \beta_{33}^3,  \\
  0 &= \lambda_{32} \lambda_{32} \beta_{21}^2 - \lambda_{32} \beta_{31}^2 - \lambda_{31} \beta_{33}^3.
\end{align*}
We can suppose $\lambda_{23}= \frac{3}{2}$, $\lambda_{32}=-\frac{1}{2}$, $\lambda_{21}=2$, $\lambda_{31}=2$. Then $\beta_{21}^2 = 4\beta_{33}^3,~ \beta_{31}^2 = 2\beta_{33}^3,~ \beta_{23}^2 = -3 \beta_{33}^3$.
Under the action of the automorphism
$$\Phi_{11}=\left(
  \begin{array}{ccc}
    1 & 0 & 0 \\
    2 & -\frac{1}{2} & \frac{3}{2} \\
    2 & -\frac{1}{2} & -\frac{1}{2} \\
  \end{array}
\right),
$$
$(A_3,\cdot,[\cdot,\cdot,\cdot])$ is isomorphic to
\begin{equation*}
  T_{11}':~ \left\{\begin{split}
   e_2 \cdot e_2 &= 4\beta_{33}^3 e_1 - 3\beta_{33}^3 e_3, \\
   e_2 \cdot e_3 &= 2\beta_{33}^3 e_1 -\beta_{33}^3 e_2, \\
   e_3 \cdot e_3 &= \beta_{33}^3 e_3.
  \end{split}\right.
\end{equation*}

\textbf{(II-S1-d)}~ If $\beta_{21}^2 \neq 0$, $\beta_{31}^2 \neq 0$, $\beta_{31}^3 \neq 0$ and suppose $\lambda_{23}= -\frac{3}{2}$, $\lambda_{32}=\frac{1}{2}$, $\lambda_{21}=3$, $\lambda_{31}=-1$, then by \eqref{eq:beta3.3} we have $\beta_{21}^2 = 3\beta_{31}^3,~ \beta_{31}^2 = 2\beta_{33}^3,~ \beta_{23}^2 = -3 \beta_{33}^3$.
Under the action of the automorphism
$$\Phi_{12}=\left(
  \begin{array}{ccc}
    1 & 0 & 0 \\
    3 & -\frac{1}{2} & -\frac{3}{2} \\
    -1 & \frac{1}{2} & -\frac{1}{2} \\
  \end{array}
\right),
$$
$(A_3,\cdot,[\cdot,\cdot,\cdot])$ is isomorphic to
\begin{equation*}
  T_{12}':~ \left\{\begin{split}
   e_2 \cdot e_2 &= 3\beta_{31}^3 e_1 - 3\beta_{33}^3 e_3, \\
   e_2 \cdot e_3 &= 2\beta_{33}^3 e_1 -\beta_{33}^3 e_2, \\
   e_3 \cdot e_3 &= \beta_{31}^3 e_1 + \beta_{33}^3 e_3.
  \end{split}\right.
\end{equation*}
%Hence, we obtain the classification (3) in Theorem \ref{classification}.

\textbf{II-S2}: If $\beta_{32}^3 \neq 0$, then by \eqref{eq:beta4.2} we have $\lambda_{22} + \lambda_{33} = -1$. We can suppose $\lambda_{22} = -2$, $\lambda_{33} = 1$ and $\lambda_{11} = 1$. Then we have $\lambda_{23}\lambda_{32} = -3$, $\lambda_{23} \beta_{32}^3 = -3\beta_{33}^3,~ \lambda_{32}\lambda_{32} \beta_{23}^2 = -3\beta_{33}^3$.
Then equations (xi) $\sim$ (xiii) in Lemma \ref{lem:equations} imply
\begin{equation}\label{eq:beta3.4}
  \left\{\begin{split}
  \beta_{21}^2 &= 4\beta_{21}^2 - 4\lambda_{23} \beta_{31}^2 + \lambda_{23} \lambda_{23} \beta_{31}^3 + \frac{3\lambda_{31}}{\lambda_{32}\lambda_{32}} \beta_{33}^3, \\
  \beta_{31}^2 &= -2\lambda_{32} \beta_{21}^2 - 5\beta_{31}^2 + \lambda_{23} \beta_{31}^3 + \lambda_{21} \beta_{33}^3,  \\
  \beta_{31}^3 &= \lambda_{32} \lambda_{32} \beta_{21}^2 + 2\lambda_{32} \beta_{31}^2 + \beta_{31}^3 + (\frac{3\lambda_{21}}{\lambda_{23}} - \lambda_{31}) \beta_{33}^3.
  \end{split}\right.
\end{equation}

Next, we distinguish the following four sub-subcases depending on whether $\beta_{21}^2, \beta_{31}^2, \beta_{31}^3$ equal to zero or not.

\textbf{(II-S2-a)}~ If $\beta_{21}^2 = 0$, then by \eqref{eq:beta3.4} we have
\begin{align*}
  0 &= - 4\lambda_{23} \beta_{31}^2 + \lambda_{23} \lambda_{23} \beta_{31}^3 + \frac{3\lambda_{31}}{\lambda_{32}\lambda_{32}} \beta_{33}^3, \\
  \beta_{31}^2 &= - 5\beta_{31}^2 + \lambda_{23} \beta_{31}^3 + \lambda_{21} \beta_{33}^3,  \\
  \beta_{31}^3 &= 2\lambda_{32} \beta_{31}^2 + \beta_{31}^3 + (\frac{3\lambda_{21}}{\lambda_{23}} - \lambda_{31}) \beta_{33}^3.
\end{align*}
We can suppose $\lambda_{23}= -3$, $\lambda_{32}=1$, $\lambda_{21}=0$, $\lambda_{31}=0$. Then
$\beta_{31}^2 = \beta_{31}^3 = 0,~ \beta_{23}^2 = -3\beta_{33}^3,~ \beta_{32}^3 = \beta_{33}^3$.
Under the action of the automorphism
$$\Phi_{13}=\left(
  \begin{array}{ccc}
    1 & 0 & 0 \\
    0 & -2 & -3 \\
    0 & 1 & 1 \\
  \end{array}
\right),
$$
$(A_3,\cdot,[\cdot,\cdot,\cdot])$ is isomorphic to
\begin{equation*}
  T_{13}':~ \left\{\begin{split}
   e_2 \cdot e_2 &= - 3\beta_{33}^3 e_3, \\
   e_2 \cdot e_3 &= -\beta_{33}^3 e_2, \\
   e_3 \cdot e_3 &= \beta_{33}^3 e_2 + \beta_{33}^3 e_3.
  \end{split}\right.
\end{equation*}

\textbf{(II-S2-b)}~ If $\beta_{21}^2 \neq 0$, $\beta_{31}^2 = 0$, then by \eqref{eq:beta3.4} we have
\begin{align*}
  \beta_{21}^2 &= 4\beta_{21}^2 + \lambda_{23} \lambda_{23} \beta_{31}^3 + \frac{3\lambda_{31}}{\lambda_{32}\lambda_{32}} \beta_{33}^3, \\
  0 &= -2\lambda_{32} \beta_{21}^2 + \lambda_{23} \beta_{31}^3 + \lambda_{21} \beta_{33}^3,  \\
  \beta_{31}^3 &= \lambda_{32} \lambda_{32} \beta_{21}^2 + \beta_{31}^3 + (\frac{3\lambda_{21}}{\lambda_{23}} - \lambda_{31}) \beta_{33}^3.
\end{align*}
We can suppose $\lambda_{23}=-3$, $\lambda_{32}=1$, $\lambda_{21}=1$, $\lambda_{31}=1$. Then
$\beta_{21}^2 = 2\beta_{33}^3,~ \beta_{31}^3 = -\beta_{33}^3,~ \beta_{23}^2 = -3\beta_{33}^3,~ \beta_{32}^3 = \beta_{33}^3$.
Under the action of the automorphism
$$\Phi_{14}=\left(
  \begin{array}{ccc}
    1 & 0 & 0 \\
    1 & -2 & -3 \\
    1 & 1 & 1 \\
  \end{array}
\right),
$$
$(A_3,\cdot,[\cdot,\cdot,\cdot])$ is isomorphic to
\begin{equation*}
  T_{14}':~ \left\{\begin{split}
   e_2 \cdot e_2 &= 2\beta_{33}^3 e_1 - 3\beta_{33}^3 e_3, \\
   e_2 \cdot e_3 &= -\beta_{33}^3 e_2, \\
   e_3 \cdot e_3 &= -\beta_{33}^3 e_1 + \beta_{33}^3 e_2 + \beta_{33}^3 e_3.
  \end{split}\right.
\end{equation*}

\textbf{(II-S2-c)}~ If $\beta_{21}^2 \neq 0$, $\beta_{31}^2 \neq 0$, $\beta_{31}^3 = 0$, then by \eqref{eq:beta3.4} we have
\begin{align*}
  \beta_{21}^2 &= 4\beta_{21}^2 - 4\lambda_{23} \beta_{31}^2 + \frac{3\lambda_{31}}{\lambda_{32}\lambda_{32}} \beta_{33}^3, \\
  \beta_{31}^2 &= -2\lambda_{32} \beta_{21}^2 - 5\beta_{31}^2 + \lambda_{21} \beta_{33}^3,  \\
  0 &= \lambda_{32} \lambda_{32} \beta_{21}^2 + 2\lambda_{32} \beta_{31}^2 + (\frac{3\lambda_{21}}{\lambda_{23}} - \lambda_{31}) \beta_{33}^3.
\end{align*}
We can suppose $\lambda_{23}=-3$, $\lambda_{32}=1$, $\lambda_{21}=0$, $\lambda_{31}= -1$. Then
$ \beta_{21}^2 = -3\beta_{33}^3,~ \beta_{31}^2 = \beta_{33}^3,~ \beta_{23}^2 = -3\beta_{33}^3,~ \beta_{32}^3 = \beta_{33}^3$.
Under the action of the automorphism
$$\Phi_{15}=\left(
  \begin{array}{ccc}
    1 & 0 & 0 \\
    0 & -2 & -3 \\
    -1 & 1 & 1 \\
  \end{array}
\right),
$$
$(A_3,\cdot,[\cdot,\cdot,\cdot])$ is isomorphic to
\begin{equation*}
  T_{15}':~ \left\{\begin{split}
   e_2 \cdot e_2 &= -3\beta_{33}^3 e_1 - 3\beta_{33}^3 e_3, \\
   e_2 \cdot e_3 &= \beta_{33}^3 e_1 - \beta_{33}^3 e_2, \\
   e_3 \cdot e_3 &= \beta_{33}^3 e_2 + \beta_{33}^3 e_3.
  \end{split}\right.
\end{equation*}

\textbf{(II-S2-d)}~ If $\beta_{21}^2 \neq 0$, $\beta_{31}^2 \neq 0$, $\beta_{31}^3 \neq 0$ and suppose $\lambda_{23}=-3$, $\lambda_{32}=1$, $\lambda_{21}= 0$, $\lambda_{31}= 0$, then by \eqref{eq:beta3.4} we have
$ \beta_{21}^2 = -2\beta_{31}^2,~ \beta_{31}^3 = -\frac{2}{3} \beta_{31}^2,~ \beta_{23}^2 = -3\beta_{33}^3,~ \beta_{32}^3 = \beta_{33}^3$.
Under the action of the automorphism
$$\Phi_{16}=\left(
  \begin{array}{ccc}
    1 & 0 & 0 \\
    0 & -2 & -3 \\
    0 & 1 & 1 \\
  \end{array}
\right),
$$
$(A_3,\cdot,[\cdot,\cdot,\cdot])$ is isomorphic to
\begin{equation*}
  T_{16}':~ \left\{\begin{split}
   e_2 \cdot e_2 &= -2\beta_{31}^2 e_1 - 3\beta_{33}^3 e_3, \\
   e_2 \cdot e_3 &= \beta_{31}^2 e_1 - \beta_{33}^3 e_2, \\
   e_3 \cdot e_3 &= -\frac{2}{3} \beta_{31}^2 e_1 + \beta_{33}^3 e_2 + \beta_{33}^3 e_3.
  \end{split}\right.
\end{equation*}

Thirdly, we discuss the \textbf{Case (III)}. By (v) $\sim$ (x) in Lemma \ref{lem:equations}, we have
\begin{equation}\label{eq:beta4.3}
  \left\{\begin{split}
  \beta_{22}^2 =& (\lambda_{22} \lambda_{22} \lambda_{33} +2\lambda_{22} \lambda_{23} \lambda_{32}) \beta_{22}^2 - \lambda_{22} \lambda_{22} \lambda_{32} \beta_{23}^2 - (2\lambda_{22}\lambda_{23}\lambda_{33} + \lambda_{23} \lambda_{23} \lambda_{32}) \beta_{33}^3 \\
  &+ \lambda_{23} \lambda_{23} \lambda_{33} \beta_{32}^3,\\
  \beta_{23}^2 =& -3\lambda_{22}\lambda_{22}\lambda_{23} \beta_{22}^2 + \lambda_{22}\lambda_{22}\lambda_{22}\beta_{23}^2 + 3\lambda_{22} \lambda_{23} \lambda_{23} \beta_{33}^3 -\lambda_{23} \lambda_{23} \lambda_{23} \beta_{32}^3,\\
  \beta_{32}^3 =& 3\lambda_{32} \lambda_{32} \lambda_{33} \beta_{22}^2 - \lambda_{32} \lambda_{32} \lambda_{32} \beta_{23}^2 - 3\lambda_{32} \lambda_{33} \lambda_{33} \beta_{33}^3 +\lambda_{33} \lambda_{33}\lambda_{33} \beta_{32}^3,\\
  \beta_{33}^3 =& -(\lambda_{23} \lambda_{32} \lambda_{32} +2\lambda_{22} \lambda_{32} \lambda_{33}) \beta_{22}^2 + \lambda_{22}\lambda_{32}\lambda_{32} \beta_{23}^2 + (2\lambda_{23} \lambda_{32} \lambda_{33} + \lambda_{22} \lambda_{33} \lambda_{33}) \beta_{33}^3 \\
  &- \lambda_{23} \lambda_{33}\lambda_{33} \beta_{32}^3.
  \end{split}\right.
\end{equation}
We discuss the following two subcases depending on whether $\beta_{23}^2 = 0$ or not.

\textbf{III-S1}:  If $\beta_{23}^2 = 0$, then \eqref{eq:beta4.3} implies
\begin{align*}
   \beta_{22}^2 =& (\lambda_{22} \lambda_{22} \lambda_{33} +2\lambda_{22} \lambda_{23} \lambda_{32}) \beta_{22}^2 - (2\lambda_{22}\lambda_{23}\lambda_{33} + \lambda_{23} \lambda_{23} \lambda_{32}) \beta_{33}^3 + \lambda_{23} \lambda_{23} \lambda_{33} \beta_{32}^3,\\
   0=& -3\lambda_{22}\lambda_{22}\lambda_{23} \beta_{22}^2 + 3\lambda_{22} \lambda_{23}\lambda_{23} \beta_{33}^3 -\lambda_{23} \lambda_{23} \lambda_{23} \beta_{32}^3,\\
   \beta_{32}^3 =& 3\lambda_{32} \lambda_{32} \lambda_{33} \beta_{22}^2 - 3\lambda_{32} \lambda_{33} \lambda_{33} \beta_{33}^3 +\lambda_{33} \lambda_{33}\lambda_{33} \beta_{32}^3,\\
   \beta_{33}^3 =& -(\lambda_{23} \lambda_{32} \lambda_{32} +2\lambda_{22} \lambda_{32} \lambda_{33}) \beta_{22}^2 + (2\lambda_{23} \lambda_{32} \lambda_{33} + \lambda_{22} \lambda_{33} \lambda_{33}) \beta_{33}^3 - \lambda_{23} \lambda_{33} \lambda_{33} \beta_{32}^3.
\end{align*}
Therefore, $\lambda_{22} + \lambda_{33} = -1$. We can suppose $\lambda_{22}=-2$ and $\lambda_{11} = 1$. Then we have $\lambda_{33}=1$, $\lambda_{23} \lambda_{32}=-3$, $\beta_{33}^3 = \lambda_{32} \beta_{22}^2$, $\beta_{32}^3 = \frac{2}{3} \lambda_{32} \lambda_{32} \beta_{22}^2$, and the equations (xi) $\sim$ (xiii) imply
\begin{equation}\label{eq:beta3.5}
  \left\{\begin{split}
  \beta_{21}^2 &= 4\beta_{21}^2 - \lambda_{21} \beta_{22}^2 -4\lambda_{23} \beta_{31}^2 + \lambda_{23} \lambda_{23} \beta_{31}^3,\\
  \beta_{31}^2 &= -2\lambda_{32} \beta_{21}^2 +(\lambda_{31} +\lambda_{21} \lambda_{32}) \beta_{22}^2 -5\beta_{31}^2 +\lambda_{23} \beta_{31}^3, \\
  \beta_{31}^3 &= \lambda_{32} \lambda_{32} \beta_{21}^2 + 2 \lambda_{32} \beta_{31}^2 - (\lambda_{31} \lambda_{32} + \frac{2}{3} \lambda_{21} \lambda_{32} \lambda_{32}) \beta_{22}^2 + \beta_{31}^3.
  \end{split}\right.
\end{equation}

Next, we distinguish the following four sub-subcases depending on whether $\beta_{21}^2, \beta_{31}^2, \beta_{31}^3 $ equal to zero or not.

\textbf{(III-S1-a)}~ If $\beta_{21}^2 = 0$, then by \eqref{eq:beta3.5} we have
\begin{align*}
  0 &= - \lambda_{21} \beta_{22}^2 -4\lambda_{23} \beta_{31}^2 + \lambda_{23} \lambda_{23} \beta_{31}^3, \\
  \beta_{31}^2 &= (\lambda_{31}+\lambda_{21}\lambda_{32}) \beta_{22}^2 -5\beta_{31}^2 +\lambda_{23} \beta_{31}^3, \\
  \beta_{31}^3 &= 2 \lambda_{32} \beta_{31}^2 - (\lambda_{31} \lambda_{32} + \frac{2}{3} \lambda_{21} \lambda_{32} \lambda_{32}) \beta_{22}^2 + \beta_{31}^3.
\end{align*}
We can suppose $\lambda_{23}=-1$, $\lambda_{32}= 3$, $\lambda_{21}=2$, $\lambda_{31}=-2$. Then $\beta_{31}^2 = \beta_{22}^2,~ \beta_{31}^3 = -2\beta_{22}^2$.
Under the action of the automorphism
$$\Phi_{17}=\left(
  \begin{array}{ccc}
    1 & 0 & 0 \\
    2 & -2 & -1 \\
    -2 & 3 & 1 \\
  \end{array}
\right),
$$
$(A_3,\cdot,[\cdot,\cdot,\cdot])$ is isomorphic to
\begin{equation*}
  T_{17}'(T_5):~ \left\{\begin{split}
   e_2 \cdot e_2 &= \beta_{22}^2 e_2, \\
   e_2 \cdot e_3 &= \beta_{22}^2 e_1 -3 \beta_{22}^2 e_2 -\beta_{22}^2 e_3, \\
   e_3 \cdot e_3 &= -2\beta_{22}^2 e_1 + 6\beta_{22}^2 e_2 + 3\beta_{22}^2 e_3.
  \end{split}\right.
\end{equation*}

\textbf{(III-S1-b)}~ If $\beta_{21}^2 \neq 0$, $\beta_{31}^2 = 0$, then by \eqref{eq:beta3.5} we have
\begin{align*}
  \beta_{21}^2 &= 4\beta_{21}^2 - \lambda_{21} \beta_{22}^2 + \lambda_{23} \lambda_{23} \beta_{31}^3,\\
  0 &= -2\lambda_{32} \beta_{21}^2 +(\lambda_{31} +\lambda_{21} \lambda_{32}) \beta_{22}^2 +\lambda_{23} \beta_{31}^3, \\
  \beta_{31}^3 &= \lambda_{32} \lambda_{32} \beta_{21}^2 - (\lambda_{31} \lambda_{32} + \frac{2}{3} \lambda_{21} \lambda_{32} \lambda_{32}) \beta_{22}^2 + \beta_{31}^3.
\end{align*}
We can suppose $\lambda_{23}= -1$, $\lambda_{32}= 3$, $\lambda_{21}=-3$, $\lambda_{31}=6$. Then $\beta_{21}^2 = 0,~ \beta_{31}^3 = -3\beta_{22}^2$.
Under the action of the automorphism
$$\Phi_{18}=\left(
  \begin{array}{ccc}
    1 & 0 & 0 \\
    -3 & -2 & -1 \\
    6 & 3 & 1 \\
  \end{array}
\right),
$$
$(A_3,\cdot,[\cdot,\cdot,\cdot])$ is isomorphic to
\begin{equation*}
  T_{18}':~ \left\{\begin{split}
   e_2 \cdot e_2 &= \beta_{22}^2 e_2,  \\
   e_2 \cdot e_3 &= -3\beta_{22}^2 e_2 -\beta_{22}^2 e_3, \\
   e_3 \cdot e_3 &= -3\beta_{22}^2 e_1 + 6\beta_{22}^2 e_2 + 3\beta_{22}^2 e_3.
  \end{split}\right.
\end{equation*}

\textbf{(III-S1-c)}~ If $\beta_{21}^2 \neq 0$, $\beta_{31}^2 \neq 0$, $\beta_{31}^3 = 0$, then by \eqref{eq:beta3.5}, we have
\begin{align*}
  \beta_{21}^2 &= 4\beta_{21}^2 - \lambda_{21} \beta_{22}^2 -4\lambda_{23} \beta_{31}^2,\\
  \beta_{31}^2 &= -2\lambda_{32} \beta_{21}^2 +(\lambda_{31} +\lambda_{21} \lambda_{32}) \beta_{22}^2 -5\beta_{31}^2, \\
  0 &= \lambda_{32} \lambda_{32} \beta_{21}^2 + 2 \lambda_{32} \beta_{31}^2 - (\lambda_{31} \lambda_{32} + \frac{2}{3} \lambda_{21} \lambda_{32} \lambda_{32}) \beta_{22}^2.
\end{align*}
We can suppose $\lambda_{23}= -1$, $\lambda_{32}=3$, $\lambda_{21}=-1$, $\lambda_{31}=3$. Then $\beta_{21}^2 = \beta_{22}^2,~ \beta_{31}^2 = -\beta_{22}^2$.
Under the action of the automorphism
$$\Phi_{19}=\left(
  \begin{array}{ccc}
    1 & 0 & 0 \\
    -1 & -2 & -1 \\
    3 & 3 & 1 \\
  \end{array}
\right),
$$
$(A_3,\cdot,[\cdot,\cdot,\cdot])$ is isomorphic to
\begin{equation*}
  T_{19}'(T_6):~ \left\{\begin{split}
   e_2 \cdot e_2 &= \beta_{22}^2 e_1 + \beta_{22}^2 e_2, \\
   e_2 \cdot e_3 &= -\beta_{22}^2 e_1 -3\beta_{22}^2 e_2 -\beta_{22}^2 e_3, \\
   e_3 \cdot e_3 &= 6\beta_{22}^2 e_2 + 3\beta_{22}^2 e_3.
  \end{split}\right.
\end{equation*}

\textbf{(III-S1-d)}~ If $\beta_{21}^2 \neq 0$, $\beta_{31}^2 \neq 0$, $\beta_{31}^3 \neq 0$ and suppose $\lambda_{23}= -1$, $\lambda_{32}=3$, $\lambda_{21}=0$, $\lambda_{31}=0$, then by \eqref{eq:beta3.5} we have $\beta_{31}^2 = -\frac{3}{2}\beta_{21}^2,~ \beta_{31}^3 = 3\beta_{21}^2$.
Under the action of the automorphism
$$\Phi_{20}=\left(
  \begin{array}{ccc}
    1 & 0 & 0 \\
    0 & -2 & -1 \\
    0 & 3 & 1 \\
  \end{array}
\right),
$$
$(A_3,\cdot,[\cdot,\cdot,\cdot])$ is isomorphic to
\begin{equation*}
  T_{20}':~ \left\{\begin{split}
   e_2 \cdot e_2 &= \beta_{21}^2 e_1 + \beta_{22}^2 e_2, \\
   e_2 \cdot e_3 &= -\frac{3}{2} \beta_{21}^2 e_1 -3\beta_{22}^2 e_2 -\beta_{22}^2 e_3, \\
   e_3 \cdot e_3 &= 3\beta_{21}^2 e_1 + 6\beta_{22}^2 e_2 + 3\beta_{22}^2 e_3.
  \end{split}\right.
\end{equation*}
%Hence, we obtain the classification (5) in Theorem \ref{classification}.

\textbf{III-S2}: If $\beta_{23}^2 \neq 0$, then we have \eqref{eq:beta4.3}.
If $\beta_{32}^3 = 0$, then we get $\lambda_{22}=0$ or $\lambda_{32}=0$, which contradicts with $\lambda_{22} \neq 0$ or $\lambda_{32} \neq 0$. So, $\beta_{32}^3 \neq 0$. By \eqref{eq:beta4.3}, we have $\lambda_{22} + \lambda_{33} = -1$. We can suppose $\lambda_{22}=-2$. Then  $\lambda_{33}=1$, $\lambda_{23} \lambda_{32}=-3$, and \eqref{eq:beta4.3} implies
\begin{align*}
   \beta_{22}^2 =& 16\beta_{22}^2 - 4\lambda_{32} \beta_{23}^2 + 7\lambda_{23} \beta_{33}^3 + \lambda_{23} \lambda_{23} \beta_{32}^3,\\
   \beta_{23}^2 =& -12\beta_{22}^2 -8\beta_{23}^2 - 6\lambda_{23}\lambda_{23} \beta_{33}^3 -\lambda_{23}\lambda_{23}\lambda_{23}\beta_{32}^3,\\
   \beta_{32}^3 =& 3\lambda_{32} \lambda_{32} \beta_{22}^2 - \lambda_{32} \lambda_{32} \lambda_{32} \beta_{23}^2 - 3\lambda_{32} \beta_{33}^3 +\beta_{32}^3,\\
   \beta_{33}^3 =& 7\lambda_{32} \beta_{22}^2 -2\lambda_{32}\lambda_{32} \beta_{23}^2 - 8\beta_{33}^3 - \lambda_{23} \beta_{32}^3.
\end{align*}
Therefore, we have $\beta_{33}^3 = \lambda_{32} \beta_{22}^2 -\frac{1}{3} \lambda_{32} \lambda_{32} \beta_{23}^2$, $\beta_{32}^3 = \frac{2}{3} \lambda_{32} \lambda_{32} \beta_{22}^2 -\frac{1}{3} \lambda_{32} \lambda_{32} \lambda_{32} \beta_{23}^2$.
We can suppose $\lambda_{11} =1 $. Then equations (xi) $\sim$ (xiii) imply
\begin{equation}\label{eq:beta3.6}
  \left\{\begin{split}
  \beta_{21}^2 &= 4\beta_{21}^2 - \lambda_{21} \beta_{22}^2 -\lambda_{31} \beta_{23}^2 -4\lambda_{23} \beta_{31}^2 + \lambda_{23} \lambda_{23} \beta_{31}^3,\\
  \beta_{31}^2 &= -2\lambda_{32} \beta_{21}^2 +(\lambda_{31} +\lambda_{21} \lambda_{32}) \beta_{22}^2 -5\beta_{31}^2 +\lambda_{23} \beta_{31}^3, \\
  \beta_{31}^3 &= \lambda_{32} \lambda_{32} \beta_{21}^2 + 2 \lambda_{32} \beta_{31}^2 - (\lambda_{31} \lambda_{32} + \frac{2}{3} \lambda_{21} \lambda_{32} \lambda_{32}) \beta_{22}^2 + \beta_{31}^3.
  \end{split}\right.
\end{equation}

Next, we distinguish the following four sub-subcases depending on whether $\beta_{21}^2, \beta_{31}^2, \beta_{31}^3 $ equal to zero or not.

\textbf{(III-S2-a)}~ If $\beta_{21}^2 = 0$, then by \eqref{eq:beta3.6} we have
\begin{align*}
  0 &= - \lambda_{21} \beta_{22}^2 -\lambda_{31} \beta_{23}^2 -4\lambda_{23} \beta_{31}^2 + \lambda_{23} \lambda_{23} \beta_{31}^3, \\
  \beta_{31}^2 &= (\lambda_{31}+\lambda_{21}\lambda_{32}) \beta_{22}^2 -5\beta_{31}^2 +\lambda_{23} \beta_{31}^3, \\
  \beta_{31}^3 &= 2 \lambda_{32} \beta_{31}^2 - (\lambda_{31} \lambda_{32} + \frac{2}{3} \lambda_{21} \lambda_{32} \lambda_{32}) \beta_{22}^2 + \beta_{31}^3.
\end{align*}
We can suppose $\lambda_{23}=-1$, $\lambda_{32}= 3$, $\lambda_{21}=2$, $\lambda_{31}=-2$. Then $\beta_{31}^2 = \beta_{22}^2 -2\beta_{23}^2,~ \beta_{31}^3 = -2\beta_{22}^2 +6\beta_{23}^2$.
Under the action of the automorphism
$$\Phi_{21}=\left(
  \begin{array}{ccc}
    1 & 0 & 0 \\
    2 & -2 & -1 \\
    -2 & 3 & 1 \\
  \end{array}
\right),
$$
$(A_3,\cdot,[\cdot,\cdot,\cdot])$ is isomorphic to
\begin{equation*}
  T_{21}':~ \left\{\begin{split}
   e_2 \cdot e_2 &= \beta_{22}^2 e_2 + \beta_{23}^2 e_3, \\
   e_2 \cdot e_3 &= (\beta_{22}^2 -2\beta_{23}^2) e_1 + (3\beta_{23}^2 -3 \beta_{22}^2) e_2 -\beta_{22}^2 e_3, \\
   e_3 \cdot e_3 &= (6\beta_{23}^2 -2\beta_{22}^2) e_1 + (6\beta_{22}^2 -9\beta_{23}^2) e_2 + (3\beta_{22}^2 -3\beta_{23}^2) e_3.
  \end{split}\right.
\end{equation*}

\textbf{(III-S2-b)}~ If $\beta_{21}^2 \neq 0$, $\beta_{31}^2 = 0$, then by \eqref{eq:beta3.6} we have
\begin{align*}
  \beta_{21}^2 &= 4\beta_{21}^2 - \lambda_{21} \beta_{22}^2 -\lambda_{31} \beta_{23}^2 + \lambda_{23} \lambda_{23} \beta_{31}^3,\\
  0 &= -2\lambda_{32} \beta_{21}^2 +(\lambda_{31} +\lambda_{21} \lambda_{32}) \beta_{22}^2 +\lambda_{23} \beta_{31}^3, \\
  \beta_{31}^3 &= \lambda_{32} \lambda_{32} \beta_{21}^2 - (\lambda_{31} \lambda_{32} + \frac{2}{3} \lambda_{21} \lambda_{32} \lambda_{32}) \beta_{22}^2 + \beta_{31}^3.
\end{align*}
We can suppose $\lambda_{23}= -1$, $\lambda_{32}= 3$, $\lambda_{21}=-3$, $\lambda_{31}=6$. Then $\beta_{21}^2 = \beta_{23}^2,~ \beta_{31}^3 = -3\beta_{22}^2 +3\beta_{23}^2$.
Under the action of the automorphism
$$\Phi_{22}=\left(
  \begin{array}{ccc}
    1 & 0 & 0 \\
    -3 & -2 & -1 \\
    6 & 3 & 1 \\
  \end{array}
\right),
$$
$(A_3,\cdot,[\cdot,\cdot,\cdot])$ is isomorphic to
\begin{equation*}
  T_{22}'(T_7):~ \left\{\begin{split}
   e_2 \cdot e_2 &= \beta_{23}^2 e_1 +\beta_{22}^2 e_2 + \beta_{23}^2 e_3,  \\
   e_2 \cdot e_3 &= (3\beta_{23}^2 -3\beta_{22}^2) e_2 -\beta_{22}^2 e_3, \\
   e_3 \cdot e_3 &= (3\beta_{23}^2 -3\beta_{22}^2) e_1 + (6\beta_{22}^2 -9\beta_{23}^2) e_2 + (3\beta_{22}^2 -3\beta_{23}^2) e_3.
  \end{split}\right.
\end{equation*}

\textbf{(III-S2-c)}~ If $\beta_{21}^2 \neq 0$, $\beta_{31}^2 \neq 0$, $\beta_{31}^3 = 0$, then by \eqref{eq:beta3.6} we have
\begin{align*}
  \beta_{21}^2 &= 4\beta_{21}^2 - \lambda_{21} \beta_{22}^2 -\lambda_{31} \beta_{23}^2 -4\lambda_{23} \beta_{31}^2,\\
  \beta_{31}^2 &= -2\lambda_{32} \beta_{21}^2 +(\lambda_{31} +\lambda_{21} \lambda_{32}) \beta_{22}^2 -5\beta_{31}^2, \\
  0 &= \lambda_{32} \lambda_{32} \beta_{21}^2 + 2 \lambda_{32} \beta_{31}^2 - (\lambda_{31} \lambda_{32} + \frac{2}{3} \lambda_{21} \lambda_{32} \lambda_{32}) \beta_{22}^2.
\end{align*}
We can suppose $\lambda_{23}= -1$, $\lambda_{32}=3$, $\lambda_{21}=-1$, $\lambda_{31}=3$. Then $\beta_{21}^2 = \beta_{22}^2 -\beta_{23}^2,~ \beta_{31}^2 = -\beta_{22}^2 +\frac{3}{2}\beta_{23}^2$.
Under the action of the automorphism
$$\Phi_{23}=\left(
  \begin{array}{ccc}
    1 & 0 & 0 \\
    -1 & -2 & -1 \\
    3 & 3 & 1 \\
  \end{array}
\right),
$$
$(A_3,\cdot,[\cdot,\cdot,\cdot])$ is isomorphic to
\begin{equation*}
  T_{23}'(T_8):~ \left\{\begin{split}
   e_2 \cdot e_2 &= (\beta_{22}^2 -\beta_{23}^2) e_1 + \beta_{22}^2 e_2 +\beta_{23}^2 e_3, \\
   e_2 \cdot e_3 &= (-\beta_{22}^2 +\frac{3}{2}\beta_{23}^2) e_1 +(3\beta_{23}^2 -3\beta_{22}^2) e_2 -\beta_{22}^2 e_3, \\
   e_3 \cdot e_3 &= (6\beta_{22}^2 -9\beta_{23}^2) e_2 + (3\beta_{22}^2 -3\beta_{23}^2) e_3.
  \end{split}\right.
\end{equation*}

\textbf{(III-S2-d)}~ If $\beta_{21}^2 \neq 0$, $\beta_{31}^2 \neq 0$, $\beta_{31}^3 \neq 0$ and suppose $\lambda_{23}= -1$, $\lambda_{32}=3$, $\lambda_{21}=0$, $\lambda_{31}=0$, then by \eqref{eq:beta3.6} we have $\beta_{31}^2 = -\frac{3}{2}\beta_{21}^2,~ \beta_{31}^3 = 3\beta_{21}^2$.
Under the action of the automorphism
$$\Phi_{24}=\left(
  \begin{array}{ccc}
    1 & 0 & 0 \\
    0 & -2 & -1 \\
    0 & 3 & 1 \\
  \end{array}
\right),
$$
$(A_3,\cdot,[\cdot,\cdot,\cdot])$ is isomorphic to
\begin{equation*}
  T_{24}'(T_9):~ \left\{\begin{split}
   e_2 \cdot e_2 &= \beta_{21}^2 e_1 + \beta_{22}^2 e_2 +\beta_{23}^2 e_3, \\
   e_2 \cdot e_3 &= -\frac{3}{2} \beta_{21}^2 e_1 + (3\beta_{23}^2 -3\beta_{22}^2) e_2 -\beta_{22}^2 e_3, \\
   e_3 \cdot e_3 &= 3\beta_{21}^2 e_1 + (6\beta_{22}^2 -9\beta_{23}^2) e_2 + (3\beta_{22}^2 -3\beta_{23}^2) e_3. \\
  \end{split}\right.
\end{equation*}
%Hence, we obtain the classification (6) in Theorem \ref{classification}.

Finally, we discuss the \textbf{Case (IV)}. By (v) $\sim$ (x) in Lemma \ref{lem:equations}, we have
\begin{align*}
  0 =& - \lambda_{22} \lambda_{22} \lambda_{32} \beta_{23}^2 + \lambda_{23} \lambda_{23} \lambda_{33} \beta_{32}^3,\\
  \beta_{23}^2 =& \lambda_{22}\lambda_{22}\lambda_{22}\beta_{23}^2 -\lambda_{23} \lambda_{23} \lambda_{23} \beta_{32}^3,\\
  \beta_{32}^3 =& -\lambda_{32} \lambda_{32} \lambda_{32} \beta_{23}^2 +\lambda_{33} \lambda_{33}\lambda_{33} \beta_{32}^3,\\
  0 =& \lambda_{22}\lambda_{32}\lambda_{32} \beta_{23}^2 - \lambda_{23} \lambda_{33}\lambda_{33} \beta_{32}^3.
\end{align*}
Therefore, $\beta_{23}^2=\beta_{32}^3=0$. We can suppose $\lambda_{11}=1$. Then by (xi) $\sim$ (xiii) in Lemma \ref{lem:equations}, we get
\begin{equation}\label{eq:beta3.7}
  \left\{\begin{split}
  \beta_{21}^2 &= \lambda_{22} \lambda_{22}\beta_{21}^2 + 2\lambda_{22} \lambda_{23} \beta_{31}^2 + \lambda_{23} \lambda_{23} \beta_{31}^3, \\
  \beta_{31}^2 &= \lambda_{22} \lambda_{32} \beta_{21}^2 +(\lambda_{22}\lambda_{33}+ \lambda_{23} \lambda_{32})\beta_{31}^2+\lambda_{23}\lambda_{33} \beta_{31}^3,  \\
  \beta_{31}^3 &= \lambda_{32} \lambda_{32} \beta_{21}^2 + 2 \lambda_{32} \lambda_{33} \beta_{31}^2 + \lambda_{33} \lambda_{33} \beta_{31}^3.
  \end{split}\right.
\end{equation}

Next, we distinguish the following four subcases depending on whether $\beta_{21}^2, \beta_{31}^2, \beta_{31}^3$ equal to zero or not.

\textbf{IV-S1}: If $\beta_{21}^2=0$, then by \eqref{eq:beta3.7} we have
\begin{align*}
  0 &= 2\lambda_{22} \lambda_{23} \beta_{31}^2 + \lambda_{23} \lambda_{23} \beta_{31}^3, \\
  \beta_{31}^2 &= (\lambda_{22}\lambda_{33}+ \lambda_{23} \lambda_{32})\beta_{31}^2 +\lambda_{23}\lambda_{33} \beta_{31}^3,  \\
  \beta_{31}^3 &= 2\lambda_{32} \lambda_{33} \beta_{31}^2 + \lambda_{33} \lambda_{33} \beta_{31}^3.
\end{align*}
which leads to  a contradiction.

\textbf{IV-S2}: If $\beta_{21}^2 \neq 0$, $\beta_{31}^2=0$, then by \eqref{eq:beta3.7} we have
\begin{align*}
  \beta_{21}^2 &= \lambda_{22} \lambda_{22}\beta_{21}^2 + \lambda_{23} \lambda_{23} \beta_{31}^3, \\
  0 &= \lambda_{22}\lambda_{32} \beta_{21}^2 +\lambda_{23}\lambda_{33} \beta_{31}^3, \\
  \beta_{31}^3 &= \lambda_{32} \lambda_{32} \beta_{21}^2 + \lambda_{33} \lambda_{33} \beta_{31}^3.
\end{align*}
Therefore, $\lambda_{22}=\lambda_{33}$. We can suppose $\lambda_{22}=\lambda_{33} =2$, $\lambda_{23}=1$, $\lambda_{32}=3$, $\lambda_{21}=\lambda_{31} =0$. Then $\beta_{31}^3= -3\beta_{21}^2$. Under the action of the automorphism
$$\Phi_{25}=\left(
  \begin{array}{ccc}
    1 & 0 & 0 \\
    0 & 2 & 1 \\
    0 & 3 & 2 \\
  \end{array}
\right),
$$
$(A_3,\cdot,[\cdot,\cdot,\cdot])$ is isomorphic to
\begin{equation*}
  T_{25}'(T_{10}):~ \left\{\begin{split}
   e_2 \cdot e_2 &= \beta_{21}^2 e_1,  \\
   e_3 \cdot e_3 &= -3\beta_{21}^2 e_1.
  \end{split}\right.
\end{equation*}

\textbf{IV-S3}: If $\beta_{21}^2 \neq 0$, $\beta_{31}^2 \neq 0$, $\beta_{31}^3 =0$, then by \eqref{eq:beta3.7} we have
\begin{align*}
  \beta_{21}^2 &= \lambda_{22} \lambda_{22}\beta_{21}^2 + 2\lambda_{22} \lambda_{23} \beta_{31}^2, \\
  \beta_{31}^2 &= \lambda_{22} \lambda_{32} \beta_{21}^2 +(\lambda_{22}\lambda_{33}+ \lambda_{23} \lambda_{32})\beta_{31}^2,  \\
  0 &= \lambda_{32} \lambda_{32} \beta_{21}^2 + 2 \lambda_{32} \lambda_{33} \beta_{31}^2.
\end{align*}
Therefore, $\lambda_{32}=0$, which contradicts with $\lambda_{32} \neq 0$.

\textbf{IV-S4}: If $\beta_{21}^2 \neq 0$, $\beta_{31}^2 \neq 0$, $\beta_{31}^3 \neq 0$ and suppose $\lambda_{22}=1$, $\lambda_{33} =3$, $\lambda_{23}=1$, $\lambda_{32}=2$, $\lambda_{21}=\lambda_{31} =0$, then by \eqref{eq:beta3.7} we have $\beta_{31}^2= \beta_{21}^2$, $\beta_{31}^3= -2\beta_{21}^2$. Under the action of the automorphism
$$\Phi_{26}=\left(
  \begin{array}{ccc}
    1 & 0 & 0 \\
    0 & 1 & 1 \\
    0 & 2 & 3 \\
  \end{array}
\right),
$$
$(A_3,\cdot,[\cdot,\cdot,\cdot])$ is isomorphic to
\begin{equation*}
  T_{26}':~ \left\{\begin{split}
   e_2 \cdot e_2 &= \beta_{21}^2 e_1,  \\
   e_2 \cdot e_3 &= \beta_{21}^2 e_1,  \\
   e_3 \cdot e_3 &= -2\beta_{21}^2 e_1.
  \end{split}\right.
\end{equation*}
%Hence, we obtain the classification (7) in Theorem \ref{classification}.

\subsection{Isomorphic cases of $T_i'$}
One checks that all the above prescriptions for $T_i', 1\leq i \leq 26$ are transposed Poisson $3$-Lie algebras.

\textbf{(1)} If the isomorphism between $T_1'(T_1)$ and $T_9'$ is given by
\begin{equation*}
\left(
  \begin{array}{ccc}
    1 & 0 & 0 \\
    0 & -\frac{1}{2} & -\frac{1}{2} \\
    0 &  \frac{1}{2} &  -\frac{3}{2} \\
  \end{array}
\right)\in \Phi,
\end{equation*}
then $T_1'(T_1)\cong T_9'$.

\textbf{(2)}
If the isomorphisms between $T_2'(T_2)$ and $T_3', T_4',  T_{10}', T_{11}', T_{12}'$ are given by the following matrices respectively:
\begin{equation*}
\left(
  \begin{array}{ccc}
    -\frac{\alpha}{\theta} & 0 & 0 \\
    1 & -\frac{1}{2} & -\frac{1}{2} \\
    -1 &  \frac{3}{2} &  -\frac{1}{2} \\
  \end{array}
\right),\quad \left(
  \begin{array}{ccc}
    \frac{3\theta+\alpha}{3\theta} & 0 & 0 \\
    \frac{2}{3} & -\frac{1}{2} & -\frac{1}{2} \\
    0 &  \frac{3}{2} &  -\frac{1}{2} \\
  \end{array}
\right),\quad \left(
  \begin{array}{ccc}
    1 & 0 & 0 \\
    0 & \frac{1}{2} & -\frac{1}{2} \\
    0 &  \frac{1}{2} &  \frac{3}{2} \\
  \end{array}
\right),
\end{equation*}
\begin{equation*}
 \left(
  \begin{array}{ccc}
    \frac{2\alpha}{3\theta} & 0 & 0 \\
    -\frac{2}{3} & \frac{1}{2} & -\frac{1}{2} \\
    -2 &  \frac{1}{2} &  \frac{3}{2} \\
  \end{array}
\right),\quad\left(
  \begin{array}{ccc}
    1 & 0 & 0 \\
    -1 & \frac{1}{2} & -\frac{1}{2} \\
    \frac{1}{2} &  \frac{1}{2} &  \frac{3}{2} \\
  \end{array}
\right)\in \Phi,\quad
\end{equation*}
where for all $\alpha,\theta \in \mathbb{C},\alpha\neq-3\theta, \alpha,\theta\neq 0$, then $T_2'(T_2)\cong T_3'\cong T_4'\cong T_{10}'\cong T_{11}'\cong T_{12}'$.

\textbf{(3)}
If the isomorphism between $T_5'(T_3)$ and $T_6'$ is given by
\begin{equation*}
\left(
  \begin{array}{ccc}
    1 & 0 & 0 \\
    1 & 1 & 0 \\
    0 & 0 & 1 \\
  \end{array}
\right) \in \Phi,
\end{equation*}
 then $T_5'(T_3)\cong T_6'$. If the isomorphisms between $T_5'(T_3)$ and $T_{13}'$, $T_{15}'$ are given by the following matrices respectively:
\begin{equation*}
 \left(
  \begin{array}{ccc}
    1 & 0 & 0 \\
    0 & -\frac{\lambda_{33}}{\lambda_{33}^2-3} & -\frac{\lambda_{33}^2+6\lambda_{33}+6}{\lambda_{33}^2 +3\lambda_{33}+3} \\
    0 & \frac{3\lambda_{33}+3}{\lambda_{33}^2-3} & \lambda_{33} \\
  \end{array}
\right),\quad
\left(
  \begin{array}{ccc}
    -\frac{2}{3} & 0 & 0 \\
    \frac{2\lambda_{33}^2+8\lambda_{33}+7}{\lambda_{33}^2+3\lambda_{33}+3} & -\frac{\lambda_{33}}{\lambda_{33}^2-3} & -\frac{\lambda_{33}^2+6\lambda_{33}+6}{\lambda_{33}^2 +3\lambda_{33}+3} \\
    -\frac{11\lambda_{33}^2 +36\lambda_{33}+24}{3\lambda_{33}^2+9\lambda_{33}+9} & \frac{3\lambda_{33}+3}{\lambda_{33}^2-3} & \lambda_{33} \\
  \end{array}
\right)
\in \Phi,
\end{equation*}
  where $\lambda_{33}=\sqrt[3]{\frac{9}{2}+\frac{3\sqrt{3}}{2}\sqrt{-1}} +\sqrt[3]{\frac{9}{2}-\frac{3\sqrt{3}}{2}\sqrt{-1}}$, then $T_5'(T_3)\cong T_{13}'\cong T_{15}'$.

\textbf{(4)}
If the isomorphism between $T_7'(T_4)$ and $T_8'$ is given by
\begin{equation*}
\left(
  \begin{array}{ccc}
    \frac{\theta}{\alpha} & 0 & 0 \\
    -\frac{\theta}{\alpha} & 1 & 0 \\
    \frac{\theta}{\alpha} & 0 & 1 \\
  \end{array}
\right)\in \Phi,\quad \forall~ \theta,\alpha\in \mathbb{C}, \theta,\alpha\neq 0,
\end{equation*}
then $T_7'(T_4)\cong T_8'$. If the isomorphism between $T_7'(T_4)$ and $T_{14}', T_{16}'$ are given by
\begin{equation*}
 \left(
  \begin{array}{ccc}
    -\frac{2}{3} & 0 & 0 \\
    \frac{2\lambda_{33}^2+8\lambda_{33}+7}{\lambda_{33}^2+3\lambda_{33}+3} & -\frac{\lambda_{33}}{\lambda_{33}^2-3} & -\frac{\lambda_{33}^2+6\lambda_{33}+6}{\lambda_{33}^2 +3\lambda_{33}+3} \\
    -\frac{11\lambda_{33}^2 +36\lambda_{33}+24}{3\lambda_{33}^2+9\lambda_{33}+9} & \frac{3\lambda_{33}+3}{\lambda_{33}^2-3} & \lambda_{33} \\
  \end{array}
\right),\quad
\left(
  \begin{array}{ccc}
    -\frac{2\theta}{3\alpha} & 0 & 0 \\
    \frac{2\theta}{3\alpha} & -\frac{\lambda_{33}}{\lambda_{33}^2-3} & -\frac{\lambda_{33}^2+6\lambda_{33}+6}{\lambda_{33}^2 +3\lambda_{33}+3} \\
    -\frac{2\theta}{3\alpha} &\frac{3\lambda_{33}+3}{\lambda_{33}^2-3} & \lambda_{33}\\
  \end{array}
\right)
\in \Phi,
\end{equation*}
  where $\lambda_{33}=\sqrt[3]{\frac{9}{2}+\frac{3\sqrt{3}}{2}\sqrt{-1}} +\sqrt[3]{\frac{9}{2}-\frac{3\sqrt{3}}{2}\sqrt{-1}}$, $\alpha,\theta\in \mathbb{C}, \theta,\alpha\neq 0$, then $T_7'(T_4)\cong T_{14}'\cong T_{16}'$.

\textbf{(5)}
If the isomorphisms between $T_{17}'(T_5)$ and $T_{18}', T_{20}'$, $T_{21}'$ are given by the following matrices respectively:
\begin{equation*}
\left(
  \begin{array}{ccc}
    -3 & 0 & 0 \\
    -3 & -2 & -1 \\
    3 &  3 & 1 \\
  \end{array}
\right),\quad \left(
  %\begin{array}{ccc}
%    0 & 0 & 0 \\
%    0 & -2 & -1 \\
%    1 & 3 & 1 \\
%  \end{array}
%\right),\quad \left(
  \begin{array}{ccc}
    \frac{3\theta}{2\alpha} & 0 & 0 \\
    \frac{\theta}{\alpha} & 1 & 1 \\
    0 &  -3 &  -2 \\
  \end{array}
\right),\quad \left(
  \begin{array}{ccc}
    1 & 0 & 0 \\
    -\frac{6\lambda_{22}+2}{(2\lambda_{22}+1)^2} & \lambda_{22} & \frac{\lambda_{22}^2}{2\lambda_{22}+1} \\
    \frac{2\lambda_{22}^2+12\lambda_{22}+4}{(2\lambda_{22}+1)^2} &  -\frac{21\lambda_{22}^2+33\lambda_{22}+9}{8\lambda_{22}^2+17\lambda_{22}+5} &  \frac{\lambda_{22}-\lambda_{22}^2}{2\lambda_{22}+1} \\
  \end{array}
\right)\in \Phi,
\end{equation*}
where $\lambda_{22}=\sqrt[3]{\frac{3}{2}+\frac{\sqrt{5}}{2}} +\sqrt[3]{\frac{3}{2}-\frac{\sqrt{5}}{2}}$, $\alpha,\theta \in \mathbb{C}, \alpha,\theta\neq 0$, then $T_{17}'(T_5)\cong T_{18}'\cong T_{20}'\cong T_{21}'$.

\textbf{(6)} If the isomorphisms between $T_{25}'(T_{10})$ and $T_{26}'$ is given by
\begin{equation*}
\left(
  \begin{array}{ccc}
    1 & 0 & 0 \\
    1 & 1 & 0 \\
    0 & -1 &  1 \\
  \end{array}
\right)\in \Phi,
\end{equation*}
then $T_{25}'(T_{10})\cong T_{26}'$.

Therefore, we have
\begin{align*}
& T_1'(T_1)\cong T_9'; \quad T_2'(T_2)\cong T_3'\cong T_4'\cong T_{10}'\cong T_{11}'\cong T_{12}'; \quad T_5'(T_3)\cong T_6'\cong T_{13}'\cong T_{15}'; \\
& T_7'(T_4)\cong T_8'\cong T_{14}'\cong T_{16}'; \quad T_{17}'(T_5)\cong T_{18}'\cong T_{20}'\cong T_{21}'; \quad  T_{25}'(T_{10}) \cong T_{26}'.
\end{align*}
\subsection{Non-isomorphic cases of $T_i'$}
Now we separate all non-isomorphic cases of $T_i', 1\leq i \leq 26$.

\textbf{(1)}
First, we prove that $T_1'(T_1)$ does not isomorphic to $T_2'(T_2)$, $T_5'(T_3)$, $T_7'(T_4)$, $T_{17}'(T_5)$ and $T_{25}'(T_{10})$.  In fact, if $\phi_{1}:T_{1}' \rightarrow T_{2}'$ is the isomorphism defined by \eqref{eq:phi-lamdae}, then by \eqref{eq:y-iso} we have
\begin{subequations} \label{eqn-3}
  \begin{numcases}{}
  \lambda_{22} =\lambda_{22}\lambda_{22} -3\lambda_{23}\lambda_{23}, \label{eqn-3-1} \\
  \lambda_{23} =-2\lambda_{22}\lambda_{23}, \label{eqn-3-2} \\
  -\lambda_{32} =\lambda_{22}\lambda_{32} -3\lambda_{23}\lambda_{33}, \label{eqn-3-3} \\
  \lambda_{33} =\lambda_{22}\lambda_{33} + \lambda_{23}\lambda_{32}, \label{eqn-3-4} \\
  3\lambda_{22} =3\lambda_{33}\lambda_{33} -\lambda_{32}\lambda_{32}, \label{eqn-3-5} \\
  3\lambda_{23} = 2\lambda_{32}\lambda_{33}, \label{eqn-3-6} \\
  \alpha\lambda_{21} = \theta(\lambda_{22}\lambda_{22} +3\lambda_{23}\lambda_{23}), \label{eqn-3-7} \\
  -\alpha\lambda_{31} = \theta(\lambda_{22}\lambda_{32} +3\lambda_{23}\lambda_{33}), \label{eqn-3-8} \\
  -3\alpha\lambda_{21} = \theta(\lambda_{32}\lambda_{32} +3\lambda_{33}\lambda_{33}). \label{eqn-3-9}
  \end{numcases}
\end{subequations}

If $\lambda_{23}=0$, then by \eqref{eqn-3-1} we have $\lambda_{22}=1$ or $\lambda_{22}=0$. Since $\lambda_{22}\lambda_{33} -\lambda_{23}\lambda_{32}=1$, we get $\lambda_{22}=1$. Substituting $\lambda_{22}$ and $\lambda_{23}$ into \eqref{eqn-3-3} we have $\lambda_{32}=0$. Substituting $\lambda_{22}$, $\lambda_{23}$ and $\lambda_{32}$ into \eqref{eqn-3-5} we have $\lambda_{33}=1$ or $\lambda_{33}=-1$. However, if $\lambda_{33}=-1$, then it contradicts with $\lambda_{22}\lambda_{33} -\lambda_{23}\lambda_{32}=1$. Hence, $\lambda_{33}=1$. Substituting $\lambda_{22}$, $\lambda_{23}$, $\lambda_{32}$ and $\lambda_{33}$ into \eqref{eqn-3-7} and \eqref{eqn-3-9} we have $\lambda_{21}=\frac{\theta}{\alpha}$ or $\lambda_{21}=-\frac{\theta}{\alpha}$, where $\alpha, \theta\neq 0$. This leads to a contradictory result.

If $\lambda_{23}\neq0$, then by \eqref{eqn-3-2} and \eqref{eqn-3-1} we have $\lambda_{22}=-\frac{1}{2}$, $\lambda_{23}\lambda_{23} =\frac{1}{4}$. By \eqref{eqn-3-7} and \eqref{eqn-3-9} we  obtain $\alpha\lambda_{21} =\theta$, $\lambda_{32}\lambda_{32} +3\lambda_{33}\lambda_{33} = -3$. Substituting $\lambda_{22}=-\frac{1}{2}$ into \eqref{eqn-3-5}, we have $3\lambda_{33}\lambda_{33} -\lambda_{32}\lambda_{32}=-\frac{3}{2}$. Then $ 6\lambda_{33}\lambda_{33}= -\frac{9}{2}$, $2\lambda_{32} \lambda_{32}=-\frac{3}{2}$, which does not satisfy $\lambda_{22}\lambda_{33} -\lambda_{23}\lambda_{32} =1$.  Hence $T_1'\ncong T_2'$.
%that is, $T_1'\ncong T_4',~T_1'\ncong T_{10}',~T_1'\ncong T_{11}',~T_1'\ncong T_{12}'$.

%By \eqref{eqn-4-2}, we have $\lambda_{22}=-\frac{1}{2}$, then take it into \eqref{eqn-4-1} and \eqref{eqn-4-4}, we have $\lambda_{23}\lambda_{23}=\frac{9}{16}$, $\lambda_{33} = -\lambda_{33} -1$, that is, $\lambda_{33}=-\frac{1}{2}$, then take $\lambda_{22},~\lambda_{33}$ into \eqref{eqn-4-5}, we have $\frac{3}{2} = \lambda_{32} \lambda_{32} -\frac{1}{3}$, that is, $\lambda_{32}\lambda_{32}= \frac{11}{6}$, then the values of $\lambda_{22}$, $\lambda_{23}$, $\lambda_{32}$ and $\lambda_{33}$ are contradict to $\lambda_{22}\lambda_{33} -\lambda_{23}\lambda_{32} =1$.

%Now we prove that $T_1'$ does not isomorphic to $T_5'$.
Let $\phi_{2}:T_1' \rightarrow T_5'$ be the isomorphism given by \eqref{eq:phi-lamdae}. Then by \eqref{eq:y-iso} we have
\begin{subequations}\label{eqn-1}
  \begin{numcases}{}
  \lambda_{22} = \lambda_{22} \lambda_{22} -3\lambda_{23} \lambda_{23}, \label{eqn-1-1} \\
  \lambda_{23} = \lambda_{22} \lambda_{22} -2\lambda_{22} \lambda_{23}, \label{eqn-1-2} \\
  -\lambda_{32} = \lambda_{22} \lambda_{32} -3\lambda_{23} \lambda_{33}, \label{eqn-1-3} \\
  -\lambda_{33} = \lambda_{22} \lambda_{32} -(2\lambda_{22} \lambda_{33} -1), \label{eqn-1-4} \\
  -3\lambda_{22} = \lambda_{32} \lambda_{32} -3\lambda_{33}\lambda_{33}, \label{eqn-1-5} \\
  -3\lambda_{23} = \lambda_{32} \lambda_{32} -2\lambda_{32}\lambda_{33}, \label{eqn-1-6} \\
  \lambda_{21} = \lambda_{31} = 0. \notag
  \end{numcases}
\end{subequations}

If $\lambda_{22}= -\frac{1}{2}$, then by \eqref{eqn-1-2} we have $\lambda_{23} = \frac{1}{4} +\lambda_{23}$, which is contradictory.

If $\lambda_{22}= \frac{1}{2}$, then by \eqref{eqn-1-1} and \eqref{eqn-1-4} we have $\lambda_{23}\lambda_{23} = -\frac{1}{12}$ and $\lambda_{32}= -2$. Substituting $\lambda_{22}$ and $\lambda_{32}$ into \eqref{eqn-1-5} we have  $\lambda_{33}\lambda_{33}=\frac{11}{6}$. Substituting $\lambda_{22}$, $\lambda_{23}$, $\lambda_{32}$ and $\lambda_{33}$ into $\lambda_{22}\lambda_{33} -\lambda_{23}\lambda_{32}=1$ we have $\frac{1}{2}\sqrt{\frac{11}{6}} +2\sqrt{\frac{1}{12}}\sqrt{-1}=1$, which is contradictory.

If $\lambda_{22}= 0$, then by \eqref{eqn-1-1} we have $\lambda_{23}=0$, which is contradictory to $\lambda_{22}\lambda_{33}-\lambda_{23}\lambda_{32}=1$.

If $\lambda_{22}\neq -\frac{1}{2},0,\frac{1}{2}$, then by \eqref{eqn-1-2} we have $\lambda_{23} = \frac{\lambda_{22} \lambda_{22}}{1+ 2\lambda_{22}}$. Substituting $\lambda_{23}$ into \eqref{eqn-1-1}, \eqref{eqn-1-4} and \eqref{eqn-1-3} we have  $\lambda_{22}^3-3\lambda_{22}-1=0$, $\lambda_{33}= \frac{\lambda_{22}\lambda_{32}+1} {2\lambda_{22}-1}$ and
$\lambda_{32}= \frac{3\lambda_{22}^2} {4\lambda_{22}^2+2\lambda_{22}}$. Substituting $\lambda_{32}$ and $\lambda_{33}$ into \eqref{eqn-1-5}, we have $40\lambda_{22}^2+16\lambda_{22}+1=0$, which contradicts the equation $\lambda_{22}^3-3\lambda_{22}-1=0$. Hence $T_1'\ncong T_5'$.
%that is, $T_1'\ncong T_6',~T_1'\ncong T_{13}',~T_1'\ncong T_{15}'$.

%Now we prove that $T_1'$ does not isomorphic to $T_4'$.
Let $\phi_{3}:T_1' \rightarrow T_7'$ be the isomorphism defined by \eqref{eq:phi-lamdae}. Then by \eqref{eq:y-iso} we get the same equations as in \eqref{eqn-1-1}$\sim$\eqref{eqn-1-6}. Then by $\phi_{2}$ we have $T_1'\ncong T_7'$. %that is, $T_1'\ncong T_8',~T_1'\ncong T_{14}',~T_1'\ncong T_{16}'$.

%Now we prove that $T_1'$ does not isomorphic to $T_{17}'$.
Let $\phi_{4}:T_{1}' \rightarrow T_{17}'$ be the isomorphism defined by \eqref{eq:phi-lamdae}. Then by \eqref{eq:y-iso} we have
\begin{subequations}\label{eqn-5}
  \begin{numcases}{}
  \lambda_{22} = \lambda_{22}\lambda_{22} -6\lambda_{22}\lambda_{23} +6\lambda_{23} \lambda_{23}, \label{eqn-5-1} \\
  \lambda_{23} = 3\lambda_{23} \lambda_{23} -2\lambda_{22} \lambda_{23}, \label{eqn-5-2} \\
  \lambda_{32} = 6\lambda_{22} \lambda_{33} -3 -\lambda_{22}\lambda_{32} -6\lambda_{23}\lambda_{33}, \label{eqn-5-3} \\
  \lambda_{33} = 2\lambda_{22} \lambda_{33} -1 - 3\lambda_{23}\lambda_{33}, \label{eqn-5-4} \\
  -3\lambda_{22} = \lambda_{32} \lambda_{32} -6\lambda_{32}\lambda_{33} +6\lambda_{33}\lambda_{33}, \label{eqn-5-5} \\
  -3\lambda_{23} = 3\lambda_{33}\lambda_{33} -2\lambda_{32}\lambda_{33}, \label{eqn-5-6} \\
  \lambda_{21} = 2\lambda_{22} \lambda_{23} -2\lambda_{23}\lambda_{23}, \notag \\
  \lambda_{31} = -2\lambda_{22}\lambda_{33} +2\lambda_{23}\lambda_{33} +1, \notag \\
  3\lambda_{21} = 2\lambda_{33}\lambda_{33} -2\lambda_{32}\lambda_{33}. \notag
  \end{numcases}
\end{subequations}

If $\lambda_{23}=0$, then by \eqref{eqn-5-1} we have $\lambda_{22}=1$ or $\lambda_{22} = 0$. However, if $\lambda_{22} = 0$, then it contradicts the equation $\lambda_{22}\lambda_{33} -\lambda_{23}\lambda_{32}=1$. Hence, $\lambda_{22}=1$. Substituting $\lambda_{22}$ and $\lambda_{23}$ into \eqref{eqn-5-4} we have $\lambda_{33}=1$. Substituting $\lambda_{22}$, $\lambda_{23}$ and $\lambda_{33}$ into the \eqref{eqn-5-3} we have $\lambda_{32}=\frac{3}{2}$. Substituting $\lambda_{22}$, $\lambda_{23}$, $\lambda_{32}$ and $\lambda_{33}$ into the \eqref{eqn-5-5} we get a contradiction $-3=\frac{9}{4}-9+6$.

If $\lambda_{23}\neq0$, then by \eqref{eqn-5-2} we have $\lambda_{23} = \frac{1+2\lambda_{22} }{3}$. Substituting $\lambda_{23}$ into the \eqref{eqn-5-1} we get $\lambda_{22}^2+\lambda_{22}-2=0$, which implies that $\lambda_{22}=-2$ or $\lambda_{22}=1$. If $\lambda_{22}=-2$, then by \eqref{eqn-5-2}, \eqref{eqn-5-4} and \eqref{eqn-5-3}  we have $\lambda_{23}=-1$, $\lambda_{33}=-\frac{1}{2}$ and $\lambda_{32}=0$. However, substituting $\lambda_{22}$, $\lambda_{23}$, $\lambda_{32}$ and $\lambda_{33}$ into \eqref{eqn-5-5}, we obtain $6 = \frac{3}{2}$, which is contradictory. If $\lambda_{22}=1$, then by \eqref{eqn-5-1}, \eqref{eqn-5-4} and \eqref{eqn-5-3}  we have $\lambda_{23}=1$, $\lambda_{33}=-\frac{1}{2}$ and $\lambda_{32}=-\frac{3}{2}$. Substituting $\lambda_{22}$, $\lambda_{23}$, $\lambda_{32}$ and $\lambda_{33}$ into \eqref{eqn-5-5}, we obtain $-3 = -\frac{3}{4}$, which is contradictory. Hence $T_1'\ncong T_{17}'$.
%that is, $T_1'\ncong T_{18}',~T_1'\ncong T_{19}',~T_1'\ncong T_{20}'$.

%Now we prove that $T_1'$ does not isomorphic to $T_{19}'$.

%Now we prove that $T_1'$ does not isomorphic to $T_{25}'$.
Let $\phi_{5}:T_{1}' \rightarrow T_{25}'$ be the isomorphism defined by \eqref{eq:phi-lamdae}. Then by \eqref{eq:y-iso} we have
\begin{subequations}\label{eqn-23}
  \begin{numcases}{}
  \lambda_{22} = \lambda_{23} = \lambda_{32} = \lambda_{33} = 0, \label{eqn-23-1} \\
  \lambda_{21}= \lambda_{22}\lambda_{22} -3\lambda_{23}\lambda_{23}, \notag \\
  -\lambda_{31} = \lambda_{22}\lambda_{32} -3\lambda_{23}\lambda_{33}, \notag \\
  -3\lambda_{21} = \lambda_{32}\lambda_{32} -3\lambda_{33}\lambda_{33}. \notag
  \end{numcases}
\end{subequations}
In the above equations, \eqref{eqn-23-1} contradicts $\lambda_{22}\lambda_{33}-\lambda_{23}\lambda_{32}=1$. Hence $T_1'\ncong T_{25}'$.
% that is, $T_1'\ncong T_{26}'$, $T_9'\ncong T_{25}'$, $T_9'\ncong T_{26}'$.

%In conclusion, $T_1\ncong T_{2}$, $T_1\ncong T_{3}$, $T_1\ncong T_{4}$, $T_1\ncong T_{5}$, $T_1\ncong T_{6}$, $T_1\ncong T_{10}$.

\textbf{(2)} Secondly,
similar to the above discussion, we can prove that the algebras $T_i', 1\leq i \leq 26$, are pairwise non-isomorphic (see Appendix A for more details of the proof), except for
\begin{align*}
& T_1'(T_1)\cong T_9'; \quad T_2'(T_2)\cong T_3'\cong T_4'\cong T_{10}'\cong T_{11}'\cong T_{12}'; \quad T_5'(T_3)\cong T_6'\cong T_{13}'\cong T_{15}'; \\
& T_7'(T_4)\cong T_8'\cong T_{14}'\cong T_{16}'; \quad T_{17}'(T_5)\cong T_{18}'\cong T_{20}'\cong T_{21}';\quad T_{25}'(T_{10}) \cong T_{26}';\\
 &T_{19}'(T_6); \quad T_{22}'(T_7);\quad T_{23}'(T_8);\quad T_{24}'(T_9).
\end{align*}
That is $T_i$ does not isomorphic to $T_j$ when $i \neq j$, where $1\leq i,j \leq 10$.

Therefore, the $3$-dimensional transposed Poisson $3$-Lie algebra $(A_3,\cdot,[\cdot,\cdot,\cdot])$ with $L_{e_1}=0$ is isomorphic to  one of $T_1 \sim T_{10}$. The proof of Theorem \ref{classification} is completed.

{\bf Acknowledgements:}
The authors would like to thank the referees for helpful comments.

\appendix
\section{}

\textbf{(1)}
$T_1'(T_1)$ does not isomorphic to $T_{19}'(T_6)$, $T_{22}'(T_7)$, $T_{23}'(T_8)$ and $T_{24}'(T_9)$.

Let $\phi_{6}:T_1' \rightarrow T_{19}'(T_6)$ be the isomorphism defined by \eqref{eq:phi-lamdae}. Then by \eqref{eq:y-iso} we obtain the same  equations as in \eqref{eqn-5-1}$\sim$\eqref{eqn-5-6}. Then by $\phi_{4}$ we have $T_1'\ncong T_{19}'$.

Let $\phi_{7}:T_{1}' \rightarrow T_{22}'(T_7)$  be the isomorphism defined by \eqref{eq:phi-lamdae}. Then by \eqref{eq:y-iso} we have
\begin{subequations}\label{eqn-21}
  \begin{numcases}{}
  \alpha\lambda_{22} = \alpha\lambda_{22}\lambda_{22} +(6\theta-6\alpha) \lambda_{22}\lambda_{23} +(6\alpha-9\theta)\lambda_{23} \lambda_{23}, \label{eqn-21-1} \\
  \alpha\lambda_{23} = \theta\lambda_{22}\lambda_{22} -2\alpha\lambda_{22}\lambda_{23} +(3\alpha-3\theta)\lambda_{23}\lambda_{23} , \label{eqn-21-2} \\
  -\alpha\lambda_{32} = \alpha\lambda_{22}\lambda_{32} +(3\theta-3\alpha)(2\lambda_{22} \lambda_{33}-1) +(6\alpha-9\theta)\lambda_{23}\lambda_{33}, \label{eqn-21-3} \\
  -\alpha\lambda_{33} = \theta\lambda_{22}\lambda_{32} -\alpha(2\lambda_{22} \lambda_{33}-1) +(3\alpha-3\theta)\lambda_{23}\lambda_{33}, \label{eqn-21-4} \\
  -3\alpha\lambda_{22} = \alpha\lambda_{32}\lambda_{32} +(6\theta-6\alpha) \lambda_{32}\lambda_{33} +(6\alpha-9\theta)\lambda_{33}\lambda_{33}, \label{eqn-21-5} \\
  -3\alpha\lambda_{23} = \theta\lambda_{32}\lambda_{32} -2\alpha\lambda_{32} \lambda_{33} +(3\alpha-3\theta)\lambda_{33}\lambda_{33}, \label{eqn-21-6} \\
  \alpha\lambda_{21}= \theta\lambda_{22}\lambda_{22} +(3\theta-3\alpha) \lambda_{23}\lambda_{23}, \notag \\
  -\alpha\lambda_{31} = \theta\lambda_{22}\lambda_{32} +(3\theta-3\alpha) \lambda_{23}\lambda_{33}, \notag \\
  -3\alpha\lambda_{21} = \theta\lambda_{32}\lambda_{32} +(3\theta-3\alpha) \lambda_{33}\lambda_{33}. \notag
  \end{numcases}
\end{subequations}
Without loss of generality, let $\alpha=1$ and $\theta=1$. Then \eqref{eqn-21-1}$\sim$\eqref{eqn-21-6} can be simplified to \eqref{eqn-1-1}$\sim$\eqref{eqn-1-6} in $\phi_2$. Then by $\phi_{3}$ we have $T_1'\ncong T_{22}'$.

Let $\phi_{8}:T_{1}' \rightarrow T_{23}'(T_8)$ and $\phi_{9}:T_{1}' \rightarrow T_{24}'(T_9)$ be the isomorphisms defined by \eqref{eq:phi-lamdae}. Then by \eqref{eq:y-iso} we get the same equations as in \eqref{eqn-21-1}$\sim$\eqref{eqn-21-6}. Then by $\phi_{7}$ we have $T_1'\ncong T_{23}'$, $T_1'\ncong T_{24}'$.

\textbf{(2)}
$T_2'(T_2)$ does not isomorphic to $T_5'(T_3)$, $T_7'(T_4)$, $T_{17}'(T_5)$, $T_{19}'(T_6)$, $T_{22}'(T_7)$, $T_{23}'(T_8)$ $T_{24}'(T_9)$ and $T_{25}'(T_{10})$.

%Now we prove that $T_2'$ does not isomorphic to $T_{6}'$ and $T_8'$. 
Let $\phi_{10}:T_2' \rightarrow T_6'$ and $\phi_{11}:T_2' \rightarrow T_8'$ be the isomorphisms defined by \eqref{eq:phi-lamdae}. Then by \eqref{eq:y-iso} we obtain the same equations as in \eqref{eqn-1-1}$\sim$\eqref{eqn-1-6}. Then by $\phi_{2}$ we have $T_2'\ncong T_6'$, $T_2'\ncong T_8'$.
%that is, $T_2'\ncong T_{18}',~T_2'\ncong T_{19}',~T_2'\ncong T_{20}'$.

Let $\phi_{12}:T_2' \rightarrow T_{17}'$ and $\phi_{13}:T_2' \rightarrow T_{19}'$ be the isomorphisms defined by \eqref{eq:phi-lamdae}. Then by \eqref{eq:y-iso} we obtain the same equations as in \eqref{eqn-5-1}$\sim$\eqref{eqn-5-6}. Then by $\phi_{4}$ we have $T_2'\ncong T_{17}'$ and $T_2'\ncong T_{19}'$.

Let $\phi_{14}:T_2' \rightarrow T_{22}'$, $\phi_{15}:T_2' \rightarrow T_{23}'$ and $\phi_{16}:T_2' \rightarrow T_{24}'$ be the isomorphisms defined by \eqref{eq:phi-lamdae}. Then by \eqref{eq:y-iso} we get the same equations as in \eqref{eqn-21-1}$\sim$\eqref{eqn-21-6}. Then by $\phi_{7}$ we have $T_2'\ncong T_{22}'$, $T_2'\ncong T_{23}'$, $T_2'\ncong T_{24}'$.

Let $\phi_{17}:T_2' \rightarrow T_{25}'$ be the isomorphism defined by \eqref{eq:phi-lamdae}. Then by \eqref{eq:y-iso} we get the same equation \eqref{eqn-23-1}. Then by $\phi_{5}$ we have $T_2'\ncong T_{25}'$.

\textbf{(3)}
$T_5'(T_3)$ does not isomorphic to $T_7'(T_4)$, $T_{17}'(T_5)$, $T_{19}'(T_6)$, $T_{22}'(T_7)$, $T_{23}'(T_8)$ $T_{24}'(T_9)$ and $T_{25}'(T_{10})$.

%Now we prove that $T_5'$ does not isomorphic to $T_{7}'$.
Let $\phi_{18}:T_5' \rightarrow T_7'$  be the isomorphism defined by  \eqref{eq:phi-lamdae}. Then by \eqref{eq:y-iso} we have
\begin{subequations}\label{eqn-{12}}
  \begin{numcases}{}
  \lambda_{22} + \lambda_{32} =\lambda_{22}\lambda_{22} -3\lambda_{23}\lambda_{23}, \label{eqn-12-1} \\
  \lambda_{23} + \lambda_{33} =\lambda_{22}\lambda_{22} -2\lambda_{22}\lambda_{23}, \label{eqn-12-2} \\
  -\lambda_{32} = \lambda_{22} \lambda_{32} -3\lambda_{23} \lambda_{33}, \label{eqn-12-3} \\
  -\lambda_{33} = \lambda_{22} \lambda_{32} -2\lambda_{22} \lambda_{33} + 1, \label{eqn-12-4} \\
  -3\lambda_{22} = \lambda_{32} \lambda_{32} -3\lambda_{33}\lambda_{33}, \label{eqn-12-5} \\
  -3\lambda_{23} = \lambda_{32} \lambda_{32} -2\lambda_{32}\lambda_{33}, \label{eqn-12-6} \\
  \lambda_{21} + \lambda_{31} = \lambda_{22}\lambda_{22} -\lambda_{22}\lambda_{23}, \label{eqn-12-7} \\
  -\lambda_{31} = \lambda_{22}\lambda_{32} -\frac{1}{2}(2\lambda_{22}\lambda_{33}-1), \label{eqn-12-8} \\
  -3\lambda_{21} = \lambda_{32} \lambda_{32} -\lambda_{32}\lambda_{33}. \label{eqn-12-9}
  \end{numcases}
\end{subequations}

If $\lambda_{22}=0$, then by \eqref{eqn-12-4} we have $\lambda_{33}=-1$. Substituting $\lambda_{22}$ and $\lambda_{33}$ into \eqref{eqn-12-2} we have $\lambda_{23}=1$. Substituting $\lambda_{22}$, $\lambda_{33}$ and $\lambda_{32}$ into \eqref{eqn-12-3} we have $\lambda_{32}=-3$. Substituting $\lambda_{22}$, $\lambda_{32}$ and $\lambda_{33}$ into $\lambda_{22}\lambda_{33} -\lambda_{23}\lambda_{32}=1$ we have $-3=1$, which is contradictory.

If $\lambda_{22}=\frac{1}{2}$, then by \eqref{eqn-12-4} we have $\lambda_{32}=-2$. Substituting $\lambda_{22}$ and $\lambda_{32}$ into \eqref{eqn-12-2} and \eqref{eqn-12-3} we have $2\lambda_{23}+\lambda_{33}=\frac{1}{4}$ and $8\lambda_{23}^2-\lambda_{23}-4=0$. Substituting $\lambda_{22}$, $\lambda_{32}$ and $\lambda_{33}$ into $\lambda_{22}\lambda_{33} -\lambda_{23}\lambda_{32}=1$ we have $\frac{1}{8}+\lambda_{23}=1$, which contradicts $8\lambda_{23}^2-\lambda_{23}-4=0$.

If $\lambda_{22}=-\frac{1}{2}$, then by \eqref{eqn-12-2} we have $\lambda_{33}=\frac{1}{4}$. Substituting $\lambda_{22}$ and $\lambda_{33}$ into the \eqref{eqn-12-4} we have $\lambda_{32}=3$. Substituting $\lambda_{22}$, $\lambda_{33}$ and $\lambda_{32}$ into \eqref{eqn-12-3} we have $\lambda_{23}=2$. Substituting $\lambda_{22}$, $\lambda_{32}$ and $\lambda_{33}$ into $\lambda_{22}\lambda_{33} -\lambda_{23}\lambda_{32}=1$ we have $-\frac{1}{8}-6=1$, which is contradictory.

If $\lambda_{22}\neq0,\frac{1}{2},-\frac{1}{2}$, then by \eqref{eqn-12-4} we have $\lambda_{33} = \frac{\lambda_{22}\lambda_{32}+1}{2\lambda_{22}-1}$. Substituting $\lambda_{33}$ into the \eqref{eqn-12-2} and \eqref{eqn-12-6} we have
\begin{equation}\label{eq:57}
6\lambda_{22}^3-3\lambda_{22}^2-7\lambda_{22}\lambda_{32}-2\lambda_{22}\lambda_{32}^2 -\lambda_{32}^2-2\lambda_{32}-3=0.
\end{equation}
 Substituting $\lambda_{33}$ into the \eqref{eqn-12-3} and \eqref{eqn-12-6} we have $(\lambda_{32}+2\lambda_{22}+1)(\lambda_{32}- \frac{2\lambda_{22}^2-\lambda_{22}-1}{\lambda_{22}})=0$. If $\lambda_{32}=-2\lambda_{22}-1$, then $\lambda_{33} = -\lambda_{22}-1$, $\lambda_{23} = \frac{\lambda_{22}^2+\lambda_{22}+1}{2\lambda_{22}+1}$. Taking $\lambda_{32}$ into \eqref{eq:57} we have $(\lambda_{22}-1)(2\lambda_{22}-1)(\lambda_{22}+2)=0$. Substituting $\lambda_{22}$, $\lambda_{23}$, $\lambda_{32}$ and $\lambda_{33}$ into \eqref{eqn-12-7}$\sim$\eqref{eqn-12-9} we have $\frac{\lambda_{22}^3-\lambda_{22}}{2\lambda_{22}+1} =\frac{2\lambda_{22}^2-2\lambda_{22}-3}{6}$, that is, $2\lambda_{22}^3+2\lambda_{22}^2+2\lambda_{22}+3=0$, which is contradictory to $(\lambda_{22}-1)(2\lambda_{22}-1)(\lambda_{22}+2)=0$. If $\lambda_{32}\neq-2\lambda_{22}-1$, then $\lambda_{32}=\frac{2\lambda_{22}^2-\lambda_{22}-1}{\lambda_{22}}$, $\lambda_{33} =\frac{2\lambda_{22}^2-\lambda_{22}}{2\lambda_{22}-1}$ and $\lambda_{23} = \frac{\lambda_{22}^2-\lambda_{22}}{2\lambda_{22}+1}$. Taking $\lambda_{32}$ into \eqref{eq:57} we have $(\lambda_{22}-1)(2\lambda_{22}-1)(\lambda_{22}^3+8\lambda_{22}^2+5\lambda_{22}+1)=0$. Substituting $\lambda_{22}$, $\lambda_{23}$, $\lambda_{32}$ and $\lambda_{33}$ into \eqref{eqn-12-7}$\sim$\eqref{eqn-12-9} we have $\frac{\lambda_{22}^3+2\lambda_{22}^2}{2\lambda_{22}+1} =\frac{-2\lambda_{22}^4+12\lambda_{22}^3-7\lambda_{22}^2-4\lambda_{22}-2}{6\lambda_{22}^2}$, that is, $10\lambda_{22}^5-10\lambda_{22}^4+2\lambda_{22}^3+15\lambda_{22}^2+8\lambda_{22}+2=0$, which contradicts $(\lambda_{22}-1)(2\lambda_{22}-1)(\lambda_{22}^3+8\lambda_{22}^2+5\lambda_{22}+1)=0$. Hence $T_5'\ncong T_7'$.

%Now we prove that $T_5'$ does not isomorphic to $T_{17}'$. 
Let $\phi_{19}:T_{5}' \rightarrow T_{17}'$ be the isomorphism defined by \eqref{eq:phi-lamdae}. Then by \eqref{eq:y-iso} we have
\begin{subequations}\label{eqn-15}
  \begin{numcases}{}
  \lambda_{22} +\lambda_{32} = \lambda_{22}\lambda_{22} -6\lambda_{22}\lambda_{23} +6\lambda_{23} \lambda_{23}, \label{eqn-15-1} \\
  \lambda_{23} +\lambda_{33} = 3\lambda_{23} \lambda_{23} -2\lambda_{22} \lambda_{23}, \label{eqn-15-2} \\
  \lambda_{32} = 6\lambda_{22} \lambda_{33} -3 -\lambda_{22}\lambda_{32} -6\lambda_{23}\lambda_{33}, \label{eqn-15-3} \\
  \lambda_{33} = 2\lambda_{22} \lambda_{33} -1 - 3\lambda_{23}\lambda_{33}, \label{eqn-15-4} \\
  -3\lambda_{22} = \lambda_{32} \lambda_{32} -6\lambda_{32}\lambda_{33} +6\lambda_{33}\lambda_{33}, \label{eqn-15-5} \\
  -3\lambda_{23} = 3\lambda_{33}\lambda_{33} -2\lambda_{32}\lambda_{33}, \label{eqn-15-6} \\
  \lambda_{21} +\lambda_{31} = 2\lambda_{22} \lambda_{23} -2\lambda_{23}\lambda_{23}, \label{eqn-15-7} \\
  -\lambda_{31} = 2\lambda_{22}\lambda_{33} -1 -2\lambda_{23}\lambda_{33}, \label{eqn-15-8} \\
  3\lambda_{21} = 2\lambda_{33}\lambda_{33} -2\lambda_{32}\lambda_{33}. \label{eqn-15-9}
  \end{numcases}
\end{subequations}

If $\lambda_{22}=0$, then by \eqref{eqn-15-2}, \eqref{eqn-15-4} and \eqref{eqn-15-3} we have $9\lambda_{23}^3-\lambda_{23}+1=0$ and $\lambda_{32}=6\lambda_{23}^2$. Substituting $\lambda_{22}$, $\lambda_{23}$ and $\lambda_{32}$ into $\lambda_{22}\lambda_{33} -\lambda_{23}\lambda_{32}=1$ we have $\frac{2-2\lambda_{23}}{3}=1$, which contradicts $9\lambda_{23}^3-\lambda_{23}+1=0$.

If $\lambda_{22}=1$, then by \eqref{eqn-15-1}, \eqref{eqn-15-2} and \eqref{eqn-15-4} we have $\lambda_{32}=2\lambda_{33}$ and $\lambda_{23}=\frac{\lambda_{33}-1}{3\lambda_{33}}$. Substituting $\lambda_{22}$ and $\lambda_{32}$ into the \eqref{eqn-12-5} we have $\lambda_{33}^2=\frac{3}{2}$. If $\lambda_{33}=\pm\frac{\sqrt{6}}{2}$, then $\lambda_{32}=\pm\sqrt{6}$, $\lambda_{23}=\frac{\sqrt{6}\pm2}{3\sqrt{6}}$ which contradicts the equation  $\lambda_{22}\lambda_{33} -\lambda_{23}\lambda_{32}=1$.

If $\lambda_{22}=-1$, then by \eqref{eqn-15-3} and \eqref{eqn-15-4} we have $\lambda_{33}+\lambda_{23}\lambda_{33}=0$. If $\lambda_{23}=-1$, then
substituting $\lambda_{22}$ and $\lambda_{23}$ into \eqref{eqn-12-2} and \eqref{eqn-12-1} we have $\lambda_{33}=2$, $\lambda_{32}=2$. However, this contradicts the equation  $\lambda_{22}\lambda_{33} -\lambda_{23}\lambda_{32}=1$. If $\lambda_{23}\neq-1$, then $\lambda_{33}=0$.
Substituting $\lambda_{22}$ and $\lambda_{23}$ into \eqref{eqn-12-2} and \eqref{eqn-12-1} we have $\lambda_{23}=0$, $\lambda_{32}=2$, which are contradictory to $\lambda_{22}\lambda_{33} -\lambda_{23}\lambda_{32}=1$.

If $\lambda_{22}=\frac{3\lambda_{23}+1}{2}$, then by \eqref{eqn-15-1} $\sim$ \eqref{eqn-15-4} we have $-4\lambda_{22}^2 \lambda_{23} +12\lambda_{22}\lambda_{23}^2 -9\lambda_{23}^3 +\lambda_{23}=1$, $\lambda_{22}^3-2\lambda_{22}^2-\lambda_{22}-1=0$, that is, $\lambda_{22}^3=2\lambda_{22}^2+\lambda_{22}+1$. Then taking it into $-4\lambda_{22}^2 \lambda_{23} +12\lambda_{22}\lambda_{23}^2 -9\lambda_{23}^3 +\lambda_{23}=1$ we have $\lambda_{22}=-\frac{5}{12}$, which is contradictory to $\lambda_{22}^3-2\lambda_{22}^2-\lambda_{22}-1=0$.

If $\lambda_{22}\neq0,1,-1,\frac{3\lambda_{23}+1}{2}$, then by \eqref{eqn-15-4} we have $\lambda_{33}=\frac{1}{2\lambda_{22} -3\lambda_{23}-1}$. Substituting $\lambda_{33}$ into \eqref{eqn-15-2}, we have
\begin{equation}  \label{eqn-15-10-1}
 -4\lambda_{22}^2 \lambda_{23} +12\lambda_{22}\lambda_{23}^2 -9\lambda_{23}^3 +\lambda_{23}=1.
\end{equation}
By \eqref{eqn-15-1}$\sim$\eqref{eqn-15-3} and \eqref{eqn-15-10-1} we have $\lambda_{23}=\frac{\lambda_{22}^3-\lambda_{22}+1}{2\lambda_{22}^2-2}$. Substituting $\lambda_{22}$ and $\lambda_{23}$ into \eqref{eqn-15-1} and \eqref{eqn-15-2} we have $\lambda_{32}=\frac{-\lambda_{22}^6-2\lambda_{22}^5+2\lambda_{22}^4 +4\lambda_{22}^3 -\lambda_{22}^2-2\lambda_{22}+3}{2(\lambda_{22}^2-1)}$, $\lambda_{33}=\frac{-\lambda_{22}^6-2\lambda_{22}^5+2\lambda_{22}^4 +6\lambda_{22}^3 -3\lambda_{22}^2-4\lambda_{22}+5}{(2\lambda_{22}^2-2)^2}$. Substituting $\lambda_{22}$, $\lambda_{23}$, $\lambda_{32}$ and $\lambda_{33}$ into \eqref{eqn-15-3} and \eqref{eqn-15-4} we have
\begin{equation}\label{eqn-15-10}
  \lambda_{22}^9-7\lambda_{22}^7 +3\lambda_{22}^6+6\lambda_{22}^5 -30\lambda_{22}^4+14\lambda_{22}^3+39\lambda_{22}^2 -14\lambda_{22}-3=0,
\end{equation}
\begin{equation}\label{eqn-15-11}
  \lambda_{22}^9-7\lambda_{22}^7 +3\lambda_{22}^6 +15\lambda_{22}^5-18\lambda_{22}^4-10\lambda_{22}^3+27\lambda_{22}^2+\lambda_{22}-3=0.
\end{equation}
Then subtracting \eqref{eqn-15-10} from \eqref{eqn-15-11} we have $9\lambda_{22}^5+12\lambda_{22}^4-24\lambda_{22}^3-12\lambda_{22}^2+15\lambda_{22}=0$, that is, $3\lambda_{22}^4 +4\lambda_{22}^3 -8\lambda_{22}^2-3\lambda_{22}+5=0$. Substituting $\lambda_{22}$, $\lambda_{23}$, $\lambda_{32}$ and $\lambda_{33}$ into $\lambda_{22}\lambda_{33} -\lambda_{23}\lambda_{32}=1$ we have
\begin{equation}\label{eqn-15-12}
  2\lambda_{22}^9+2\lambda_{22}^8-4\lambda_{22}^7 -9\lambda_{22}^6 +10\lambda_{22}^5 +2\lambda_{22}^4-16\lambda_{22}^3+7\lambda_{22}^2 +10\lambda_{22}-7=0.
\end{equation}
Substituting $\lambda_{22}$, $\lambda_{23}$, $\lambda_{32}$ and $\lambda_{33}$ into \eqref{eqn-15-7}$\sim$\eqref{eqn-15-9} we have
\begin{equation*}
  \frac{2\lambda_{22}^6-4\lambda_{22}^4 +2\lambda_{22}^2-2}{(2\lambda_{22}^2-2)^2}=\frac{-24\lambda_{22}^8+42\lambda_{22}^7 +70\lambda_{22}^6-192\lambda_{22}^5 +58\lambda_{22}^4+234\lambda_{22}^3 -224\lambda_{22}^2-30\lambda_{22}+16}{36\lambda_{22}^8-144\lambda_{22}^6 +216\lambda_{22}^4-144\lambda_{22}^2+36},
\end{equation*}
that is,
\begin{equation}\label{eqn-15-13}
  78\lambda_{22}^8-96\lambda_{22}^7-232\lambda_{22}^6+516\lambda_{22}^5 -4\lambda_{22}^4-720\lambda_{22}^3+332\lambda_{22}^2 +84\lambda_{22}-70=0.
\end{equation}
Then taking $\lambda_{22}^4 =-\frac{4\lambda_{22}^3 -8\lambda_{22}^2 -3\lambda_{22}+5}{3}$ into \eqref{eqn-15-10}, \eqref{eqn-15-12} and \eqref{eqn-15-13} we have
\begin{subequations}\label{eqn-90}
  \begin{numcases}{}
  23866\lambda_{22}^3-14573\lambda_{22}^2-17136\lambda_{22}+19763=0,\label{eqn-90-1} \\
  60848\lambda_{22}^3-35965\lambda_{22}^2-37080\lambda_{22}+37147=0,\label{eqn-90-2} \\
  64996\lambda_{22}^3-44726\lambda_{22}^2-39276\lambda_{22}+103025=0,\label{eqn-90-3}
  \end{numcases}
\end{subequations}
Multiplying \eqref{eqn-90-1} by $60848$, multiplying \eqref{eqn-90-2} by $23866$, then subtracting the results we have $1577623\lambda_{22}^2+8763336\lambda_{22}-17554929=0$.
Multiplying \eqref{eqn-90-1} by $64996$, multiplying \eqref{eqn-90-3} by $23866$, then subtracting the results we have $20040668\lambda_{22}^2-29401740\lambda_{22}-195713117=0$, which is contradictory to $1577623\lambda_{22}^2+8763336\lambda_{22}-17554929=0$. Hence $T_5'\ncong T_{17}'$.

Let $\phi_{20}:T_5' \rightarrow T_{19}'$ be the isomorphism defined by \eqref{eq:phi-lamdae}. Then by \eqref{eq:y-iso} we get the same equations as in \eqref{eqn-15-1}$\sim$\eqref{eqn-15-6}. Then by $\phi_{19}$ we have $T_5'\ncong T_{19}'$.

Let $\phi_{21}:T_{5}' \rightarrow T_{22}'$ be the isomorphism defined by \eqref{eq:phi-lamdae}. Then by \eqref{eq:y-iso} we have
\begin{subequations}\label{eqn-22}
  \begin{numcases}{}
  \alpha(\lambda_{22}+\lambda_{32}) = \alpha\lambda_{22}\lambda_{22} +(6\theta-6\alpha) \lambda_{22}\lambda_{23} +(6\alpha-9\theta)\lambda_{23} \lambda_{23}, \label{eqn-22-1} \\
  \alpha(\lambda_{23}+\lambda_{33}) = \theta\lambda_{22}\lambda_{22} -2\alpha\lambda_{22}\lambda_{23} +(3\alpha-3\theta)\lambda_{23}\lambda_{23} , \label{eqn-22-2} \\
  -\alpha\lambda_{32} = \alpha\lambda_{22}\lambda_{32} +(3\theta-3\alpha)(2\lambda_{22} \lambda_{33}-1) +(6\alpha-9\theta)\lambda_{23}\lambda_{33}, \label{eqn-22-3} \\
  -\alpha\lambda_{33} = \theta\lambda_{22}\lambda_{32} -\alpha(2\lambda_{22} \lambda_{33}-1) +(3\alpha-3\theta)\lambda_{23}\lambda_{33}, \label{eqn-22-4} \\
  -3\alpha\lambda_{22} = \alpha\lambda_{32}\lambda_{32} +(6\theta-6\alpha) \lambda_{32}\lambda_{33} +(6\alpha-9\theta)\lambda_{33}\lambda_{33}, \label{eqn-22-5} \\
  -3\alpha\lambda_{23} = \theta\lambda_{32}\lambda_{32} -2\alpha\lambda_{32} \lambda_{33} +(3\alpha-3\theta)\lambda_{33}\lambda_{33}, \label{eqn-22-6} \\
  \alpha(\lambda_{21} +\lambda_{31}) = \theta\lambda_{22}\lambda_{22} +(3\theta-3\alpha) \lambda_{23}\lambda_{23}, \notag \\
  -\alpha\lambda_{31} = \theta\lambda_{22}\lambda_{32} +(3\theta-3\alpha) \lambda_{23}\lambda_{33}, \notag \\
  -3\alpha\lambda_{21} = \theta\lambda_{32}\lambda_{32} +(3\theta-3\alpha) \lambda_{33}\lambda_{33}. \notag
  \end{numcases}
\end{subequations}
Without loss of generality, let $\alpha=1$ and $\theta=1$. Then \eqref{eqn-22-1}$\sim$\eqref{eqn-22-6} can be simplified to \eqref{eqn-12-1}$\sim$\eqref{eqn-12-6}. Then by $\phi_{18}$ we have $T_5'\ncong T_{22}'$.

Let $\phi_{22}:T_{5}' \rightarrow T_{23}'$ and $\phi_{23}:T_{5}' \rightarrow T_{24}'$ be the isomorphisms defined by \eqref{eq:phi-lamdae}. Then by \eqref{eq:y-iso} we get the same equations as in \eqref{eqn-22-1}$\sim$\eqref{eqn-22-6}. Then by $\phi_{21}$ we have $T_5'\ncong T_{23}'$, $T_5'\ncong T_{24}'$.

Let $\phi_{24}:T_{5}' \rightarrow T_{25}'$ be the isomorphism defined by \eqref{eq:phi-lamdae}. Then by \eqref{eq:y-iso} we have
\begin{subequations}\label{eqn-24}
  \begin{numcases}{}
  \lambda_{22} = \lambda_{23}= \lambda_{32}= \lambda_{33}=0, \label{eqn-24-1} \\
  \lambda_{21} +\lambda_{31} = \lambda_{22}\lambda_{22} -3\lambda_{23}\lambda_{23}, \notag \\
  -\lambda_{31} = \lambda_{22}\lambda_{32} -3\lambda_{23}\lambda_{33}, \notag \\
  -3\lambda_{21} = \lambda_{32}\lambda_{32} -3\lambda_{33}\lambda_{33}. \notag
  \end{numcases}
\end{subequations}
\eqref{eqn-24-1} contradicts with $\lambda_{22}\lambda_{33}-\lambda_{23}\lambda_{32}=1$. Hence $T_5'\ncong T_{25}'$.
%Equations \eqref{eqn-24-1}-\eqref{eqn-24-6} are obviously contradict to $\lambda_{22}\lambda_{33}-\lambda_{23}\lambda_{32}=1$. Hence $T_5'\ncong T_{25}'$.
% that is, $T_5'\ncong T_{26}'$, $T_6'\ncong T_{25}'$, $T_6'\ncong T_{26}'$, $T_{13}'\ncong T_{25}'$, $T_{13}'\ncong T_{26}'$, $T_{15}'\ncong T_{25}'$, $T_{15}'\ncong T_{26}'$.

\textbf{(4)}
$T_7'(T_4)$ does not isomorphic to $T_{17}'(T_5)$, $T_{19}'(T_6)$, $T_{22}'(T_7)$, $T_{23}'(T_8)$ $T_{24}'(T_9)$ and $T_{25}'(T_{10})$.

Let $\phi_{25}:T_7' \rightarrow T_{17}'$ and $\phi_{26}:T_7' \rightarrow T_{19}'$ be the isomorphisms defined by \eqref{eq:phi-lamdae}. Then by \eqref{eq:y-iso} we get the same equations as in \eqref{eqn-15-1}$\sim$\eqref{eqn-15-6}. Then by $\phi_{19}$ we have $T_7'\ncong T_{17}'$.

Let $\phi_{27}:T_{7}' \rightarrow T_{22}'$, $\phi_{28}:T_{7}' \rightarrow T_{23}'$ and $\phi_{29}:T_{7}' \rightarrow T_{24}'$ be the isomorphisms defined by \eqref{eq:phi-lamdae}. Then by \eqref{eq:y-iso} we get the same equations as in \eqref{eqn-22-1}$\sim$\eqref{eqn-22-6}. Then by $\phi_{21}$ we have $T_7'\ncong T_{22}'$, $T_7'\ncong T_{23}'$, $T_7'\ncong T_{24}'$.

Let $\phi_{30}:T_7' \rightarrow T_{25}'$ be the isomorphism defined by \eqref{eq:phi-lamdae}. Then by \eqref{eq:y-iso} we get the same equation \eqref{eqn-24-1}. Then by $\phi_{24}$ we have $T_7'\ncong T_{25}'$.

\textbf{(5)}
$T_{17}'(T_5)$ does not isomorphic to $T_{19}'(T_6)$, $T_{25}'(T_{10})$, $T_{22}'(T_7)$, $T_{23}'(T_8)$ and $T_{24}'(T_9)$.

%Now we prove that $T_{17}'$ does not isomorphic to $T_{19}'$. 
Let $\phi_{31}:T_{17}' \rightarrow T_{19}'$ be the isomorphism defined by \eqref{eq:phi-lamdae}. Then by \eqref{eq:y-iso} we have
\begin{subequations}\label{eqn-35}
  \begin{numcases}{}
  \lambda_{22} = \lambda_{22}\lambda_{22} -6\lambda_{22}\lambda_{23} +6\lambda_{23}\lambda_{23}, \label{eqn-35-1} \\
  \lambda_{23} = -2\lambda_{22}\lambda_{23} +3\lambda_{23}\lambda_{23}, \label{eqn-35-2} \\
  -3\lambda_{22} -\lambda_{32} = \lambda_{22}\lambda_{32} -3(2\lambda_{22}\lambda_{33}-1) +6\lambda_{23}\lambda_{33}, \label{eqn-35-3} \\
  -3\lambda_{23} -\lambda_{33} = -2\lambda_{22}\lambda_{33} +1 +3\lambda_{23}\lambda_{33}, \label{eqn-35-4} \\
  6\lambda_{22} +3\lambda_{32} = \lambda_{32}\lambda_{32} -6\lambda_{32}\lambda_{33} +6\lambda_{33}\lambda_{33}, \label{eqn-35-5} \\
  6\lambda_{23} +3\lambda_{33} = -2\lambda_{32}\lambda_{33} +3\lambda_{33}\lambda_{33}, \label{eqn-35-6} \\
  \lambda_{21} = \lambda_{22}\lambda_{22} -2\lambda_{22}\lambda_{23}, \label{eqn-35-7} \\
  \lambda_{11} -3\lambda_{21} -\lambda_{31} = \lambda_{22}\lambda_{32} -2\lambda_{22}\lambda_{33}+1, \label{eqn-35-8} \\
  -2\lambda_{11} +6\lambda_{21} +3\lambda_{31} = \lambda_{32}\lambda_{32} -2\lambda_{32}\lambda_{33}.  \label{eqn-35-9}
  \end{numcases}
\end{subequations}

If $\lambda_{23}=0$, then by \eqref{eqn-35-1}, \eqref{eqn-35-4} and \eqref{eqn-35-3} we have $\lambda_{22}=1$, $\lambda_{33}=1$ and $\lambda_{32}=0$. Substituting $\lambda_{22}$, $\lambda_{23}$, $\lambda_{33}$ and $\lambda_{32}$ into \eqref{eqn-35-7}$\sim$\eqref{eqn-35-9} we have $\lambda_{11}=0$, $\lambda_{21}=1$ and $\lambda_{31}=-2$, However, this contradicts with $\lambda_{11}\neq 0$.

If $\lambda_{23}\neq0$, then by \eqref{eqn-35-1} and \eqref{eqn-35-2} we have either $\lambda_{22}=1$ and $\lambda_{23}=1$ or $\lambda_{22}=-2$ and $\lambda_{23}=-1$. If $\lambda_{22}=1$, $\lambda_{23}=1$, then by \eqref{eqn-35-3}$\sim$\eqref{eqn-35-4} we have $\lambda_{32}=-3$, $\lambda_{33}=-2$. Substituting $\lambda_{22}$, $\lambda_{33}$ and $\lambda_{32}$ into \eqref{eqn-35-7}$\sim$\eqref{eqn-35-9} we have $\lambda_{11}=0$, $\lambda_{21}=-1$ and $\lambda_{31}=1$, which contradicts with $\lambda_{11}\neq0$. If $\lambda_{22}=-2$, $\lambda_{23}=-1$, then by \eqref{eqn-35-3}$\sim$\eqref{eqn-35-4} we have $\lambda_{32}=3$, $\lambda_{33}=1$. Substituting $\lambda_{22}$, $\lambda_{33}$ and $\lambda_{32}$ into \eqref{eqn-35-7}$\sim$\eqref{eqn-35-9} we have $\lambda_{11}=0$, $\lambda_{21}=0$ and $\lambda_{31}=1$. However, this contradicts with $\lambda_{11}\neq0$. Hence $T_{17}'\ncong T_{19}'$.

Let $\phi_{32}:T_{17}' \rightarrow T_{25}'$ be the isomorphism defined by \eqref{eq:phi-lamdae}. Then by \eqref{eq:y-iso} we have
\begin{subequations}\label{eqn-25}
  \begin{numcases}{}
  \lambda_{22} = \lambda_{23} =0, \label{eqn-25-1} \\
  -3\lambda_{22} -\lambda_{32} = 0, \label{eqn-25-3} \\
  -3\lambda_{23} -\lambda_{33} = 0, \label{eqn-25-4} \\
  6\lambda_{22} +3\lambda_{32} = 0, \label{eqn-25-5} \\
  6\lambda_{23} +3\lambda_{33} = 0, \label{eqn-25-6} \\
  \lambda_{21} = \lambda_{22}\lambda_{22} -3\lambda_{23}\lambda_{23}, \notag \\
  \lambda_{11} -3\lambda_{21} -\lambda_{31} = \lambda_{22}\lambda_{32} -3\lambda_{23}\lambda_{33}, \notag \\
  -2\lambda_{11} +6\lambda_{21} +3\lambda_{31} = \lambda_{32}\lambda_{32} -3\lambda_{33}\lambda_{33}. \notag
  \end{numcases}
\end{subequations}
By \eqref{eqn-25-3}$\sim$\eqref{eqn-25-4}, we have $\lambda_{32}=\lambda_{33}=0$, which contradicts $\lambda_{22}\lambda_{33}-\lambda_{23}\lambda_{32}=1$. Hence $T_{17}'\ncong T_{25}'$.

Let $\phi_{33}:T_{21}' \rightarrow T_{22}'$ be the isomorphism defined by \eqref{eq:phi-lamdae}. Then by \eqref{eq:y-iso} we have
\begin{subequations}\label{eqn-36}
  \begin{numcases}{}
  \alpha\lambda_{22} +\theta\lambda_{32} = \alpha\lambda_{22}\lambda_{22} +2(3\theta-3\alpha) \lambda_{22}\lambda_{23} +(6\alpha-9\theta)\lambda_{23}\lambda_{23}, \label{eqn-36-1} \\
  \alpha\lambda_{23} +\theta\lambda_{33} = \theta\lambda_{22}\lambda_{22} -2\alpha\lambda_{22} \lambda_{23} +(3\alpha-3\theta)\lambda_{23}\lambda_{23}, \label{eqn-36-2} \\
  (3\theta-3\alpha)\lambda_{22} -\alpha\lambda_{32} = \alpha\lambda_{22}\lambda_{32} +(3\theta-3\alpha)(2\lambda_{22}\lambda_{33}-1) \label{eqn-36-3} \\
  \quad\quad\quad\quad\quad\quad\quad\quad\quad\quad +(6\alpha-9\theta)\lambda_{23} \lambda_{33}, \notag \\
  (3\theta-3\alpha)\lambda_{23} -\alpha\lambda_{33} = \theta\lambda_{22}\lambda_{32} -\alpha(2\lambda_{22}\lambda_{33}-1) +(3\alpha-3\theta)\lambda_{23}\lambda_{33}, \label{eqn-36-4} \\
  (6\alpha-9\theta)\lambda_{22} +(3\alpha-3\theta)\lambda_{32} = \alpha\lambda_{32}\lambda_{32} +2(3\theta-3\alpha)\lambda_{32}\lambda_{33} \label{eqn-36-5}\\
  \quad\quad\quad\quad\quad\quad\quad\quad\quad\quad +(6\alpha-9\theta)\lambda_{33} \lambda_{33}, \notag \\
  (6\alpha-9\theta)\lambda_{23} +(3\alpha-3\theta)\lambda_{33} = \theta\lambda_{32}\lambda_{32} -2\alpha\lambda_{32}\lambda_{33} +(3\alpha-3\theta)\lambda_{33}\lambda_{33}, \label{eqn-36-6} \\
  \alpha\lambda_{21} +\theta\lambda_{31} = \theta\lambda_{22}\lambda_{22} +(3\theta-3\alpha) \lambda_{23}\lambda_{23}, \notag \\
  (\alpha-2\theta)\lambda_{11} +(3\theta-3\alpha)\lambda_{21} -\alpha\lambda_{31} = \theta\lambda_{22}\lambda_{32} +(3\theta-3\alpha) \lambda_{23}\lambda_{33}, \notag \\
  (6\theta-2\alpha)\lambda_{11} +(6\alpha-9\theta)\lambda_{21} +(3\alpha-3\theta)\lambda_{31} = \theta\lambda_{32}\lambda_{32} +(3\theta-3\alpha) \lambda_{33}\lambda_{33}. \notag
  \end{numcases}
\end{subequations}
Without loss of generality, let $\alpha=1$ and $\theta=1$. Then \eqref{eqn-36-1}$\sim$\eqref{eqn-36-6} can be simplified to \eqref{eqn-12-1}$\sim$\eqref{eqn-12-6}. Then by $\phi_{18}$ we have $T_{21}'\ncong T_{22}'$.

Let $\phi_{34}:T_{21}' \rightarrow T_{23}'$ and $\phi_{35}:T_{21}' \rightarrow T_{24}'$ be the isomorphisms defined by \eqref{eq:phi-lamdae}. Then by \eqref{eq:y-iso} we get the same equations as in \eqref{eqn-36-1}$\sim$\eqref{eqn-36-6}. Then by $\phi_{33}$ we have $T_{21}'\ncong T_{23}'$, $T_{21}'\ncong T_{24}'$.

\textbf{(6)}
$T_{19}'(T_6)$ does not isomorphic to $T_{22}'(T_7)$, $T_{23}'(T_8)$ $T_{24}'(T_9)$ and $T_{25}'(T_{10})$.

Let $\phi_{36}:T_{19}' \rightarrow T_{22}'$ be the isomorphism defined by \eqref{eq:phi-lamdae}. 
%\begin{subequations}\label{eqn-56}
%  \begin{numcases}{}
%  \alpha\lambda_{22} = \alpha\lambda_{22}\lambda_{22} +2(3\theta-3\alpha) \lambda_{22}\lambda_{23} +(6\alpha-9\theta)\lambda_{23}\lambda_{23}, \label{eqn-56-1} \\
%  \alpha\lambda_{23} = \theta\lambda_{22}\lambda_{22} -2\alpha\lambda_{22} \lambda_{23} +(3\alpha-3\theta)\lambda_{23}\lambda_{23}, \label{eqn-56-2} \\
%  -3\alpha\lambda_{22} -\alpha\lambda_{32} = \alpha\lambda_{22}\lambda_{32} +(3\theta-3\alpha)(2\lambda_{22}\lambda_{33}-1) +(6\alpha-9\theta)\lambda_{23} \lambda_{33}, \label{eqn-56-3}  \\
%  -3\alpha\lambda_{23} -\alpha\lambda_{33} = \theta\lambda_{22}\lambda_{32} -\alpha(2\lambda_{22}\lambda_{33}-1) +(3\alpha-3\theta)\lambda_{23}\lambda_{33}, \label{eqn-56-4} \\
%  6\alpha\lambda_{22} +3\alpha\lambda_{32} = \alpha\lambda_{32}\lambda_{32} +2(3\theta-3\alpha)\lambda_{32}\lambda_{33} +(6\alpha-9\theta)\lambda_{33} \lambda_{33}, \label{eqn-56-5}\\
%  6\alpha\lambda_{23} +3\alpha\lambda_{33} = \theta\lambda_{32}\lambda_{32} -2\alpha\lambda_{32}\lambda_{33} +(3\alpha-3\theta)\lambda_{33}\lambda_{33}, \label{eqn-56-6} \\
%  \alpha\lambda_{11} +\alpha\lambda_{21} = \theta\lambda_{22}\lambda_{22} +(3\theta-3\alpha) \lambda_{23}\lambda_{23}, \label{eqn-56-7} \\
%  -\alpha\lambda_{11} -3\alpha\lambda_{21} -\alpha\lambda_{31} = \theta\lambda_{22}\lambda_{32} +(3\theta-3\alpha) \lambda_{23}\lambda_{33}, \label{eqn-56-8} \\
%  6\alpha\lambda_{21} +3\alpha\lambda_{31} = \theta\lambda_{32}\lambda_{32} +(3\theta-3\alpha) \lambda_{33}\lambda_{33}. \label{eqn-56-9}
%  \end{numcases}
%\end{subequations}
Without loss of generality, let $\alpha=1$ and $\theta=1$. Then by \eqref{eq:y-iso} we have
\begin{subequations}\label{eqn-57}
  \begin{numcases}{}
  \lambda_{22} =\lambda_{22}\lambda_{22} -3\lambda_{23}\lambda_{23}, \label{eqn-57-1}\\
  \lambda_{23} =\lambda_{22}\lambda_{22} -2\lambda_{22}\lambda_{23}, \label{eqn-57-2}\\
  -3\lambda_{22} -\lambda_{32} =\lambda_{22}\lambda_{32} -3\lambda_{23}\lambda_{33}, \label{eqn-57-3} \\
  -3\lambda_{23} -\lambda_{33} =\lambda_{22}\lambda_{32} -(2\lambda_{22}\lambda_{33}-1), \label{eqn-57-4} \\
  6\lambda_{22} +3\lambda_{32} = \lambda_{32}\lambda_{32}-3\lambda_{33} \lambda_{33}, \label{eqn-57-5}\\
  6\lambda_{23} +3\lambda_{33} = \lambda_{32}\lambda_{32} -2\lambda_{32}\lambda_{33}, \label{eqn-57-6} \\
  \lambda_{11} +\lambda_{21} = \lambda_{22}\lambda_{22}, \label{eqn-57-7} \\
  -\lambda_{11}-3\lambda_{21} -\lambda_{31} = \lambda_{22}\lambda_{32}, \label{eqn-57-8} \\
  6\lambda_{21} +3\lambda_{31} = \lambda_{32}\lambda_{32}. \label{eqn-57-9}
  \end{numcases}
\end{subequations}

If $\lambda_{22}=0$, then by \eqref{eqn-57-1} we have $\lambda_{23}=0$, which contradicts $\lambda_{22}\lambda_{33}-\lambda_{23}\lambda_{32}=1$.

If $\lambda_{22}=-\frac{1}{2}$, then by \eqref{eqn-57-1}, we have either $\lambda_{23}=\frac{1}{2}$ or $\lambda_{23}=-\frac{1}{2}$. However, this contradicts \eqref{eqn-57-2}.

If $\lambda_{22}=-1$, then by \eqref{eqn-57-1} and \eqref{eqn-57-2} we have $\lambda_{23}^2=\frac{2}{3}$, $\lambda_{23}=-1$, which are contradictory.

If $\lambda_{22}\neq0,-\frac{1}{2},-1$, then by \eqref{eqn-57-2} and \eqref{eqn-57-1} we have $\lambda_{23}=\frac{\lambda_{22}^2}{2\lambda_{22}+1}$ and $\lambda_{22}^3 -3\lambda_{22}-1=0$. Substituting $\lambda_{22}$ and $\lambda_{23}$ into \eqref{eqn-57-3} and \eqref{eqn-57-4} we have $\lambda_{33}=\frac{\lambda_{22}-\lambda_{22}^2}{2\lambda_{22}+1}$ and $\lambda_{32}=-\frac{21\lambda_{22}^2+33\lambda_{22}+9} {(\lambda_{22}+1)(2\lambda_{22}+1)}$. Adding \eqref{eqn-57-7} and \eqref{eqn-57-8} we have $-2\lambda_{21}-\lambda_{31}=\lambda_{22}^2+\lambda_{22}\lambda_{32}$, then multiplying the result by $3$ and adding \eqref{eqn-57-9} we have $3\lambda_{22}^2+3\lambda_{22}\lambda_{32}+\lambda_{32}^2=0$. By substituting $\lambda_{32}=-\frac{21\lambda_{22}^2+33\lambda_{22}+9} {(\lambda_{22}+1)(2\lambda_{22}+1)}$ into it, we obtain $\frac{561\lambda_{22}^2+1050\lambda_{22}+297}{561\lambda_{22}^2 +1050\lambda_{22}+297}=0$, which contradicts $\lambda_{22}^3 -3\lambda_{22}-1=0$. Hence $T_{19}'\ncong T_{22}'$.

Let $\phi_{37}:T_{19}' \rightarrow T_{23}'$ be the isomorphism defined by \eqref{eq:phi-lamdae}. 
Without loss of generality, let $\alpha=1$ and $\theta=1$. Then by \eqref{eq:y-iso} we have
\begin{subequations}\label{eqn-59}
  \begin{numcases}{}
  \lambda_{22} =\lambda_{22}\lambda_{22} -3\lambda_{23}\lambda_{23}, \label{eqn-59-1}\\
  \lambda_{23} =\lambda_{22}\lambda_{22} -2\lambda_{22}\lambda_{23}, \label{eqn-59-2}\\
  -3\lambda_{22} -\lambda_{32} =\lambda_{22}\lambda_{32} -3\lambda_{23}\lambda_{33}, \label{eqn-59-3} \\
  -3\lambda_{23} -\lambda_{33} =\lambda_{22}\lambda_{32} -(2\lambda_{22}\lambda_{33}-1), \label{eqn-59-4} \\
  6\lambda_{22} +3\lambda_{32} = \lambda_{32}\lambda_{32}-3\lambda_{33} \lambda_{33}, \label{eqn-59-5}\\
  6\lambda_{23} +3\lambda_{33} = \lambda_{32}\lambda_{32} -2\lambda_{32}\lambda_{33}, \label{eqn-59-6} \\
  \lambda_{11} +\lambda_{21} = \lambda_{22}\lambda_{23}, \label{eqn-59-7} \\
  -\alpha\lambda_{11} -3\alpha\lambda_{21} -\alpha\lambda_{31} = \frac{1}{2} (2\lambda_{22}\lambda_{33}-1), \label{eqn-59-8} \\
  6\alpha\lambda_{21} +3\alpha\lambda_{31} = \lambda_{32}\lambda_{33}. \label{eqn-59-9}
  \end{numcases}
\end{subequations}
Then \eqref{eqn-59-1}$\sim$\eqref{eqn-59-6} are the same as the \eqref{eqn-57-1}$\sim$\eqref{eqn-57-6}. By $\phi_{36}$ we have $\lambda_{23}=\frac{\lambda_{22}^2}{2\lambda_{22}+1}(\lambda_{22}\neq\frac{1}{2})$ and $\lambda_{22}^3 -3\lambda_{22}-1=0$. Substituting $\lambda_{22}$ and $\lambda_{23}$ into \eqref{eqn-57-3} and \eqref{eqn-57-4} we have $\lambda_{33}=\frac{\lambda_{22}-\lambda_{22}^2}{2\lambda_{22}+1}$, $\lambda_{32}=-\frac{21\lambda_{22}^2+33\lambda_{22}+9} {(\lambda_{22}+1)(2\lambda_{22}+1)}(\lambda_{22}\neq-1)$. Adding \eqref{eqn-57-7} and \eqref{eqn-57-8} we have $-2\lambda_{21}-\lambda_{31}=\lambda_{22}\lambda_{23}+\lambda_{22}\lambda_{33} -\frac{1}{2}$, then multiplying the result by $3$ and adding \eqref{eqn-57-9} we have $3\lambda_{22}\lambda_{23}+3\lambda_{22}\lambda_{33} -\frac{3}{2} +\lambda_{32}\lambda_{33}=0$. Taking $\lambda_{23}=\frac{\lambda_{22}^2}{2\lambda_{22}+1}$, $\lambda_{32}=-\frac{21\lambda_{22}^2+33\lambda_{22}+9} {(\lambda_{22}+1)(2\lambda_{22}+1)}$ and $\lambda_{33}=\frac{\lambda_{22}-\lambda_{22}^2}{2\lambda_{22}+1}$ into it we have $\frac{126\lambda_{22}^2+225\lambda_{22}+63}{84\lambda_{22}^2 +150\lambda_{22}+42}=0$  contradicts  with $\lambda_{22}^3 -3\lambda_{22}-1=0$. Hence $T_{19}'\ncong T_{23}'$.

Let $\phi_{38}:T_{19}' \rightarrow T_{24}'$ be the isomorphism defined by \eqref{eq:phi-lamdae}. 
Without loss of generality, let $\alpha=1$, $\theta=1$ and $\rho=1$. Then by \eqref{eq:y-iso} we have
\begin{subequations}\label{eqn-61}
  \begin{numcases}{}
  \lambda_{22} =\lambda_{22}\lambda_{22} -3\lambda_{23}\lambda_{23}, \label{eqn-61-1}\\
  \lambda_{23} =\lambda_{22}\lambda_{22} -2\lambda_{22}\lambda_{23}, \label{eqn-61-2}\\
  -3\lambda_{22} -\lambda_{32} =\lambda_{22}\lambda_{32} -3\lambda_{23}\lambda_{33}, \label{eqn-61-3} \\
  -3\lambda_{23} -\lambda_{33} =\lambda_{22}\lambda_{32} -(2\lambda_{22}\lambda_{33}-1), \label{eqn-61-4} \\
  6\lambda_{22} +3\lambda_{32} = \lambda_{32}\lambda_{32}-3\lambda_{33} \lambda_{33}, \label{eqn-61-5}\\
  6\lambda_{23} +3\lambda_{33} = \lambda_{32}\lambda_{32} -2\lambda_{32}\lambda_{33}, \label{eqn-61-6} \\
  \lambda_{11} +\lambda_{21} = \lambda_{22}\lambda_{22} -3\lambda_{22}\lambda_{23} +3\lambda_{23}\lambda_{23}, \label{eqn-61-7} \\
  -\lambda_{11} -3\lambda_{21} -\lambda_{31} = \lambda_{22}\lambda_{32} -\frac{3}{2}(2\lambda_{22}\lambda_{33}-1) +3\lambda_{23}\lambda_{33}, \label{eqn-61-8} \\
  6\lambda_{21} +3\lambda_{31} = \lambda_{32}\lambda_{32} -3\lambda_{32}\lambda_{33} +3\lambda_{33}\lambda_{33}. \label{eqn-61-9}
  \end{numcases}
\end{subequations}
Then  \eqref{eqn-61-1}$\sim$\eqref{eqn-61-6} are the same as \eqref{eqn-57-1}$\sim$\eqref{eqn-57-6}. Then by $\phi_{36}$ we have $\lambda_{23}=\frac{\lambda_{22}^2}{2\lambda_{22}+1}(\lambda_{22}\neq\frac{1}{2})$, $\lambda_{22}^3 -3\lambda_{22}-1=0$, $\lambda_{33}=\frac{\lambda_{22}-\lambda_{22}^2}{2\lambda_{22}+1}$ and $\lambda_{32}=-\frac{21\lambda_{22}^2+33\lambda_{22}+9} {(\lambda_{22}+1)(2\lambda_{22}+1)}(\lambda_{22}\neq-1)$. Multiplying \eqref{eqn-61-7} and \eqref{eqn-61-8} by $3$ and then adding them to \eqref{eqn-61-9} we have $3(\lambda_{22}^2-3\lambda_{22}\lambda_{23}+3\lambda_{23}^2+\lambda_{22}\lambda_{32} -3\lambda_{22}\lambda_{33}+\frac{3}{2}+3\lambda_{23}\lambda_{33}) +\lambda_{32}^2-3\lambda_{32}\lambda_{33}+3\lambda_{33}^2=0$. Substituting $\lambda_{23}$, $\lambda_{32}$ and $\lambda_{33}$ into it we have $3(1-\frac{3}{2})+3=0$, which is contradictory. Hence $T_{19}'\ncong T_{24}'$.

Let $\phi_{39}:T_{19}' \rightarrow T_{25}'$ be the isomorphism defined by \eqref{eq:phi-lamdae}. Then by \eqref{eq:y-iso} we get the same equations as in \eqref{eqn-25-1}$\sim$\eqref{eqn-25-6}. Then by $\phi_{32}$ we have $T_{19}'\ncong T_{25}'$.

\textbf{(7)}
$T_{22}'(T_7)$ does not isomorphic to $T_{23}'(T_8)$, $T_{24}'(T_9)$ and $T_{25}'(T_{10})$.

Let $\phi_{40}:T_{22}' \rightarrow T_{23}'$ and $\phi_{41}:T_{22}' \rightarrow T_{24}'$ be the isomorphisms defined by \eqref{eq:phi-lamdae}. Then by \eqref{eq:y-iso} we get the same equations as in \eqref{eqn-36-1}$\sim$\eqref{eqn-36-6}. Then by $\phi_{33}$ we have $T_{22}'\ncong T_{23}'$, $T_{22}'\ncong T_{24}'$.

Let $\phi_{42}:T_{22}' \rightarrow T_{25}'$ be the isomorphism defined by \eqref{eq:phi-lamdae}. Then by \eqref{eq:y-iso} we have
\begin{subequations}\label{eqn-26}
  \begin{numcases}{}
  \alpha\lambda_{22} +\theta\lambda_{32} = 0, \label{eqn-26-1} \\
  \alpha\lambda_{23} +\theta\lambda_{33} = 0, \label{eqn-26-2} \\
  (3\theta-3\alpha)\lambda_{22} -\alpha\lambda_{32} = 0, \label{eqn-26-3} \\
  (3\theta-3\alpha)\lambda_{23} -\alpha\lambda_{33} = 0, \label{eqn-26-4} \\
  (6\alpha-9\theta)\lambda_{22} +(3\alpha-3\theta)\lambda_{32} = 0, \label{eqn-26-5} \\
  (6\alpha-9\theta)\lambda_{23} +(3\alpha-3\theta)\lambda_{33} = 0, \label{eqn-26-6} \\
  \theta\lambda_{11} +\alpha\lambda_{21} +\theta\lambda_{31} = \alpha(\lambda_{22}\lambda_{22} -3\lambda_{23}\lambda_{23}), \notag \\
  (3\theta-3\alpha)\lambda_{21} -\alpha\lambda_{31} = \alpha(\lambda_{22}\lambda_{32} -3\lambda_{23}\lambda_{33}), \notag \\
  (3\theta-3\alpha)\lambda_{11} +(6\alpha-9\theta)\lambda_{21} +(3\alpha-3\theta)\lambda_{31} = \alpha(\lambda_{32}\lambda_{32} -3\lambda_{33} \lambda_{33}). \notag
  \end{numcases}
\end{subequations}
By \eqref{eqn-26-1} and \eqref{eqn-26-3}, we have $(3\theta-3\alpha+\frac{\alpha^2}{\theta}) \lambda_{22}=0$. By \eqref{eqn-26-2} and \eqref{eqn-26-4}, we have $(3\theta-3\alpha+\frac{\alpha^2}{\theta}) \lambda_{23}=0$. Without loss of generality, let $3\theta-3\alpha+\frac{\alpha^2}{\theta}\neq 0$. Then we have $\lambda_{22}=0$ and $\lambda_{23}=0$. Substituting $\lambda_{22}$ and $\lambda_{23}$ into \eqref{eqn-26-1} and \eqref{eqn-26-2} we have $\lambda_{32}=0$ and $\lambda_{33}=0$, which contradicts the equation $\lambda_{22}\lambda_{33}-\lambda_{23}\lambda_{32}=1$. Hence $T_{22}'\ncong T_{25}'$.

\textbf{(8)}
$T_{23}'(T_8)$ does not isomorphic to $T_{24}'(T_9)$ and $T_{25}'(T_{10})$. $T_{24}'(T_9)$ does not isomorphic to $T_{25}'(T_{10})$.

Let $\phi_{43}:T_{23}' \rightarrow T_{24}'$ be the isomorphisms defined by \eqref{eq:phi-lamdae}. Then by \eqref{eq:y-iso} we get the same equations as in \eqref{eqn-36-1}$\sim$\eqref{eqn-36-6}. Then by $\phi_{33}$ we have $T_{23}'\ncong T_{24}'$.

Let $\phi_{44}:T_{23}' \rightarrow T_{25}'$ and $\phi_{45}:T_{24}' \rightarrow T_{25}'$ be the isomorphism defined by \eqref{eq:phi-lamdae}. Then by \eqref{eq:y-iso} we get the same equations as in \eqref{eqn-26-1}$\sim$\eqref{eqn-26-6}. Then by $\phi_{42}$ we have $T_{23}'\ncong T_{25}'$, $T_{24}'\ncong T_{25}'$.

In conclusion, $T_1\ncong T_{7}$, $T_1\ncong T_{8}$, $T_1\ncong T_{9}$, $T_2\ncong T_{7}$, $T_2\ncong T_{8}$, $T_2\ncong T_{9}$, $T_3\ncong T_{7}$, $T_3\ncong T_{8}$, $T_3\ncong T_{9}$, $T_4\ncong T_{7}$, $T_4\ncong T_{8}$, $T_4\ncong T_{9}$, $T_6\ncong T_{7}$, $T_6\ncong T_{8}$, $T_6\ncong T_{9}$, $T_5\ncong T_{7}$, $T_5\ncong T_{8}$, $T_5\ncong T_{9}$, $T_7\ncong T_{8}$, $T_7\ncong T_{9}$, $T_8\ncong T_{9}$, $T_7\ncong T_{10}$, $T_8\ncong T_{10}$, $T_9\ncong T_{10}$.

\end{document}